\documentclass[a4paper,12pt]{article}

\usepackage[margin=1.2in]{geometry}

\usepackage{amsmath}
\usepackage{amssymb}
\usepackage{algorithmic}
\usepackage{algorithm}
\usepackage{amsthm}
\usepackage{amsfonts}
\usepackage{comment}
\usepackage{graphicx}
\usepackage[hidelinks]{hyperref}
\usepackage{color}
\usepackage{bbm}
\usepackage{makecell}

\newtheorem{theorem} {Theorem}
\newtheorem{lemma} {Lemma}

\newtheorem{corollary} {Corollary}

\newtheorem{remark} {Remark}

\def\s{{\mathbf{s}}}
\def\d{{\mathbf{d}}}
\def\m{{\mathbf{m}}}
\def\x{{\mathbf{x}}}
\def\c{{\mathbf{c}}}
\def\g{{\mathbf{g}}}
\def\e{{\mathbf{e}}}
\def\a{{\mathbf{a}}}
\def\u{{\mathbf{u}}}
\def\v{{\mathbf{v}}}
\def\z{{\mathbf{z}}}
\def\w{{\mathbf{w}}}
\def\y{{\mathbf{y}}}
\def\q{{\mathbf{q}}}

\def\b{{\mathbf{b}}}

\def\A{{\mathbf{A}}}
\def\M{{\mathbf{M}}}

\def\matE{{\mathbf{E}}}

\newcommand{\mX}{\mathcal{X}}
\newcommand{\mI}{\mathcal{I}}

\newcommand{\mV}{\mathcal{V}}
\newcommand{\mP}{\mathcal{P}}

\newcommand{\mF}{\mathcal{F}}

\newcommand{\mS}{\mathcal{S}}

\newcommand{\dist}{\textrm{dist}}

\newcommand{\conv}{\textrm{conv}}

\newcommand{\rank}{\textrm{rank}}
\newcommand{\reals}{\mathbb{R}}

\DeclareMathOperator*{\argmin}{argmin}
\DeclareMathOperator*{\argmax}{argmax}

\usepackage{pifont}
\usepackage{subcaption}
\usepackage{float}

\title{Revisiting Decomposition-Invariant Conditional Gradient Methods for Polytopes}
\date{}
\author{Dan Garber \\ \small{dangar@technion.ac.il} \\ {Faculty of Data and Decision Sciences} \vspace{0pt}\\  {Technion - Israel Institute of Technology}}

\begin{document}
\maketitle

\begin{abstract}
We revisit Decomposition-Invariant Conditional Gradient methods, originally introduced by Garber and Meshi in 2016, for minimizing a convex and $\beta$-smooth function over a polytope in $\reals^n$ given by linear equalities and entrywise nonnegativity constraints, under an $\alpha$-quadratic growth condition. For 2-level polytopes we design a simple and parameter-free dyadic step-size rule that yields a linear convergence rate which depends on the optimal-face dimension and has no explicit ambient-dimension factor present in standard away-step-based conditional gradient methods.

For general polytopes in the representation above,  we introduce a parameter-free method with exact line-search whose total number of linear optimization calls to reach an $\epsilon$-approximate solution is
\[
O\left(
\max\left\{1,\frac{(d^*+1)\beta D^2}{\alpha}\right\}
\left[1+\log\max\left\{1,\frac{\beta D^2}{\epsilon}\right\}\right]
\right),
\]
where $d^*$ is the dimension of the optimal face and $D$ is the Euclidean diameter of the polytope after normalizing the minimum positive constrained vertex coordinate to one. To the best of our knowledge, this is the first conditional gradient method for general polytopes in our representation that has a linear convergence rate that depends on the optimal-face dimension and has no explicit ambient-dimension factor.
The key is a new penalized linear objective which leads to descent directions that compensate for the increase in the dimension of the current face, and thus provides a new mechanism for controlling so-called \textit{bad away steps} in decomposition-invariant conditional gradient methods.

\end{abstract}

\section{Introduction}
This work considers conditional gradient methods (aka Frank-Wolfe algorithms) for the following convex optimization problem:
\begin{align}\label{eq:optProb}
\min_{\x\in\mP}f(\x),
\end{align}
where $f:\reals^n\rightarrow\reals$ is continuously differentiable and convex, and $\mP$ is a polytope in $\reals^n$ of the form:
\begin{align}\label{eq:polyStruct}
\mP := \{\x\in\reals^n~|~\A\x=\b, \x_i \geq 0 ~\forall i\in\mI\},~~~\mI\subseteq[n].
\end{align}

We let $\mV$ denote the set of vertices of $\mP$ and we denote the minimal slack of a vertex with respect to the non-negativity constraints $\xi = \min\{\v(i)~|~\v\in\mV,~i\in\mI,~\v(i) > 0\}$, i.e., for any $i\in\mI$ and any vertex $\v\in\mV$, either $\v(i) = 0$ or $\v(i) \geq \xi$. Without loss of generality, throughout this work we assume the normalization $\xi =1$. We let $\mX^*$ denote the set of optimal solutions to Problem \eqref{eq:optProb} and we let $f^*$ denote the optimal value. Throughout this work we let $\langle{\cdot,\cdot}\rangle$ and $\Vert{\cdot}\Vert$ denote the standard inner product  and the Euclidean norm, respectively. Our algorithms will be  independent of the choice of norm; the choice of the Euclidean norm will only be used for the convergence analysis. We let $D$ denote the Euclidean diameter of $\mP$.  We denote the distance of a point $\x\in\reals^n$ from a closed and convex set $\mS\subset\reals^n$ by
$\mathrm{dist}(\x,\mS) = \min_{\y\in\mS}\Vert{\x-\y}\Vert$.

Throughout, and unless stated otherwise, we assume that $f$ is $\beta$-smooth and that Problem \eqref{eq:optProb} satisfies a quadratic growth bound with some parameter $\alpha > 0$, i.e.,
\begin{align}\label{eq:qg}
\forall \x\in\mP: \qquad \dist(\x,\mX^*)^2 \leq \frac{2}{\alpha}\left({f(\x) - f^*}\right).
\end{align}
Recall this condition holds in particular if $f$ is $\alpha$-strongly convex, or more generally and due to Hoffman's bound, if $f$ is of the form $f(\x) = g(\M\x)$, where $g$ is strongly convex and $\M$ is a linear map, see for instance Theorem 10 in \cite{necoara2019linear}.

For any $\x\in\mP$ we let $\mI(\x)\subseteq\mI$ denote the support of $\x$ with respect to $\mI$ (i.e., $i\in\mI(\x)$ if and only if $\x(i) > 0$) and  we let 
$\mF(\x)$ denote the minimal face of $\mP$ containing $\x$, i.e., $\mF(\x) = \{\y\in\mP ~|~ \y(i) = 0~ \forall i\in\mI\setminus\mI(\x)\}$. We let $\mF^*$ denote the minimal face of $\mP$ containing the optimal set $\mX^*$ and we let $d^*$ denote its dimension, i.e., 
\begin{align*}
&d^* := \dim(\mF^*)
= n-\operatorname{rank}
\begin{bmatrix}
\A\\
\matE_{J^*}
\end{bmatrix}, \quad \textrm{where}\\
&J^* := \left\{i \in \mI :
\x^*(i)=0 \ \forall \x^* \in \mX^* \right\},
~~
\matE_{J^*} :=
\begin{bmatrix}
\e_i^\top
\end{bmatrix}_{i\in J^*}
\in \mathbb{R}^{|J^*|\times n}.
\end{align*}
In the above $\e_i$ denotes the $i$th standard basis vector in $\reals^n$ and we use the convention that $\matE_{\emptyset}$ is the empty matrix (in which case $d^* = n- \rank(\A)$).

We are interested in the complexity of conditional gradient methods (aka Frank-Wolfe algorithms) for solving Problem \eqref{eq:optProb} to a desired approximation error $\epsilon >0$ with respect to function value. Indeed for the setting described above, this family of algorithms (which accesses the polytope $\mP$ through a linear optimization oracle, and in particular does not require projections) has received significant interest in recent years. Unfortunately, due to the vast body of work on the subject, we cannot survey it thoroughly. Instead, we focus on the most relevant items. The interested reader may refer to the recent textbook \cite{braun2022conditional}. 

Under our Assumptions on Problem \eqref{eq:optProb}, the standard conditional gradient method (with line-search or a predetermined step-size schedule) is well known to reach an $\epsilon$-approximate solution within worst-case $O(\beta{}D^2/\epsilon)$ steps (where a step corresponds to one gradient computation of $f$ and one linear optimization step over the polytope $\mP$) \cite{jaggi2013revisiting, frank1956algorithm, levitin1966constrained}. Gu{\'e}lat and Marcotte \cite{guelat1986some} introduced the away-step Frank-Wolfe method and proved that under an additional \textit{strict complementarity condition} (not assumed in this work), after a finite number of iterations, the method converges linearly (i.e., the number of iterations scales with $\log(1/\epsilon)$ and not with $1/\epsilon$). However, their convergence analysis depends on the distance of the optimal solution from the boundary of the optimal face. Importantly, they did not provide a rigorous complexity bound for their method.
Garber and Hazan \cite{garber2013playing, garber2016linearly} gave the first variant of the conditional gradient method (also based on incorporating certain away steps) that converges linearly, without requiring strict complementarity and without dependence on the location of the optimal solutions, but with a number of iterations that scales linearly with the ambient dimension $n$. Shortly after, Lacoste-Julien and Jaggi \cite{lacoste2015global} presented a modern analysis for the method  of \cite{guelat1986some} and proved that it has a global linear convergence rate without assuming strict complementarity and without dependence on the location of optimal solutions, however, as in \cite{garber2013playing, garber2016linearly}, with at least linear dependence on the ambient dimension $n$. Notably, in both works, the dependence on the ambient dimension $n$ holds even when all optimal solutions lie on a low-dimensional face of the polytope $\mP$.

We mention in passing that in two recent works \cite{garber2020revisiting, garber2025accelerated} Garber revisited the linear convergence rates of the away-step Frank-Wolfe method under the \textit{strict complementarity condition}, i.e., as in the setting of  \cite{guelat1986some}, and presented explicit complexity bounds that scale with the sparsity measure $d^*$ and not with the ambient dimension, including for accelerated methods. Nevertheless, here we do not consider such strong complementarity conditions. 

Both lines of work on linearly converging conditional gradient methods surveyed above (without complementarity conditions) \cite{garber2016linearly, lacoste2015global}  rely on maintaining the current feasible iterate as an explicit convex combination of vertices of the polytope, which are in turn used to construct the away step directions. That is, on each iteration $t$, these methods maintain the current iterate $\x_t$ in the explicit form $\x_t = \sum_{\v\in\mS_t}\lambda(\v)\v$, where $\mS_t \subseteq\mV$ and $\lambda$ is in the $|\mS_t|$-dimensional simplex. The away step is then chosen as a point in $\argmax_{\v\in\mS_t}\langle{\v,\nabla{}f(\x_t)}\rangle$.

Garber and Meshi \cite{garber2016linear_memory} noted that these decomposition-dependent methods have three significant limitations: 
\begin{enumerate}
\item
they require storing the decomposition which can amount to significant memory overhead (typically $O(tn)$, where $t$ is the number of iterations executed so far),
\item
they require on each iteration to scan through all vertices in the decomposition to find the one which maximizes the inner product with the gradient (again, typically $O(tn)$ time, where $t$ is the number of iterations executed so far), 
\item
the decomposition-dependent analysis introduces a concrete obstacle towards removing the dependence on the ambient dimension $n$ in the convergence rate, and replacing it with a quantity that scales only with the sparsity of optimal solutions. See a detailed discussion in Section \ref{sec:obstacle}.
\end{enumerate}
To overcome these limitations, \cite{garber2016linear_memory} suggested that, instead of computing an away-step using an explicit decomposition, it may be preferable to compute one from an implicit decomposition using the linear optimization oracle of $\mP$ which, by denoting the lowest-dimension face of $\mP$ containing the current iterate $\x_t$ as $\mF(\x_t) = \{\x\in\mP~|~\x_t(i) = 0 \Rightarrow \x(i) =0~ \forall i\in\mI\}$, takes the form:
\begin{align}\label{eq:awaystep}
\v_{t,-}\gets \argmin_{\v\in\mF(\x_t)}\langle{\v,-\nabla{}f(\x_t)}\rangle.
\end{align}
We note such an implicit computation was already suggested by Gu{\'e}lat and Marcotte \cite{guelat1986some}, however without any discussion or any analytic treatment for the benefits of this approach or the limitations it resolves.

The restriction of the $\argmin$ to $\mF(\x_t)$ guarantees that there always exists $\eta >0$ such that $\x_t(i) - \eta\v_{t,-} \geq 0$ for all $i\in\mI$, and hence such an away step maintains feasibility with respect to the nonnegativity constraints. As noted in  \cite{garber2016linear_memory}, under the polytope representation in \eqref{eq:polyStruct}, the computation in \eqref{eq:awaystep} can be implemented numerically, by computing an element in $\argmax_{\v\in\mP}\langle{\v,\widetilde{\nabla{}}f(\x_t)}\rangle$, where $\widetilde{\nabla{}}f(\x_t)$ assigns $-\infty$ (or simply a sufficiently large negative value) to each entry in the gradient that corresponds to a zero entry in $\x_t$ (and leaves other entries unchanged). Alternatively, for many polytopes of interest, each face has the same structure, and hence the same implementation of a linear optimization oracle, as that of the original polytope $\mP$, e.g., the unit simplex, the hypercube, the unit flow polytope, and the bipartite matching polytope of a bipartite graph, and thus computing Eq. \eqref{eq:awaystep} has the same structure as linear optimization over $\mP$, see discussions in \cite{garber2016linear_memory}. Since this approach does not rely on a specific decomposition, Garber and Meshi named their method \textit{decomposition-invariant conditional gradient} (DICG). Nevertheless, when the polytope $\mP$ is given by an arbitrary representation (and not necessarily the one in \eqref{eq:polyStruct}), the oracle assumed in \eqref{eq:awaystep} amounts in general to a stronger \textit{black-box optimization oracle} than the standard linear optimization oracle that optimizes with respect to the entire polytope, which is the most standard assumption in the literature on conditional gradient / Frank-Wolfe algorithms. 
In our result for general polytopes, generalizing \eqref{eq:awaystep}, and by identifying a face $\mF$ of $\mP$ with a subset of coordinates $J\subseteq\mI$ that are set to zero (i.e., $\x\in\mF$ if and only if $\x(i) =0$ for all $i\in{}J)$, we shall consider the optimization oracle:
\begin{align}\label{eq:faceConstOracle}
\argmin_{\v\in\mV\cap\mF}\langle{\v,\c}\rangle \equiv \argmin_{\v\in\mV\cap\{\x|\x(i)=0~\forall i\in{}J\}}\langle{\v,\c}\rangle,
\end{align}
for a given linear objective $\c\in\reals^n$.

We shall refer to such an oracle as a \textit{face-constrained linear optimization oracle}. Note again that, as discussed above, for the representation of the polytope in \eqref{eq:polyStruct}, implementing this oracle amounts to setting each entry $i\in{}J$ in $\c$ to  $+\infty$ and calling the standard linear optimization oracle of $\mP$ w.r.t. this modified linear objective.

Concretely, \cite{garber2016linear_memory} proved that for a \textit{2-level polytope}, that is a polytope in which for any $i\in\mI$ there exists $a_i >0$, such that for any vertex $\v\in\mV$ either $\v(i) = 0$ or $\v(i)=a_i$, by designing a specialized \textit{dyadic} step-size sequence, their DICG method, which applies pairwise updates, produces feasible iterates that converge linearly while replacing the dependence on the ambient dimension $n$ with $\rm{card}(\mX^*)$ --- an upper bound on the number of nonzero entries in any optimal solution (indeed if the optimal set lies on a low-dimensional face and all vertices are entrywise sparse, we have $\rm{card}(\mX^*) << n$). However, their step-size sequence relies on unknown quantities which are difficult to estimate such as $\alpha, \beta, \rm{card}(\mX^*)$. It is a simple observation, that was indeed made in  a follow-up work by Bashiri and Zhang \cite{bashiri2017decomposition}, that the analysis in \cite{garber2016linear_memory} in fact allows one to readily replace the entrywise sparsity quantity $\rm{card}(\mX^*)$ with the more desirable quantity $d^*$ --- the dimension of the optimal face. \cite{bashiri2017decomposition} also considered decomposition-invariant methods for general polytopes (not necessarily 2-level) which use line-search to set the step-size (and hence parameter-free), however they were only able to prove a linear convergence rate that scales with $nd^*$, i.e., worse than the decomposition-based methods \cite{garber2016linearly, lacoste2015global} by a factor of $d^*$.

More recently, Wirth, Pe{\~n}a, and Pokutta~\cite{wirth2026fast} developed an
affine-invariant error-bound framework for several Frank--Wolfe variants,
including decomposition-invariant methods with a face-constrained linear optimization oracle as in Eq. \eqref{eq:faceConstOracle}. For $2$-level
polytopes in the representation \eqref{eq:polyStruct}, their pairwise method refines the dyadic scheme of \cite{garber2016linear_memory} by requiring only the extended-curvature parameter $L$ in order to set the step-size; it nevertheless
remains parameter-dependent, and generic estimates such as
$L\leq\beta D^2$ may lead to highly conservative step-sizes. For general polytopes in the representation \eqref{eq:polyStruct}
their rates retain, in the worst case, a multiplicative dependence on the dimension of the feasible polytope, and hence
do not provide the explicit $d^*$-dependent complexity pursued here.

To conclude, despite their benefits, the current decomposition-invariant methods suffer from two major shortcomings:
\begin{enumerate}
\item
When $\mP$ is a 2-level polytope, the only provable methods with a linear convergence rate that scales with the dimension of the optimal face $d^*$ (and not with the ambient dimension), require a difficult-to-tune step-size sequence.
\item
No currently known method (even with a difficult-to-tune step-size sequence) has a linear rate that scales only with the dimension of the optimal face $d^*$ for general polytopes (i.e., not necessarily 2-level).
\end{enumerate}

This work resolves both issues in the setting described above through the following contributions:
\begin{enumerate}
\item
When $\mP$ is 2-level in the representation \eqref{eq:polyStruct}, we construct a simple and parameter-free dyadic step-size rule that when combined with the DICG method from \cite{garber2016linear_memory}, yields a state-of-the-art linear convergence rate that scales with $d^*$ and not with $n$. We also prove sublinear primal and dual convergence rates for this new step-size rule that do not require the quadratic growth condition.
\item
When $\mP$ is an arbitrary polytope in the representation \eqref{eq:polyStruct}, we construct a parameter-free method based on  exact line-search calls that reaches an $\epsilon$-approximate solution using
\[
O\left(\max\left\{1,\frac{(d^*+1)\beta D^2}{\alpha}\right\}
\left[1+\log\max\left\{1,\frac{\beta D^2}{\epsilon}\right\}\right]\right)
\]
linear optimization oracle calls. To the best of our knowledge this gives the first global linear rate governed by $d^*$ and with no explicit dependence on the ambient dimension $n$ for general polytopes in the representation \eqref{eq:polyStruct}. The key technical contribution is the introduction of a new approach to control so-called \textit{bad away steps} in decomposition invariant conditional gradient methods. We introduce a new descent step mechanism, implemented using only the linear optimization oracle of the polytope, which is guaranteed to produce sufficient objective reduction that can compensate subsequent potential bad away steps caused by the increase of the dimension of the active face.
\end{enumerate}

Table \ref{table:res} summarizes the above discussions.

\begin{table*}\renewcommand{\arraystretch}{1.7}
{\small
\setlength{\tabcolsep}{5pt}
\begin{center}

\newcolumntype{C}[1]{>{\centering\arraybackslash}m{#1}}
\begin{tabular}{ | C{7.3em}  | C{4.0em}  | C{6.7em} | C{4.9em} | C{11.0em} |} 
  \hline
  reference & polytope type & decomposition-invariant  &parameter-free & \#linear-oracle calls to $\epsilon$ error \\
  \hline
  Garber and Hazan \cite{garber2016linearly}  & general & \ding{55} & \ding{55}  & $n\frac{\beta{}D^2}{\alpha}\log\frac{1}{\epsilon}$  \\ \hline
 Lacoste-Julien \& Jaggi \cite{lacoste2015global} & general &  \ding{55} & \ding{51} & $n\frac{\beta{}D^2}{\alpha}\log\frac{1}{\epsilon}$  \\ \hline
  Garber and Meshi \cite{garber2016linear_memory}, Bashiri and Zhang \cite{bashiri2017decomposition} &  2-level & \ding{51} & \ding{55} & $(d^*+1)\frac{\beta{}D^2}{\alpha}\log\frac{1}{\epsilon}$   \\ \hline
   Bashiri and Zhang \cite{bashiri2017decomposition}&  general & \ding{51} & \ding{51} & $n(d^*+1)\frac{\beta{}D^2}{\alpha}\log\frac{1}{\epsilon}$   \\ \hline
  Theorem \ref{thm:2levelDICG} &  2-level & \ding{51} & \ding{51} & $(d^*+1)\frac{\beta{}D^2}{\alpha}\log\frac{1}{\epsilon}$   \\ \hline
  Theorem \ref{thm:generalDICG} &  general & \ding{51} & \ding{51} & $
\frac{(d^*+1)\beta D^2}{\alpha}\log\frac{1}{\epsilon}$   \\ \hline
 
\end{tabular}\caption{Summary of most relevant previous works and our main contributions. We omit universal constants and display the previous bounds in simplified form. The polytope is assumed in the representation \eqref{eq:polyStruct}. In the highly specialized cases: when every optimal solution can be expressed as a convex combination of fewer than $(d^*+1)$ vertices, or when $\mP$ is a product polytope, the parameter $d^*$ could  readily be replaced (without changes to the algorithms) with more refined parameters, see Remark \ref{remark:fewVertices} and Lemma \ref{lem:productMassTransfer} in the sequel, respectively.}\label{table:res}
\end{center}}
\vskip -0.2in
\end{table*}\renewcommand{\arraystretch}{1.5}  

The rest of the paper is organized as follows. In Section \ref{sec:prelim} we present two core technical lemmas that are at the heart of our convergence analyses. There we also demonstrate a concrete obstacle for obtaining dimension-independent linear rates for decomposition-based methods using existing analyses. In Section \ref{sec:2level} we present our result for 2-level polytopes and in
Section \ref{sec:genPoly} we present our new result for general polytopes. Finally, in Section \ref{sec:numerics} we present numerical demonstrations.

\section{Technical Preliminaries}\label{sec:prelim}
We begin with two technical tools that, in similar versions, have also played a central part in previous linearly converging conditional gradient methods for polytopes. The first is the following \textit{mass transfer lemma} which scales with the sparsity $d^*$. This argument was originally developed in \cite{garber2016linearly, garber2013playing} for a decomposition-based method, and hence it originally had an explicit dependence on the ambient dimension $n$ in the RHS of \eqref{eq:lem:L1dist}. It was later refined in \cite{garber2016linear_memory} and \cite{bashiri2017decomposition} for decomposition-invariant methods and it was shown that in such a case, the ambient dimension could be replaced with that of the optimal face $d^*$.  Here we further refine it (by obtaining the new upper-bound on the middle term in \eqref{eq:lem:L1dist}) in a way that will be central for our result for general polytopes.
\begin{lemma}\label{lem:L1dist}[mass transfer lemma]
Let $\x\in\mP$ and let $\x^*\in\mX^*$. $\x$ can be written as a convex
combination $\x=\sum_{i=1}^k\lambda_i\v_i$ of vertices in $\mV$, with
$\lambda_i>0$ for all $i\in[k]$, such that
\[
\x^*=\sum_{i=1}^k(\lambda_i-\Delta_i)\v_i
+\sum_{i=1}^k\Delta_i\u_i,
\]
where $\Delta_i\in[0,\lambda_i]$ and
$\u_i\in\mV\cap\mF(\x^*)$ for all $i\in[k]$, and
\begin{align}\label{eq:lem:L1dist}
\sum_{i=1}^k\Delta_i
\leq
\sum_{i=1}^k\Delta_i
\sqrt{\max\Bigg\{1,\sum_{\ell\in\mI\setminus\mI(\x)}\u_i(\ell)\Bigg\}}
\leq\sqrt{d^*+1}\Vert{\x-\x^*}\Vert.
\end{align}
\end{lemma}

\begin{proof}
Throughout assume $\x\neq\x^*$. By Carath\'eodory's theorem, applied in
the affine hull of $\mF(\x^*)\subseteq\mF^*$, write
$\x^*=\sum_{i=1}^k\gamma_i\u_i$, where $k\leq d^*+1$,
$\gamma_i>0$, and $\u_i\in\mV\cap\mF(\x^*)$ for all $i\in[k]$.
Applying Lemma 5.3 from \cite{garber2016linearly} (as in the proof of
Lemma 2 in \cite{garber2016linear_memory}), we can write
\begin{align}\label{eq:lem1:GHdecomp}
\x=\sum_{i=1}^k(\gamma_i-\Delta_i)\u_i
+\left(\sum_{i=1}^k\Delta_i\right)\z,
\end{align}
where $\Delta_i\in[0,\gamma_i]$, $\z\in\mP$, and for every $i$ with
$\Delta_i>0$ there exists $j_i\in\mI$ such that $\z(j_i)=0$ and
$\u_i(j_i)>0$.
We first establish the bound in \eqref{eq:lem:L1dist} for these
coefficients, and then split them to obtain the representations in the
statement without changing either transferred-mass sum.

Since $\x\neq\x^*$, we have $\sum_i\Delta_i>0$. Also,  $\z\in\mF(\x)$. Set
$C:=\mI\setminus\mI(\z)$. Then, $\mI\setminus\mI(\x)\subseteq C$. Let $\y_C$ denote the restriction of a
vector $\y$ to $C$.  For every
$i$ with $\Delta_i>0$, the coordinate $j_i\in C$ satisfies
$\u_i(j_i)\geq1$, by the normalization $\xi=1$. The same normalization
gives
\[
\max\Big\{1,\sum_{\ell\in\mI\setminus\mI(\x)}\u_i(\ell)\Big\}
\leq\Vert{(\u_i)_C}\Vert^2.
\]
Moreover, \eqref{eq:lem1:GHdecomp} and $\z_C=\mathbf{0}$ imply
$(\x^*-\x)_C=\sum_i\Delta_i(\u_i)_C$. Since these summands have
nonnegative entries, their pairwise inner products are nonnegative.
Consequently, by Cauchy--Schwarz,
\begin{align*}
\Bigg(\sum_{i=1}^k\Delta_i
\sqrt{\max\Big\{1,\sum_{\ell\in\mI\setminus\mI(\x)}\u_i(\ell)\Big\}}
\Bigg)^2
&\leq\left(\sum_{i=1}^k
\Vert{\Delta_i(\u_i)_C}\Vert\right)^2\\
&\leq k\sum_{i=1}^k\Vert{\Delta_i(\u_i)_C}\Vert^2\\
&\leq k\left\Vert{\sum_{i=1}^k\Delta_i(\u_i)_C}\right\Vert^2\\
&=k\Vert{(\x^*-\x)_C}\Vert^2\\
&\leq(d^*+1)\Vert{\x-\x^*}\Vert^2.
\end{align*}
Taking square roots gives the second inequality in \eqref{eq:lem:L1dist}.
The first follows since each square-root weight is at least one.

To obtain the representations in the statement, write
 $\z=\sum_{\ell=1}^q\nu_\ell\w_\ell$, where
$\w_\ell\in\mV$, $\nu_\ell>0$, and $\sum_{\ell=1}^q\nu_\ell=1$.
For $i\in[k]$ and $\ell\in\{0,\ldots,q\}$, define
\[
(\lambda_{i,\ell},\Delta_{i,\ell},\v_{i,\ell},\u_{i,\ell})
:=
\begin{cases}
(\gamma_i-\Delta_i,\,0,\,\u_i,\,\u_i),
    & \ell=0,\\
(\Delta_i\nu_\ell,\,\Delta_i\nu_\ell,\,\w_\ell,\,\u_i),
    & \ell\in[q].
\end{cases}
\]
Then $\sum_{i,\ell}\lambda_{i,\ell}=1$,
$0\leq\Delta_{i,\ell}\leq\lambda_{i,\ell}$, and
\begin{align*}
\sum_{i,\ell}\lambda_{i,\ell}\v_{i,\ell}
=\sum_{i=1}^k(\gamma_i-\Delta_i)\u_i
  +\sum_{i=1}^k\Delta_i\sum_{\ell=1}^q\nu_\ell\w_\ell
&=\x,\\
\sum_{i,\ell}(\lambda_{i,\ell}-\Delta_{i,\ell})\v_{i,\ell}
+\sum_{i,\ell}\Delta_{i,\ell}\u_{i,\ell}
=\sum_{i=1}^k(\gamma_i-\Delta_i)\u_i
  +\sum_{i=1}^k\Delta_i\u_i
&=\x^*.
\end{align*}
Moreover, for every $i\in[k]$, $\sum_{\ell=0}^q\Delta_{i,\ell}=\Delta_i$ and 
$\u_{i,\ell}=\u_i$  for all $\ell\in\{0,\ldots,q\}$,
so both transferred-mass sums in \eqref{eq:lem:L1dist} are unchanged.
Discarding terms with $\lambda_{i,\ell}=0$  gives the claimed representations and bound.

\end{proof}

\begin{remark}\label{remark:fewVertices}
Note that the $(d^*+1)$ factor in the bound in Lemma \ref{lem:L1dist} comes from an upper-bound on the number of vertices needed to represent some optimal solution $\x^*$ as a convex combination of vertices,  via Carath\'eodory's theorem. However, hypothetically, it may be the case that any optimal solution could be represented by a combination of at most $s^*$ vertices with $s^* \ll  d^*+1$. Indeed, in such case the $(d^*+1)$ factor in the lemma, and as a result also in all further derivations, and in particular in our convergence rates listed in Table \ref{table:res}, could be replaced with $s^*$. However, such definition of sparsity is extremely brittle. Take an optimal solution $\x^*$ which lies in the relative interior of $\mF^*$ and replace it with a point sampled uniformly from an arbitrarily small (non-empty) ball in the relative interior of $\mF^*$ centered at $\x^*$. This will almost surely produce a point that admits a convex combination of no less than $d^*+1$ vertices. We therefore regard $d^*$ as a more robust  sparsity measure.   
\end{remark}

The following lemma is a refinement of Lemma \ref{lem:L1dist} and gives improved mass-transfer bounds for the special case in which the polytope $\mP$ is the Cartesian product of polytopes. The proof is deferred to the appendix.
\begin{lemma}\label{lem:productMassTransfer}[mass transfer for Cartesian products]
Suppose that $\mP=\mP_1\times\cdots\times\mP_m$ is not a singleton,
where each $\mP_j\subset\reals^{n_j}$ is a polytope of the form~\eqref{eq:polyStruct} (under the normalization $\xi=1$), and let $\mV_j$ denote its set of vertices. Write
\[
\x=(\x^{(1)},\ldots,\x^{(m)})\in\mP,
\qquad
\x^*=(\x^{*(1)},\ldots,\x^{*(m)})\in\mX^*.
\]
Since every face of a Cartesian product is a Cartesian product of faces, write
\[
\mF^*=\mF_1^*\times\cdots\times\mF_m^*,
\qquad
d_j^*:=\dim(\mF_j^*) .
\]
Then, $\x$ can be written as a convex combination
$\x=\sum_{i=1}^k\lambda_i\v_i$ of vertices in
$\mV=\mV_1\times\cdots\times\mV_m$, with $\lambda_i>0$ for all $i$, such that $\x^*$ can be written as
\[
\x^*=\sum_{i=1}^k(\lambda_i-\Delta_i)\v_i
+\sum_{i=1}^k\Delta_i\u_i,
\]
where $\Delta_i\in[0,\lambda_i]$ and
$\u_i\in\mV\cap\mF(\x^*)$ for all $i$, and
\begin{align}
\sum_{i=1}^k\Delta_i
&\leq
\sqrt{1+\max_{j\in[m]}d_j^*}\,
\Vert{\x-\x^*}\Vert,
\label{eq:lem:productMassTransfer}\\
\sum_{i=1}^k\Delta_i
\sqrt{\max\Big\{1,
\sum_{\ell\in\mI\setminus\mI(\x)}\u_i(\ell)\Big\}}
&\leq
\sqrt{
\Big(1+\frac12\log D\Big)
\Big(1+\max_{j\in[m]}d_j^*\Big)
}\,
\Vert{\x-\x^*}\Vert.
\label{eq:lem:productMassTransferWeighted}
\end{align}
Here $\log$ denotes the natural logarithm. 
\end{lemma}

In particular, if $\mP$ is a product of $m$ identical polytopes of
diameter $D_0$, then $D=\sqrt{m}D_0$ and
\[
1+\frac12\log D
=
1+\frac12\log D_0+\frac14\log m.
\]
Thus, for a fixed factor, the additional factor in the squared weighted
mass-transfer bound grows only logarithmically with the number of
factors, while the optimal-face dependence is through
$\max_{j\in[m]}d_j^*$.

As an example for the strength of Lemma \ref{lem:productMassTransfer} consider the $[0,1]^n$ hypercube in the representation \eqref{eq:polyStruct}:
\[
\mP_{\rm cube}
=\{(\x,\s)\in\reals^n\times\reals^n:\x+\s=\mathbf{1}_n,\ \x,\s\geq0\}
=(\Delta_2)^n,
\]
where $\Delta_2=\{(u_1,u_2)\in\reals_+^2:u_1+u_2=1\}$. Since the dimension of a face of $\Delta_2$ is at most $1$,
the construction in Lemma~\ref{lem:productMassTransfer} gives
$\sum_{i=1}^k\Delta_i \leq \sqrt{2}\Vert{\x-\x^*}\Vert$.
Hence we obtain an unweighted mass-transfer bound that is independent of both $n$ and $d^*$.

Note however, that even a single equality constraint that couples coordinates can completely destroy such independence of $d^*,n$. For the unit simplex in $\reals^n$ (an intersection of the unit cube with a single linear equality), such a bound fails if we take for instance $\x = \sum_{i=1}^{n/2}\frac{2}{n}\e_i$ and $\x^* = \sum_{j=n/2+1}^{n}\frac{2}{n}\e_j$ (assuming $n$ is even), in which case we only have the bound $\sum_{i=1}^k\Delta_i = 1 = \frac{\sqrt{n}}{2}\Vert{\x-\x^*}\Vert$.

\begin{remark}
While as demonstrated above, the specialization of Lemma \ref{lem:L1dist} to product polytopes can be very significant, we view this as a specialized case, and hence use the generic parameter $d^*+1$ in our main bounds. For product polytopes, the factor $d^*+1$ arising from the mass-transfer bound can be replaced with $1+\max_{j\in[m]}d_j^*$ in Lemma~\ref{lem:pairwiseLB} and in the two-level analysis of Section~\ref{sec:2level}. The weighted bound \eqref{eq:lem:productMassTransferWeighted} also allows replacing $d^*+1$ with
$(1+\frac12\log D)(1+\max_{j\in[m]}d_j^*)$
throughout the analysis of Section~\ref{sec:genPoly}, and hence in both the linear convergence rate and the total linear optimization oracle complexity of Theorem~\ref{thm:generalDICG}. This requires no change to Algorithms~\ref{alg:openingSearch} and~\ref{alg:openingDICG}, and no knowledge of $D$ or the factor dimensions.\end{remark}

The second preliminary tool is the following lemma, which builds on Lemma \ref{lem:L1dist}, and establishes that a face-constrained away vertex ($\v_{-}$) and the standard Frank-Wolfe vertex ($\v_{+}$) can lead to sufficient descent directions which are tied to the dimension of the optimal face $d^*$ and not the ambient dimension $n$. This lemma is a core technical argument that allows us to get the linear rates which scale with $d^*$ instead of $n$.

\begin{lemma}\label{lem:pairwiseLB}[face-constrained descent directions]
Let $\x\in\mP$ and let $\x^*\in\mX^*, \x^*\neq \x$. Let
 \begin{align*}
\v_{+}\in\argmin_{\v\in\mV}\langle{\v,\nabla{}f(\x)}\rangle, \qquad \v_{-}\in\argmax_{\v\in\mV\cap\mF(\x)}\langle{\v,\nabla{}f(\x)}\rangle.  	
 \end{align*}
Then, 
\begin{align*}
\langle{\v_{-}-\x,\nabla{}f(\x)}\rangle + \langle{\x-\v_{+},\nabla{}f(\x)}\rangle \geq \frac{\left({f(\x)-f^*}\right)}{\sqrt{d^*+1}\Vert{\x-\x^*}\Vert}.
\end{align*}
\end{lemma}

\begin{proof}
Denote $h=f(\x)-f^*$ and $\g=\nabla{}f(\x)$. 
From Lemma \ref{lem:L1dist}, we can write
$\x=\sum_{i=1}^k\lambda_i\v_i$ with $\v_i\in\mV\cap\mF(\x)$ and
$\lambda_i>0$, such that
\[
\x^*
=
\sum_{i=1}^k(\lambda_i-\Delta_i)\v_i
+
\sum_{i=1}^k\Delta_i\u_i,
\]
where $\u_i\in\mV\cap\mF(\x^*)$,
$\Delta_i\in[0,\lambda_i]$, and
\[
\sum_{i=1}^k\Delta_i
\leq
\sqrt{d^*+1}\Vert{\x-\x^*}\Vert .
\]
Denote $\Delta=\sum_{i=1}^k\Delta_i$. By convexity of $f$, $h \leq \langle{\x-\x^*,\g}\rangle $.
Using the above decompositions of $\x$ and $\x^*$, we get
\[
\langle{\x-\x^*,\g}\rangle
=
\sum_{i=1}^k\Delta_i\langle{\v_i-\u_i,\g}\rangle .
\]
Since $\v_i\in\mF(\x)$ and $\u_i\in\mV$ for all $i$, and since $\v_{+}$ and $\v_{-}$ are
the minimum and maximum vertices over $\mV$ and $\mV\cap\mF(\x)$ with respect to $\g$, respectively, we have that,
\begin{align*}
h \leq \Delta\langle{\v_{-}-\v_{+},\g}\rangle = \Delta\left({\langle{\v_{-}-\x,\g}\rangle + \langle{\x-\v_{+},\g}\rangle}\right).
\end{align*}
Combining the last inequalities indeed yields
\[
\langle{\v_{-}-\x,\g}\rangle
+
\langle{\x-\v_{+},\g}\rangle
\geq
\frac{h}{\Delta},
\]
and the lemma follows.
\end{proof}

\subsection{An obstacle for sparsity-dependent rates for decomposition-based methods}\label{sec:obstacle}

Having presented Lemmas \ref{lem:L1dist} and \ref{lem:pairwiseLB}, whose bounds scale with the sparsity measure $d^*$ and not with the ambient dimension $n$, we show how these bypass the inherent limitations in the analyses of previous decomposition-based methods \cite{garber2016linearly, lacoste2015global}. We emphasize that below we do not presume to give sparsity-based lower bounds on decomposition-based methods. Rather, we show that the existing architectures for deriving linear convergence rates are not compatible with sparsity-based bounds.

Consider the problem of minimizing the squared Euclidean norm over the hypercube $[0,1]^n$. This could be written in the form \eqref{eq:polyStruct} as:
\begin{align*}
&\min_{(\x,\y)}\{f(\x,\y)=\frac{1}{2}\Vert{\x}\Vert^2\} \\ 
&\textrm{s.t.} ~ (\x,\y)\in\mP_{\rm{cube}} := \{(\z,\w)\in\reals^{n}\times\reals^n~|~\z \geq 0, \w\geq 0, ~\z+\w = \mathbf{1}_n\}.
\end{align*}

Clearly the vertices of $\mP_{\rm{cube}}$ are given by 
\begin{align*}
\mV_{\rm{cube}} = \{(\v, \mathbf{1}-\v)~|~\v\in\{0,1\}^n\}, 
\end{align*}
and the unique minimizer is $(\x^*,\y^*) = (\mathbf{0}_n, \mathbf{1}_n)$, and clearly $f^* = 0$. Note here $(\x^*,\y^*)$ is a vertex and thus in particular $d^* = 0$.

Now fix $\epsilon\in(0,1)$ and consider the point $(\x_{\epsilon}, \y_{\epsilon})\in\mP_{\rm{cube}}$ given by the convex combination:
\begin{align}\label{eq:decomProb:1}
(\x_{\epsilon}, \y_{\epsilon}) = (1-\epsilon)(\x^*,\y^*) + \frac{\epsilon}{n}\sum_{i=1}^n(\e_i, \mathbf{1}_n-\e_i).
\end{align}
A simple calculation yields that
\begin{align*}
&\Vert{(\x_{\epsilon}, \y_{\epsilon}) - (\x^*,\y^*)}\Vert^2 =  \frac{2\epsilon^2}{n}, \qquad f(\x_{\epsilon}, \y_{\epsilon}) = \frac{1}{2}\Vert{\x_{\epsilon}}\Vert^2 = \frac{\epsilon^2}{2n}, \\
&\nabla{}f(\x_{\epsilon}, \y_{\epsilon}) = (\x_{\epsilon}, \mathbf{0}_n) = \left({\frac{\epsilon}{n}\mathbf{1}_n, \mathbf{0}_n}\right).
\end{align*}

While the core function value descent argument in both \cite{garber2016linearly} and \cite{lacoste2015global} is essentially very similar, they are presented somewhat differently. In \cite{garber2016linearly} the main argument is as in Lemma \ref{lem:L1dist} (see Lemma 5.5 in  \cite{garber2016linearly}): 
an upper bound on the overall mass that needs to be transferred from a decomposition of an input point, say $(\x_{\epsilon}, \y_{\epsilon})$, where the decomposition is specific in their case (and not existential as in our Lemma \ref{lem:L1dist}) to some decomposition of the optimal point $(\x^*,\y^*)$, which in our case is unique since it is a vertex. Considering the decomposition in \eqref{eq:decomProb:1} it is clear that a total mass of $\Delta = \epsilon$ needs to be transferred (all the mass from the vertices $(\e_i, \mathbf{1}_n-\e_i), i=1,\dots,n$). However, this yields,
\begin{align*}
\Delta = \epsilon = \sqrt{\frac{n}{2}} \Vert{(\x_{\epsilon}, \y_{\epsilon}) - (\x^*,\y^*)}\Vert,
\end{align*}
which is worse by a factor $\Theta(\sqrt{n})$ than the bound in our decomposition-invariant Lemma \ref{lem:L1dist} (since here $d^*=0$), or Lemma \ref{lem:productMassTransfer}.

Moving to the decomposition-based away step Frank-Wolfe analysis in \cite{lacoste2015global}, their core argument relies on establishing that there exists a polytope-dependent constant $C$ such that for any query point given by a specific decomposition, say   $(\x_{\epsilon}, \y_{\epsilon})$, it holds that
\begin{align}
\langle{\v_- - \v_+, \nabla{}f(\x_{\epsilon}, \y_{\epsilon})}\rangle \geq C\frac{f(\x_{\epsilon}, \y_{\epsilon})- f^*}{\Vert{(\x_{\epsilon}, \y_{\epsilon}) - (\x^*,\y^*)}\Vert},
\end{align}
where $\v_-$ is any vertex in the decomposition \eqref{eq:decomProb:1} that maximizes the inner product with $\nabla{}f(\x_{\epsilon}, \y_{\epsilon})$, and $\v_+$ is the standard FW vertex, i.e., $\v_+\in\argmin_{\v\in\mV_{\rm{cube}}}\langle{\v, \nabla{}f(\x_{\epsilon}, \y_{\epsilon})}\rangle$, see Theorem 3 in \cite{lacoste2015global} (see also Lemma 2.26 in \cite{braun2022conditional}).

From the decomposition in \eqref{eq:decomProb:1} we have that $\v_{-} = (\e_i, \mathbf{1}_n-\e_i)$ for some $i\in[n]$. Also, we clearly have that  $\v_{+} = (\x^*,\y^*)$. This gives,
\begin{align*}
\langle{\v_{-} -\v_{+},\nabla{}f(\x_{\epsilon}, \y_{\epsilon})}\rangle = \frac{\epsilon}{n} = \frac{4}{\sqrt{2n}}\frac{f(\x_{\epsilon}, \y_{\epsilon}) - f^*}{\Vert{(\x_{\epsilon}, \y_{\epsilon}) - (\x^*,\y^*)}\Vert}.
\end{align*}
Thus, we must have that the constant $C$ satisfies $C \leq \frac{4}{\sqrt{2n}}$, which is indeed again worse by a factor $\Theta(\sqrt{n})$ than the bound implied from the decomposition-invariant Lemma \ref{lem:pairwiseLB}, when instantiated with the point $(\x_{\epsilon}, \y_{\epsilon})$.

\section{Algorithm for 2-Level Polytopes}\label{sec:2level}
In this section we focus on the special case that $\mP$ is a 2-level polytope, i.e., for any $\v\in\mV$ and $i\in{}\mI$, if $\v_i > 0$ then $\v_i = a_i$ for some scalar $a_i > 0$. Without loss of generality, we assume the scaling $a_i = 1$, i.e., $\v(i)\in\{0,1\}$ for any $\v\in\mV$ and $i\in \mI$.

Our algorithm for this setting, given as Algorithm \ref{alg:DICG} below, is the same as the DICG method of \cite{garber2016linear_memory}, only that instead of their dyadic step-size rule, which depends on a pre-specified step-size sequence, here we construct a monotone dyadic step-size rule based only on evaluations of the objective function $f$.

\begin{algorithm}
\begin{algorithmic}
\caption{Pairwise Decomposition-Invariant Conditional Gradient with Dyadic Step Sizes for 2-Level Polytopes}\label{alg:DICG}
\STATE input: $\x_1$ --- some vertex in $\mV$
\STATE $\eta \gets 1$
\FOR{$t=1,2,\dots$}
\STATE $\v_{t,+} \gets\arg\min_{\v\in\mV}\langle{\v,\nabla{}f(\x_t)}\rangle$ \COMMENT{Frank-Wolfe vertex}
\IF{$\langle{\x_t - \v_{t,+}, \nabla{}f(\x_t)}\rangle = 0$}
\RETURN
\ENDIF
\STATE $\v_{t,-} \gets\arg\max_{\v\in\mV\cap\mF(\x_t)}\langle{\v,\nabla{}f(\x_t)}\rangle$ \COMMENT{away vertex}
\STATE $i \gets $ smallest integer such that $2^{-i} \leq \eta$ and $f(\x_t + 2^{-i}(\v_{t,+}-\v_{t,-})) < f(\x_t)$
\STATE $\eta \gets 2^{-i}$, $\eta_t \gets \eta / 2$
\STATE $\x_{t+1} \gets \x_t + \eta_t(\v_{t,+}-\v_{t,-})$
\ENDFOR
\end{algorithmic}
\end{algorithm}

The proof of the following lemma is essentially a straightforward consequence of Lemma 1 in \cite{garber2016linear_memory}, however, for the sake of completeness we provide a complete proof in the appendix.
\begin{lemma}\label{lem:dicg:feas}
The iterates of Algorithm \ref{alg:DICG} are always feasible. 
\end{lemma}

Throughout our analysis of Algorithm \ref{alg:DICG} we will use the following notation for any iteration $t\geq 1$:
\begin{eqnarray}\label{eq:notation}
&h_t=f(\x_t)-f^*,\qquad \nabla_t=\nabla f(\x_t),\qquad \d_t=\v_{t,+}-\v_{t,-}, & \nonumber \\
&G_t=\langle{\v_{t,-}-\v_{t,+},\nabla_t}\rangle, \qquad g_t = \langle{\x_t-\v_{t,+},\nabla_t}\rangle.&
\end{eqnarray}
Note that by convexity of $f$ and definition of $\v_{t,-}$, we clearly have 
\begin{align}
\forall t: \qquad h_t \leq g_t \leq G_t.
\end{align}

\begin{theorem}\label{thm:2levelDICG}
Suppose Algorithm \ref{alg:DICG} is initialized with $\x_1\in\mV$ such that $\x_1\in\argmin_{\v\in\mV}\langle{\v,\nabla{}f(\x_0)}\rangle$ for some arbitrary $\x_0\in\mP$. Then, 
\begin{align}
\forall t\geq 1:\qquad h_{t+1} \leq \left( 1-\frac{1}{16}\min\left\{1,\frac{\alpha}{(d^*+1)\beta D^2}\right\} \right)h_t. \label{thm1:res:primLin}
\end{align}
Moreover, even without assuming quadratic growth, the following sublinear convergence guarantees hold:
\begin{align}
&\forall t\geq 1: \qquad h_t \leq \frac{8\beta{}D^2}{t+15}, \label{thm1:res:primSub}\\
&\forall t\geq 3: \qquad \min_{\tau\in\{1,\dots,t\}} g_{\tau} \leq \frac{16\beta{}D^2}{t-1}\log(t-1). \label{thm1:res:dualSub}
\end{align}

\end{theorem}
\begin{remark}\label{rem:dualgap}
Note that the sublinear dual convergence guarantee \eqref{thm1:res:dualSub} is worse by a log factor compared to the standard Frank-Wolfe method \cite{jaggi2013revisiting}. In case the quadratic growth condition holds, then the linear convergence guarantee \eqref{thm1:res:primLin} could be turned to a linear rate for the dual gap using the following generic (i.e., independent of the algorithm used) result (see for instance Theorem 2 in \cite{lacoste2015global}): 
\begin{align}\label{eq:genericDualGap}
h_t \leq \beta{}D^2/2 \quad \Longrightarrow \quad g_t \leq D\sqrt{2\beta{}h_t}. 
\end{align}
\end{remark}

\begin{proof} First note that according to Lemma \ref{lem:dicg:feas}, all iterates are indeed feasible with respect to $\mP$. 

Consider some monotone nonincreasing sequence $(\rho_t)_{t\geq 1}\subset[0,\min\{1,G_t/(\beta{}D^2)\}]$. As a first step, we establish that on any iteration $t$ of Algorithm \ref{alg:DICG} it holds that
\begin{align}\label{eq:thm1:1}
h_{t+1} \leq h_t - \frac{\rho_tG_t}{8}.
\end{align} 
Fix some iteration $t$. From the smoothness of $f$ we have that, 
\begin{align} 
\forall \rho\in[0,\rho_t]: \quad f(\x_t+\rho\d_t) &\leq f(\x_t)+\rho\langle{\d_t,\nabla_t}\rangle +\frac{\beta\rho^2}{2}\Vert{\d_t}\Vert^2 \nonumber \\ 
&\leq f(\x_t)-\rho G_t+\frac{\beta D^2\rho^2}{2} \nonumber\\
& \leq f(\x_t)-\frac{\rho G_t}{2}. \label{eq:thm1:2} 
\end{align}

We next show that 
\begin{align} 
\eta_t\geq \frac{\rho_t}{4}. \label{eq:thm1:3} 
\end{align} 
In the following let $\bar{\eta}_t$ denote the value of the variable $\eta$ after the backtracking step on iteration $t$, i.e., the step-size used  is $\eta_t=\bar{\eta}_t/2$. 

If on iteration $t$ the value of $\eta$ is decreased by the backtracking step, then the previous dyadic value $2\bar{\eta}_t$ did not lead to a decrease in function value. By Eq.~\eqref{eq:thm1:2}, this implies that $2\bar{\eta}_t>\rho_t$, and hence $\eta_t=\bar{\eta}_t/2>\rho_t/4$. Otherwise, if $\eta$ was not decreased on iteration $t$, then either it was never decreased before, in which case $\bar{\eta}_t=1\geq \rho_t$, or, letting $t'<t$ denote the last iteration on which $\eta$ was decreased, we have $\bar{\eta}_t=\bar{\eta}_{t'}$ and, by the previous case, $\bar{\eta}_{t'}>\rho_{t'}/2$. Since $\rho_t\leq\rho_{t'}$, it follows again that $\eta_t=\bar{\eta}_t/2\geq \rho_t/4$. Thus, Eq. \eqref{eq:thm1:3} indeed holds. 

We now prove Eq. \eqref{eq:thm1:1}. We consider two cases. First suppose that $\bar{\eta}_t\leq \rho_t$. Since $\eta_t=\bar{\eta}_t/2\leq\rho_t$, using Eq.~\eqref{eq:thm1:2} and Eq.~\eqref{eq:thm1:3}, we indeed obtain 
\[ 
h_{t+1} \leq h_t-\frac{\eta_tG_t}{2} \leq h_t-\frac{\rho_tG_t}{8}. 
\] 
It remains to consider the case $\bar{\eta}_t>\rho_t$. Define $\phi_t(\rho)=f(\x_t+\rho\d_t)-f^*$. Since $f$ is convex, $\phi_t$ is convex. Since $\bar{\eta}_t>\rho_t$, the point $\bar{\eta}_t/2$ belongs to the interval $[\rho_t/2,\bar{\eta}_t]$. Hence, there exists $\lambda\in[1/2,1]$ such that 
\begin{align*}
\frac{\bar{\eta}_t}{2} = \lambda\frac{\rho_t}{2}+(1-\lambda)\bar{\eta}_t . 
\end{align*} \label{eq:thm1:4}
Thus, we have that
\begin{align*}
h_{t+1} = \phi_t(\bar{\eta}_t/2) &\leq \lambda\phi_t(\rho_t/2)+(1-\lambda)\phi_t(\bar{\eta}_t) \\
&\leq \lambda\left({h_t - \frac{\rho_tG_t}{4}}\right) + (1-\lambda)h_t \\
&\leq h_t - \frac{\rho_tG_t}{8},
\end{align*}
where the first inequality follows from convexity, the second inequality follows from Eq. \eqref{eq:thm1:2} and since by definition of $\bar{\eta}_t$, $\phi_t(\bar{\eta}_t) < h_t$, and the third inequality follows since $\lambda\in[1/2,1]$.

Thus, we proved Eq. \eqref{eq:thm1:1}. Note Eq. \eqref{eq:thm1:1}, by setting $\rho_t = 0$ for all $t$, in particular implies the monotonicity of $(h_t)_{t\geq 1}$:
\begin{align}\label{eq:thm1:45}
\forall t\geq 1: \qquad h_{t+1} \leq h_t.
\end{align}
We now prove Result \eqref{thm1:res:primLin}. Fix some iteration $t$ and let $\x_t^*=\argmin_{\x^*\in\mX^*}\Vert{\x^*-\x_t}\Vert$. Applying Lemma \ref{lem:pairwiseLB} with $\x=\x_t$ and $\x_t^*$, we have 
\begin{align} 
G_t &= \langle{\v_{t,-}-\x_t,\nabla_t}\rangle + \langle{\x_t-\v_{t,+},\nabla_t}\rangle \nonumber \\
& \geq \frac{h_t}{\sqrt{d^*+1}\Vert{\x_t-\x_t^*}\Vert} \nonumber \\
&\geq \sqrt{\frac{\alpha{}h_t}{2(d^*+1)}},
 \label{eq:thm:1:5}
 \end{align}
 where the last inequality follows from the quadratic growth property.
 
Set
\[ 
\forall t\geq 1: \quad \rho_t= \min\left\{ 1,\, \frac{1}{\beta D^2} \sqrt{\frac{\alpha{}h_t}{2(d^*+1)}}\right\}. 
\] 
By Eq. \eqref{eq:thm:1:5}, we have $\rho_t\leq\min\{1,G_t/(\beta D^2)\}$. Also, by Eq. \eqref{eq:thm1:45} we have that this choice indeed satisfies that $(\rho_t)_{t\geq 1}$ is monotone non-increasing. Plugging this choice of $(\rho_t)_{t\geq 1}$ into Eq. \eqref{eq:thm1:1} together with Eq. \eqref{eq:thm:1:5} yields:
\begin{align}\label{eq:eq:thm1:65}
\forall t\geq 1: \quad h_{t+1} &\leq h_t - \frac{1}{8} \min\left\{ 1,\, \frac{1}{\beta D^2} \sqrt{\frac{\alpha{}h_t}{2(d^*+1)}}\right\}\sqrt{\frac{\alpha{}h_t}{2(d^*+1)}} \nonumber \\
&\leq h_t - \frac{1}{8}\min\left\{\beta{}D^2, \frac{\alpha{}h_t}{2\beta{}D^2(d^*+1)}\right\}.
\end{align}

It remains to use the special initialization to bound $h_1$. Let $\x^*\in\mX^*$. Using the smoothness of $f$ again and the choice of $\x_1$ we have that,
\begin{align}\label{eq:thm1:7}
f(\x_1) &\leq f(\x_0)+\langle{\nabla f(\x_0),\x_1-\x_0}\rangle + \frac{\beta}{2}\Vert{\x_1-\x_0}\Vert^2  \nonumber \\
&\leq  f(\x_0) +  \langle{\nabla f(\x_0),\x^*-\x_0}\rangle + \frac{\beta}{2}D^2  \nonumber  \\
&\leq f(\x^*) +   \frac{\beta}{2}D^2,
\end{align}
where the last inequality is again due to the convexity of $f$.

Thus, due to the monotonicity of $(h_t)_{t\geq 1}$ we have that for all $t\geq 1$, $h_t \leq h_1 \leq \beta{}D^2/2$. Plugging into the RHS of \eqref{eq:eq:thm1:65} yields 
\begin{align*}
\forall t\geq 1: \quad h_{t+1} 
\leq \left({1 - \frac{1}{16}\min\left\{1, \frac{\alpha{}}{\beta{}D^2(d^*+1)}\right\}}\right)h_t,
\end{align*}
which is exactly Result \eqref{thm1:res:primLin}.

We now turn to prove the primal sublinear rate \eqref{thm1:res:primSub}.
First note that by definition of $\v_{t,-},\v_{t,+}$, clearly
\begin{align}
G_t=\langle{\v_{t,-}-\v_{t,+},\nabla_t}\rangle \geq \langle{\x_t -\x^*,\nabla_t}\rangle \geq h_t, \label{eq:sublinear:G-lower-bound}
\end{align}
where the last inequality is due to convexity of $f$.

Define now
\begin{align}\label{eq:thm1:6}
\forall t\geq 1: \quad \rho_t=\min\left\{1,\frac{h_t}{\beta D^2}\right\}.
\end{align}
By Eq.~\eqref{eq:sublinear:G-lower-bound} we indeed have that $\rho_t \leq \min\left\{1,G_t/\beta D^2\right\}$, and by Eq. \eqref{eq:thm1:45} we have that the sequence $(\rho_t)_{t\geq 1}$ is monotone non-increasing. Thus, using Eq. \eqref{eq:thm1:1} implies that,
\begin{align}
\forall t\geq 1: \quad h_{t+1}
\leq
h_t-\frac{1}{8}\min\left\{h_t,\frac{h_t^2}{\beta{}D^2}\right\}.
\label{eq:sublinear:main-recursion}
\end{align}

Using \eqref{eq:thm1:7}, Eq.~\eqref{eq:sublinear:main-recursion}
simplifies to
\begin{align}
h_{t+1}
\leq
h_t-\frac{h_t^2}{8\beta{}D^2}.
\label{eq:sublinear:quadratic-recursion}
\end{align}
If $h_t=0$, the theorem is immediate. Otherwise, Eq.~\eqref{eq:sublinear:quadratic-recursion}
gives
\[
\frac{1}{h_{t+1}}
\geq
\frac{1}{h_t-h_t^2/(8\beta{}D^2)}
=
\frac{1}{h_t}\cdot \frac{1}{1-h_t/(8\beta{}D^2)}
\geq
\frac{1}{h_t}+\frac{1}{8\beta{}D^2},
\]
where the last inequality uses $1/(1-a)\geq 1+a$ for $a\in[0,1)$.
Iterating this inequality, we obtain
\[
\frac{1}{h_t}
\geq
\frac{1}{h_1}+\frac{t-1}{8\beta{}D^2}
\geq
\frac{2}{\beta{}D^2}+\frac{t-1}{8\beta{}D^2}
=
\frac{t+15}{8\beta{}D^2}.
\]
Equivalently,
\[
h_t\leq 
\frac{8\beta D^2}{t+15},
\]
as claimed.

Finally, we turn to prove the dual gap convergence \eqref{thm1:res:dualSub}. Consider again the sequence $(\rho_t)_{t\geq 1}$ defined in Eq. \eqref{eq:thm1:6} and recall that Eq. \eqref{eq:thm1:7} implies that $h_1 \leq \beta{}D^2/2$. Using Eq. \eqref{eq:thm1:1} again we have that for any $t$,
\begin{align*}
h_{t+1} \leq h_{t} - \frac{\rho_tG_t}{8} = h_t\left({1 - \frac{G_t}{8\beta{}D^2}}\right) &\leq h_t\exp\left({- \frac{G_t}{8\beta{}D^2}}\right)\leq h_t\exp\left({- \frac{g_t}{8\beta{}D^2}}\right).
\end{align*}
Rolling the above recursion from time $T$  and recalling the generic bound  $g_t \leq \sqrt{2\beta{}D^2h_t}$ (see Eq. \eqref{eq:genericDualGap}) gives,
\begin{align*}
\frac{g_T^2}{2\beta{}D^2} \leq h_T \leq h_1\exp\left({- \frac{1}{8\beta{}D^2}\sum_{t=1}^{T-1}g_t}\right) \leq\frac{\beta{}D^2}{2}\exp\left({- \frac{1}{8\beta{}D^2}\sum_{t=1}^{T-1}g_t}\right).
\end{align*}
Thus, if for all $t\in[T-1]$ we have $g_t \geq \frac{16\beta{}D^2}{T-1}\log(T-1)$ we get that
\begin{align*}
\frac{g_T^2}{2\beta{}D^2} \leq \frac{\beta{}D^2}{2}\cdot \left({\frac{1}{T-1}}\right)^2,
\end{align*}
which implies that $g_T \leq \frac{\beta{}D^2}{T-1}$.

Considering all cases, Result \eqref{thm1:res:dualSub} follows.

\end{proof}

\section{Algorithm for General Polytopes}\label{sec:genPoly}
In this section we consider an arbitrary polytope of the form
\eqref{eq:polyStruct}, under the normalization $\xi=1$. Unfortunately, moving beyond 2-level polytopes breaks the feasibility of using a dyadic step-size sequence as in our Algorithm \ref{alg:DICG}. We thus abandon the use of pairwise steps as in Algorithm \ref{alg:DICG} and we consider using only a single vertex in order to perform the update step on each iteration.
The main difficulty is to control the number of bad away steps without storing a
convex decomposition. This is done by a new and novel approach. First, as opposed to Algorithm \ref{alg:DICG} for 2-level polytopes, here we allow the forward vertex $\v_{t,+}$ to optimize only on the current active face $\mF(\x_t)$ and not over the entire polytope. With this modification both the forward step and the away step cannot increase the dimension of the current face. Staying in the current face of course cannot lead to convergence to an optimal solution. We resolve this by introducing a new novel forward step that allows to increase the dimension of the current face, however it is selected in such a way that the progress in the objective function is proportional to the increase in dimension of the face. Thus, it compensates for the possibility of many subsequent non-productive bad away steps. Importantly, this new step, which we refer to as an \textit{opening step}, is also directly implementable via the linear optimization oracle of the polytope. 

For this purpose, for any $\x\in\mP$, define
\begin{align}\label{eq:opening:mass}
\m_{\x}(i):=\begin{cases}1,i\in\mI\setminus\mI(\x),\\0,\textrm{otherwise},\end{cases}, \qquad \langle{\m_{\x},\v}\rangle=\sum_{i\in\mI\setminus\mI(\x)}\v(i).
\end{align}
We call $\langle{\m_{\x},\v}\rangle$ the \textit{opening mass} of $\v$ at $\x$. It is
nonnegative on $\mP$, and $\langle{\m_{\x},\v}\rangle=0$ if and only if
$\v\in\mF(\x)$. For a vertex, the normalization $\xi=1$  implies
\begin{align}\label{eq:opening:coordinates}
\langle{\m_{\x},\v}\rangle>0\ \Longrightarrow\ \langle{\m_{\x},\v}\rangle\geq1,
\qquad
|\mI(\v)\setminus\mI(\x)|\leq \langle{\m_{\x},\v}\rangle.
\end{align}
Thus, the opening mass $\langle{\m_{\x},\v}\rangle$ of a vertex $\v$ upper-bounds the number of new coordinates indexed by $\mI$ which it introduces.

The following key lemma establishes that with a correctly chosen penalty parameter $\omega$ (the exact details for the tuning of this parameter will be introduced in the sequel) we can write a linear optimization over the polytope, whose solution yields a descent direction similar to that of our Lemma~\ref{lem:pairwiseLB}, but scales also with the square-root of its opening mass $\sqrt{L}$, which is exactly the key ingredient needed to compensate  for potential subsequent bad away steps.

\begin{lemma}[penalized opening descent directions]\label{lem:penalizedDirections}
Fix $\x\in\mP\setminus\mX^*$ and $\x^*\in\mX^*$. Denote:
\begin{align*}
\v_+&\in\argmin_{\v\in\mV\cap\mF(\x)}
\langle{\v,\nabla f(\x)}\rangle,\qquad
\v_-\in\argmax_{\v\in\mV\cap\mF(\x)}
\langle{\v,\nabla f(\x)}\rangle,\\
G&:=\langle{\v_--\v_+,\nabla f(\x)}\rangle,\qquad
g:=\max_{\v\in\mV}\langle{\x-\v,\nabla f(\x)}\rangle.
\end{align*}
Let $0<c_1\leq c_2\leq1/3$. Let $\omega\in\reals$ be such that $G<2\omega$ and
\begin{align}\label{eq:penalized:interval}
c_1\frac{f(\x)-f^*}{\sqrt{d^*+1}\Vert{\x-\x^*}\Vert}
\leq\omega\leq
c_2\frac{f(\x)-f^*}{\sqrt{d^*+1}\Vert{\x-\x^*}\Vert}.
\end{align}
There exists a closed interval contained in
$[1,4(g/\omega)^2]$, whose upper endpoint is four times its lower endpoint,
such that, for every $\widehat{L}$ in this interval, every
\begin{align}\label{eq:penalized:oracle}
\w\in\argmin_{\v\in\mV}
\Big\langle{\v,\nabla f(\x)+\frac{\omega}{2\sqrt{\widehat{L}}}\m_{\x}}\Big\rangle
\end{align}
has opening mass $L:=\langle{\m_{\x},\w}\rangle$ satisfying $\widehat{L}/4\leq L\leq 9\widehat{L}/4$, and
\begin{align}
\langle{\x-\w,\nabla f(\x)}\rangle
&\geq\frac{\sqrt{3}}{2}\omega\sqrt L\geq
\frac{\sqrt{3}c_1\bigl(f(\x)-f^*\bigr)}
{2\sqrt{d^*+1}\Vert{\x-\x^*}\Vert}\sqrt L.\label{eq:penalized:descent}
\end{align}
\end{lemma}

Note that a value of $\omega$ satisfying the assumptions of the lemma exists
if and only if
\[
G<2c_2\frac{f(\x)-f^*}
{\sqrt{d^*+1}\Vert{\x-\x^*}\Vert}.
\]
Indeed, when this inequality holds, one may take for instance
$\omega=c_2(f(\x)-f^*)/(\sqrt{d^*+1}\Vert{\x-\x^*}\Vert)$.
Otherwise, by the definition of $G$,
\[
\max\left\{
\langle{\x-\v_+,\nabla f(\x)}\rangle,\,
\langle{\v_--\x,\nabla f(\x)}\rangle
\right\}
\geq\frac G2
\geq c_2\frac{f(\x)-f^*}
{\sqrt{d^*+1}\Vert{\x-\x^*}\Vert},
\]
which means, as shown below, that choosing between $\v_+$ and $\v_-$ according
to the larger directional gap yields proportional error decrease
after line-search (except possibly when $\v_-$ is used with its
maximal feasible step-size, which we will deal with separately).

\begin{proof}[Proof of Lemma \ref{lem:penalizedDirections}]
By $G<2\omega$, \eqref{eq:penalized:interval}, and $c_2\leq1/3$,
\begin{align}\label{eq:penalized:general-condition}
G+\omega<3\omega
\leq\frac{f(\x)-f^*}{\sqrt{d^*+1}\Vert{\x-\x^*}\Vert}.
\end{align}
Write $\g=\nabla f(\x)$. Let
\[
\v^*\in\argmin_{\v\in\mV}
\left\{\phi(\v) := \langle{\v,\g}\rangle
+\omega\sqrt{\langle{\m_{\x},\v}\rangle}\right\},
\]
and denote $L^*:=\langle{\m_{\x},\v^*}\rangle$. 

The minimization of $\phi$ identifies an ideal vertex $\v^*$ for the
analysis. To obtain a direction using only a linear optimization oracle,
we replace the square-root penalty in $\phi$ by its tangent at the
trial mass $\widehat{L}$:
\[
\omega\sqrt{\widehat{L}}
+\frac{\omega}{2\sqrt{\widehat{L}}}
\left(\langle{\m_{\x},\v}\rangle-\widehat{L}\right)
=
\frac{\omega}{2\sqrt{\widehat{L}}}\langle{\m_{\x},\v}\rangle
+\frac{\omega\sqrt{\widehat{L}}}{2}.
\]
Omitting the last constant term  gives
\eqref{eq:penalized:oracle}.
When $\widehat{L}=L^*$, $\v^*$ also minimizes this linear objective,
and the proof shows that the required guarantees hold throughout
$\widehat{L}\in[L^*,4L^*]$. Thus, the search over $\widehat{L}$ finds
a suitable linearization without computing $\v^*$ or knowing $L^*$.

We first establish that $L^* > 0$.
Suppose by way of contradiction that $L^*=0$. Since $\langle{\m_{\x}, \y}\rangle=0$ for any $\y\in\mF(\x)$, this implies that $\v^*$ is a minimizer of $\langle{\v,\g}\rangle$ over $\mF(\x)$, which in turn implies that for every vertex $\v\in\mV$,
\begin{align*}
\langle{\v_- - \v,\g}\rangle &= \langle{\v_- - \v^*, \g}\rangle + \langle{\v^* - \v, \g}\rangle  \\
&\leq \langle{\v_- - \v_+, \g}\rangle + \phi(\v^*) - \phi(\v) + \omega\sqrt{\langle{\m_{\x}, \v}\rangle} \\
&\leq G+\omega\sqrt{\langle{\m_{\x},\v}\rangle}.
\end{align*}
Apply Lemma~\ref{lem:L1dist} to $\x,\x^*$. Writing $\x, \x^*$ as convex combinations of vertices as in the lemma, and noting that 
every $\v_i$ in the decomposition of $\x$ lies in $\mF(\x)$ (since $\lambda_i>0$), we have that
\begin{align*}
f(\x)-f^* &\leq \langle{\x - \x^*, \g}\rangle =\sum_{i=1}^k\Delta_i\langle{\v_i-\u_i,\g}\rangle  \leq \sum_{i=1}^k\Delta_i\langle{\v_- - \u_i,\g}\rangle\\
&\leq\sum_{i=1}^k\Delta_i
\left(G+\omega\sqrt{\langle{\m_{\x},\u_i}\rangle}\right) \leq(G+\omega)\sum_{i=1}^k\Delta_i
\sqrt{\max\{1,\langle{\m_{\x},\u_i}\rangle\}}\\
&\leq(G+\omega)\sqrt{d^*+1}\Vert{\x-\x^*}\Vert,
\end{align*}
which contradicts Eq. \eqref{eq:penalized:general-condition}. 

Therefore, we must have that $L^* > 0$, and since we assume the  normalization $\xi =1$, we in particular have $L^*\geq1$.
Comparison with $\v_+$, and then with $\x$, and recalling that $\langle{\m_{\x}, \y}\rangle=0$ for every $\y\in\mF(\x)$, gives the inequalities 
\begin{align}\label{eq:lem:penalizedDirections:1}
\omega\sqrt{L^*} \leq  \min_{\v\in\mV\cap\mF(\x)}\phi(\v) - \phi(\v^*) + \omega\sqrt{L^*} &= \min_{\v\in\mV\cap\mF(\x)}\langle{\v-\v^*,\g}\rangle \nonumber \\
&\leq \langle{\x-\v^*,\g}\rangle\leq g.
\end{align}
This proves $L^*\leq(g/\omega)^2$. Consequently, the closed interval
$[L^*,4L^*]$ is contained in $[1,4(g/\omega)^2]$ and has upper
endpoint four times its lower endpoint.

Fix $\widehat{L}\in[L^*,4L^*]$ and any minimizer $\w$ in
\eqref{eq:penalized:oracle}. Write $L=\langle{\m_{\x},\w}\rangle$.
Optimality of $\v^*$ and $\w$ for their respective objectives gives
\begin{align}\label{eq:lem:penalizedDirections:2}
\omega(\sqrt{L^*}-\sqrt L)
\leq\langle{\w-\v^*,\g}\rangle
\leq\frac{\omega(L^*-L)}{2\sqrt{\widehat{L}}}.
\end{align}
Factoring the RHS and rearranging gives
\[
(\sqrt L-\sqrt{L^*})
(\sqrt L+\sqrt{L^*}-2\sqrt{\widehat{L}})\leq0.
\]
Since $\widehat{L}\geq L^*$, the two roots of this quadratic in $\sqrt L$ are
ordered as $\sqrt{L^*}\leq2\sqrt{\widehat{L}}-\sqrt{L^*}$. Thus, we have that
$\sqrt{L^*}\leq\sqrt L\leq2\sqrt{\widehat{L}}-\sqrt{L^*}$.
Using $\widehat{L}\leq4L^*$, we obtain $\sqrt{\widehat{L}}/2\leq\sqrt L\leq 3\sqrt{\widehat{L}}/2$,
which proves the bounds on $\widehat{L}$ in \eqref{eq:penalized:descent}.

We turn to prove the second part of \eqref{eq:penalized:descent}. Using Eq. \eqref{eq:lem:penalizedDirections:1} and  \eqref{eq:lem:penalizedDirections:2} we have that
\begin{align*}
\langle{\x-\w,\g}\rangle &= \langle{\x-\v^*,\g}\rangle + \langle{\v^*-\w,\g}\rangle \\
&\geq \omega\sqrt{L^*} + \frac{\omega(L-L^*)}{2\sqrt{\widehat{L}}} \\
&= \frac{\omega L}{2\sqrt{\widehat{L}}}
+\omega\sqrt{L^*}-\frac{\omega L^*}{2\sqrt{\widehat{L}}}\nonumber\\
&\geq \frac{\omega L}{2\sqrt{\widehat{L}}}
+\frac{3\omega\sqrt{\widehat{L}}}{8}
\geq \frac{\sqrt{3}}{2}\omega\sqrt L.
\end{align*}
The second inequality follows from
$\sqrt{L^*/\widehat{L}}\in[1/2,1]$ and
\[
\sqrt{\frac{L^*}{\widehat{L}}}
-\frac{L^*}{2\widehat{L}}
=
\frac12\left[1-\left(1-\sqrt{\frac{L^*}{\widehat{L}}}\right)^2\right]
\geq\frac38.
\]
The third inequality is the AM--GM inequality.

Plugging-in the assumed lower-bound on $\omega$ completes the proof of \eqref{eq:penalized:descent}.
\end{proof}

\subsection{The algorithm and its guarantee}
We use $h_t$ and $\nabla_t$ as in \eqref{eq:notation}. Recall that in this section,
both the forward and away vertices are chosen from the current face:
\begin{align}\label{eq:opening:localgaps}
\v_{t,+}&\in\argmin_{\v\in\mV\cap\mF(\x_t)}
\langle{\v,\nabla_t}\rangle,\qquad
\v_{t,-}\in\argmax_{\v\in\mV\cap\mF(\x_t)}
\langle{\v,\nabla_t}\rangle,\nonumber\\
G_t&:=\langle{\v_{t,-}-\v_{t,+},\nabla_t}\rangle,\qquad
g_t:=\max_{\v\in\mV}\langle{\x_t-\v,\nabla_t}\rangle.
\end{align}
Here $G_t$ is the pairwise gap within $\mF(\x_t)$, while $g_t$ remains
the global Frank-Wolfe gap. For a nonzero
away direction $\x-\v$, where $\v\in\mF(\x)$, its maximal feasible step is
\begin{align}\label{eq:opening:awaycap}
\gamma(\x,\v):=
\min_{i\in\mI:\,\v(i)>\x(i)}
\frac{\x(i)}{\v(i)-\x(i)},
\end{align}
which is strictly positive.

On each iteration $t$, Algorithm~\ref{alg:openingDICG} computes the two face-constrained vertices $\v_{t,+}, \v_{t,-}$,
and their gap $G_t$, and passes $G_t$ and $g_t$ to
Algorithm~\ref{alg:openingSearch} which is responsible for computing the \textit{opening vertex candidate}. Algorithm~\ref{alg:openingSearch} maintains a single penalty
$\omega>0$, which only decreases: it initializes $\omega=g$ on its first
invocation and retains its current value between invocations.
When $G\geq2\omega$, it returns $(\x,0)$, indicating that a local step using either $\v_{t,+}$ or $\v_{t,-}$ can produce sufficient progress. Otherwise, it scans the trial opening masses $\widehat{L}_j=(2^j)^2$ for
$j=0,\ldots,j_{\max}$. At trial mass $\widehat{L}_j$, the penalty coefficient is
$\omega/(2\sqrt{\widehat{L}_j})$, as in \eqref{eq:penalized:oracle}. Each query is a standard linear
optimization call over $\mP$, and $L$ denotes the actual opening mass of
the returned vertex. Such a vertex is accepted under the two conditions  $L\geq\widehat{L}_j/4$ and
the squared directional-gap bound in \eqref{eq:penalized:descent}. The first
requirement controls the number of search queries, and the second gives
decrease proportional to the opening mass. An accepted vertex is returned together
with its opening mass. If no candidate is accepted, the penalty is reduced
and the local-gap test is repeated.

Algorithm~\ref{alg:openingDICG} uses the accepted opening vertex when the
returned mass is positive, and otherwise selects the better of the local
forward and away directions. The step-size is chosen by standard exact
line-search over the feasible interval.

\begin{algorithm}
\begin{algorithmic}
\caption{Search for an opening vertex}
\label{alg:openingSearch}
\STATE input: $\x\in\mP$, $\g=\nabla f(\x)$, its global FW gap $g>0$, and its local pairwise gap $G$
\STATE persistent penalty: $\omega$, initialized to $g$ on the first call in a run of Algorithm~\ref{alg:openingDICG} and retained between calls
\STATE $\m_{\x}(i)\gets\begin{cases}1,&i\in\mI\ \textrm{and}\ \x(i)=0,\\0,&\textrm{otherwise}\end{cases}$
\WHILE{$G<2\omega$}
\STATE $j_{\max}\gets\left\lceil\log_4\left(4\max\{1,(g/\omega)^2\}\right)\right\rceil$
\FOR{$j=0,\ldots,j_{\max}$}
\STATE $\widehat{L}_j\gets(2^j)^2$, $\w\gets\argmin_{\v\in\mV}
\left\langle{\v,\g+\frac{\omega}{2\sqrt{\widehat{L}_j}}\m_{\x}}\right\rangle$
\STATE $L\gets\langle{\m_{\x},\w}\rangle=\sum_{i\in\mI:\,\x(i)=0}\w(i)$
\IF{$L\geq\widehat{L}_j/4$ \textbf{and} $4\langle{\x-\w,\g}\rangle^2\geq3\omega^2L$}
\RETURN $(\w,L)$
\ENDIF
\ENDFOR
\STATE $\omega\gets\frac{\sqrt{3}}{2}\omega$
\ENDWHILE
\RETURN $(\x,0)$
\end{algorithmic}
\end{algorithm}

\begin{algorithm}[H]
\begin{algorithmic}
\caption{Decomposition-Invariant Conditional Gradient with Opening-Mass Search and Line-Search}
\label{alg:openingDICG}
\STATE input: $\x_0\in\mP$
\STATE $\x_1\gets\argmin_{\v\in\mV}\langle{\v,\nabla f(\x_0)}\rangle$
\FOR{$t=1,2,\ldots$}
\STATE $\nabla_t\gets\nabla f(\x_t)$,
$g_t\gets\max_{\v\in\mV}\langle{\x_t-\v,\nabla_t}\rangle$
\IF{$g_t=0$}
\RETURN $\x_t$
\ENDIF
\STATE $\v_{t,+}\gets\argmin_{\v\in\mV\cap\mF(\x_t)}\langle{\v,\nabla_t}\rangle$,
$\v_{t,-}\gets\argmax_{\v\in\mV\cap\mF(\x_t)}\langle{\v,\nabla_t}\rangle$
\STATE $G_t\gets\langle{\v_{t,-}-\v_{t,+},\nabla_t}\rangle$
\STATE $(\w_t,L_t)\gets\text{Algorithm~\ref{alg:openingSearch}}(\x_t,\nabla_t,g_t,G_t)$
\IF{$L_t>0$}
\STATE $(\q_t,\gamma_t)\gets(\w_t-\x_t,1)$
\ELSIF{$\langle{\x_t-\v_{t,+},\nabla_t}\rangle\geq\langle{\v_{t,-}-\x_t,\nabla_t}\rangle$}
\STATE $(\q_t,\gamma_t)\gets(\v_{t,+}-\x_t,1)$
\ELSE
\STATE $(\q_t,\gamma_t)\gets(\x_t-\v_{t,-},\gamma(\x_t,\v_{t,-}))$
\ENDIF
\STATE $\eta_t\gets\argmin_{\eta\in[0,\gamma_t]}f(\x_t+\eta\q_t)$
\STATE $\x_{t+1}\gets\x_t+\eta_t\q_t$
\ENDFOR
\end{algorithmic}
\end{algorithm}

We can now formally state our main theorem for general polytopes. 
\begin{theorem}\label{thm:generalDICG}
Suppose $f$ is convex and $\beta$-smooth on $\mP$, and the quadratic growth
condition \eqref{eq:qg} holds.
For every iterate $\x_t$, $t\geq1$, generated by
Algorithm~\ref{alg:openingDICG},
\begin{align}\label{eq:opening:linearrate}
h_t\leq\frac{\beta D^2}{2}
\exp\left(-\frac{1}{128}\min\left\{1,\frac{\alpha}{(d^*+1)\beta D^2}\right\}(t-1)\right).
\end{align}
In particular, for every $\epsilon>0$, Algorithm~\ref{alg:openingDICG} produces and
certifies a feasible iterate $\x_t$ with $h_t\leq g_t\leq\epsilon$ with
total number of global and face-constrained linear optimization calls upper bounded by
\begin{align}\label{eq:opening:complexity}
O\left(
\max\left\{1,\frac{(d^*+1)\beta D^2}{\alpha}\right\}\left[1+\log\max\left\{1,\frac{\beta D^2}{\epsilon}\right\}\right]
\right).
\end{align}
\end{theorem}

\subsection{Proof of Theorem \ref{thm:generalDICG}}
We first establish the decrease bounds needed for the linear rate.
Lemma~\ref{lem:openingSearch} treats the newly-introduced opening steps, showing that
their error reduction scales with their opening mass.
Lemma~\ref{lem:localUpdate} treats updates using $\v_{t,+}$ or
$\v_{t,-}$, giving sufficient decrease except for bad away steps.
These lemmas also describe how each type of update changes the
face dimension, which allows Lemma~\ref{lem:openingBadSteps} to
bound the number of bad away steps by the accumulated opening mass.
We then combine these results to prove the linear rate, using the
additional decrease on opening steps to account for the bad steps.
Next, we bound the total number of linear optimization calls.
Lemma~\ref{lem:openingFailedSearches} accounts for the discarded
calls inside Algorithm~\ref{alg:openingSearch} which do not lead to termination of the algorithm. We combine this bound with the decrease estimates to
obtain the oracle complexity.
Supporting these arguments is the auxiliary
Lemma~\ref{lem:openingPenalty}, which we prove first.
Using Lemma~\ref{lem:penalizedDirections}, it verifies that the
adaptive search terminates and maintains a sufficiently large
penalty for the decrease bounds to hold. It also provides the
control on penalty reductions needed for the oracle count.

We first note that due to the use of line-search (with feasible intervals) all iterates of Algorithm \ref{alg:openingDICG} are feasible and $f(\x_t)$ is nonincreasing for $t\geq1$. The initialization argument in \eqref{eq:thm1:7} applies verbatim and gives
$h_1\leq \beta D^2/2$. Consequently, \eqref{eq:genericDualGap} gives
\begin{align}\label{eq:opening:smallgap}
h_t\leq \beta D^2/2,\qquad g_t^2\leq2\beta D^2h_t,\qquad g_t\leq \beta D^2.
\end{align}
Any iteration of Algorithm \ref{alg:openingDICG} is with $g_t>0$, and hence $h_t>0$.
Throughout the analysis we denote the dimension of the current active face $d_t:=\dim\mF(\x_t)$. Throughout, an away step (i.e., an update with the vertex $\v_{t,-}$) is called \textit{bad} if its line-search step size equals its feasible cap,
$\eta_t=\gamma_t$; all other local steps (updates with either $\v_{t,+}$ or $\v_{t,-}$) are called \textit{good}.

\begin{lemma}[lower bound on the penalty $\omega$]\label{lem:openingPenalty}
On any invocation by Algorithm \ref{alg:openingDICG} on some iteration $t$, the penalty $\omega$ maintained by Algorithm~\ref{alg:openingSearch} satisfies
\begin{align}\label{eq:opening:scales}
\omega^2\geq\frac{\alpha h_t}{24(d^*+1)},
\end{align}
 including after every reduction made on that invocation. In particular, since $h_t>0$ remains fixed during an invocation, \eqref{eq:opening:scales}  rules out infinitely many reductions of $\omega$. Each for loop is finite, and so Algorithm \ref{alg:openingSearch} returns after finitely many linear optimization calls.
\end{lemma}

\begin{proof}
At initialization of $\omega$, since $\x_1$ is a vertex, Lemma~\ref{lem:pairwiseLB} and the quadratic growth
\eqref{eq:qg} give $\omega^2=g_1^2\geq(\alpha h_1)/(2(d^*+1))$.

An iterate update in Algorithm \ref{alg:openingDICG} preserves the claimed bound,
since by the use of line-search $h_t$ does not increase and $\omega$ is unchanged.

Consider now a reduction to $\omega$ during the run of Algorithm~\ref{alg:openingSearch}. The preceding for loop was entered,
so $G_t<2\omega$. If $\omega\leq h_t/(3\sqrt{d^*+1}\dist(\x_t,\mX^*))$,
Lemma~\ref{lem:penalizedDirections}, with $c_2=1/3$ and sufficiently
small $c_1>0$, guarantees acceptance: its interval contains a trial
mass tested by the for loop, and \eqref{eq:penalized:descent} implies
both acceptance conditions. Hence a reduction occurs only when
\[
\omega>
\frac{h_t}{3\sqrt{d^*+1}\dist(\x_t,\mX^*)}
\geq\sqrt{\frac{\alpha h_t}{18(d^*+1)}},
\]
where the second inequality uses the quadratic growth \eqref{eq:qg}.
The replacement $\omega\gets\sqrt{3}\omega/2$ therefore preserves
\eqref{eq:opening:scales}.
\end{proof}

\begin{lemma}[guarantees for Algorithm \ref{alg:openingSearch}]\label{lem:openingSearch}
Consider an invocation of Algorithm~\ref{alg:openingSearch} by
Algorithm~\ref{alg:openingDICG} at iterate $\x_t$.
If it returns $(\w_t,L_t)$ with $L_t>0$, the subsequent exact line-search
update satisfies
\begin{align}\label{eq:opening:searchprogress}
h_t-h_{t+1}\geq\frac{L_t h_t}{64}\frac{\alpha}{(d^*+1)\beta D^2},\qquad
d_{t+1}\leq d_t+L_t.
\end{align}
\end{lemma}
\begin{proof}
Suppose Algorithm~\ref{alg:openingSearch} returns $(\w_t,L_t)$ with
$L_t>0$, and let $\omega$ and $\widehat{L}_j$ be the penalty and
trial mass at which $\w_t$ is accepted.
Since $\langle{\m_{\x_t},\x_t}\rangle=0$, comparison of the oracle
output with $\x_t$ gives
\[
\rho_t:=\langle{\x_t-\w_t,\nabla_t}\rangle
\geq\frac{\omega L_t}{2\sqrt{\widehat{L}_j}}>0.
\]
Moreover, $\rho_t\leq g_t\leq\beta D^2$, so
$\rho_t/(\beta D^2)\in[0,1]$ is a feasible comparison step.

By smoothness of $f$, for every $\eta\in[0,1]$,
\begin{align*}
f(\x_t+\eta(\w_t-\x_t))
&\leq f(\x_t)-\eta\rho_t+\frac{\beta D^2\eta^2}{2}.
\end{align*}
Exact line-search yields an objective value no larger than that
obtained with the comparison step $\rho_t/(\beta D^2)$. Hence
\begin{align*}
h_t-h_{t+1}
&=f(\x_t)-f(\x_{t+1})\geq f(\x_t)-
f\left(\x_t+\frac{\rho_t}{\beta D^2}(\w_t-\x_t)\right)\\
&\geq\frac{\rho_t^2}{\beta D^2}
-\frac{\rho_t^2}{2\beta D^2}
=\frac{\rho_t^2}{2\beta D^2}.
\end{align*}
The second acceptance condition gives
$4\rho_t^2\geq3\omega^2L_t$, and therefore
\[
h_t-h_{t+1}\geq\frac{3\omega^2L_t}{8\beta D^2}.
\]
By Lemma~\ref{lem:openingPenalty},
$\omega^2\geq\alpha h_t/(24(d^*+1))$. Substituting this bound gives
\begin{align*}
h_t-h_{t+1}
&\geq
\frac{3L_t}{8\beta D^2}\cdot
\frac{\alpha h_t}{24(d^*+1)}
=\frac{\alpha L_t h_t}{64(d^*+1)\beta D^2},
\end{align*}
which proves the first part of \eqref{eq:opening:searchprogress}.

It remains to bound the increase in the face dimension. For any feasible $\x$ we have that
\begin{align*}
\dim(\mF(\x))
= n-\operatorname{rank}
\begin{bmatrix}
\A \vspace{-6pt}\\ 
\matE_{J(\x)}
\end{bmatrix},
\end{align*}
where $J(\x) := \left\{i \in \mI :
\x(i)=0 \right\}$ and
$\matE_{J(\x)} :=
\begin{bmatrix}
\e_i^\top
\end{bmatrix}_{i\in J(\x)}
\in \mathbb{R}^{|J(\x)|\times n}$.
Thus, for feasible $\x,\y$, $\dim\mF(\y)-\dim\mF(\x)
\leq|\mI(\y)\setminus\mI(\x)|$.
Since $\x_{t+1}$ is a convex combination of $\x_t$ and $\w_t$,
$\vert{\mI(\x_{t+1})\setminus\mI(\x_t)}\vert
\leq 
\vert{\mI(\w_t)\setminus\mI(\x_t)}\vert \leq \langle{\m_{\x_t},\w_t}\rangle = L_t$.
Thus, $d_{t+1}\leq d_t+L_t$, completing the proof of
\eqref{eq:opening:searchprogress}.
\end{proof}

\begin{lemma}[descent for a local update]\label{lem:localUpdate}
Suppose Algorithm~\ref{alg:openingDICG} updates using $\v_{t,+}$ or
$\v_{t,-}$. Then $d_{t+1}\leq d_t$. If the step is good, then
\begin{align}\label{eq:opening:localprogress}
h_t-h_{t+1}\geq\frac{h_t}{48}\min\left\{1,\frac{\alpha}{(d^*+1)\beta D^2}\right\}.
\end{align}
If the step is bad, then
\begin{align}\label{eq:opening:localdrop}
h_{t+1}\leq h_t,\qquad d_{t+1}\leq d_t-1.
\end{align}
\end{lemma}
\begin{proof}

Both types of updates keep the next iterate in $\mF(\x_t)$, and so $d_{t+1}\leq d_t$. If an away step attains its maximal feasible cap,
it makes a previously positive constrained coordinate zero. Its endpoint
therefore lies in a proper subface of $\mF(\x_t)$, and so in this case $d_{t+1} < d_t$. The objective is nonincreasing
by the use of line-search. It thus remains to prove \eqref{eq:opening:localprogress} for any good step.

$\v_{t,+}$ or $\v_{t,-}$ is used only when Algorithm~\ref{alg:openingSearch}
returns $(\x_t,0)$, and hence $G_t  \geq2\omega$.
The selected direction $\q_t$ has the larger of the two
local directional gaps, whose sum is $G_t$. Hence
$\rho_t:=-\langle{\q_t,\nabla_t}\rangle
\geq G_t/2\geq\omega>0$.

Write $\psi_t(\eta):=f(\x_t+\eta\q_t)$. Since
$\Vert{\q_t}\Vert\leq D$, smoothness gives
\begin{align}\label{eq:opening:smoothtrial}
\psi_t(\eta)\leq f(\x_t)-\eta\rho_t+\beta D^2\eta^2/2
\qquad 0\leq\eta\leq\gamma_t.
\end{align}
If the direction is forward (update with $\v_{t,+}$), $\gamma_t=1$ and
$\rho_t\leq g_t\leq \beta D^2$, so $\rho_t/(\beta D^2)$ is a feasible comparison step.
If the direction is away, goodness of the step and $\psi_t'(0)=-\rho_t<0$ imply that
$0<\eta_t<\gamma_t$ and so $\psi_t'(\eta_t)=0$. Lipschitz continuity of
the gradient on the feasible segment gives
\begin{align*}
\rho_t &=\psi_t'(\eta_t)-\psi_t'(0) =\langle{\q_t,\nabla f(\x_t+\eta_t\q_t)-\nabla f(\x_t)}\rangle\\
&\leq
\beta\Vert{\q_t}\Vert\,
\Vert{(\x_t+\eta_t\q_t)-\x_t}\Vert
\leq\beta D^2\eta_t.
\end{align*}
Hence $\rho_t/(\beta D^2)\leq\eta_t<\gamma_t$ is feasible also in case of a good away step.
Exact line-search and \eqref{eq:opening:smoothtrial}, evaluated at
$\rho_t/(\beta D^2)$, indeed give
\begin{align}\label{eq:opening:gapdecrease}
h_t-h_{t+1}&\geq\frac{\rho_t^2}{2\beta D^2}
\geq\frac{\omega^2}{2\beta D^2}
\geq\frac{\alpha h_t}{48(d^*+1)\beta D^2}\geq\frac{h_t}{48}\min\left\{1,\frac{\alpha}{(d^*+1)\beta D^2}\right\},
\end{align}
where we used \eqref{eq:opening:scales}, proving
\eqref{eq:opening:localprogress}.

\end{proof}

\begin{lemma}[number of bad away steps]\label{lem:openingBadSteps}
Among the first $T$ completed updates of Algorithm~\ref{alg:openingDICG},
the number of bad away steps is at most $\sum_{t=1}^{T-1}L_t$.
\end{lemma}
\begin{proof}
Since $\x_1$ is a vertex, $d_1=0$.
By Lemma~\ref{lem:localUpdate}, every bad step decreases the face
dimension by at least one, while all other updates using $\v_{t,+}$
or $\v_{t,-}$ do not increase it.
By Lemma~\ref{lem:openingSearch}, an update using $\w_t$ increases
the face dimension by at most $L_t$.

If there are no bad steps, the claim is immediate. Otherwise, sum
these dimension changes up to and including the last bad step among
the first $T$ updates. Since the resulting face dimension is
nonnegative, the number of bad steps is at most the total opening
mass accumulated strictly before that step, which is at most
$\sum_{t=1}^{T-1}L_t$.
\end{proof}

\begin{lemma}[number of discarded linear optimization calls in Algorithm~\ref{alg:openingSearch}]\label{lem:openingFailedSearches}
Fix $\epsilon\in(0, g_1)$. Consider the first $T$ completed iterations of
Algorithm~\ref{alg:openingDICG}, with $g_t>\epsilon$ for every
$t=1,\ldots,T$. A linear optimization oracle call inside
Algorithm~\ref{alg:openingSearch} is \textit{discarded} if its candidate is
not accepted and returned at that query. The total number of discarded
calls in these $T$ iterations is at most
\begin{align}\label{eq:opening:failedsearchcount}
\sum_{t=1}^{T}L_t+O\Bigg(
&\left[1+\log\max\left\{1,\frac{(d^*+1)\beta D^2}{\alpha}\right\}\right]
\notag\\
&\quad\cdot
\left[1+\log\left(\frac{\beta D^2}{\epsilon}\right)
+\log\max\left\{1,\frac{(d^*+1)\beta D^2}{\alpha}\right\}\right]
\Bigg).
\end{align}
\end{lemma}
\begin{proof}
We split the counting of discarded calls into two parts: we first count the number of discarded calls originating from a full scan of the for loop in Algorithm \ref{alg:openingSearch} that did not terminate with a return of a successful  candidate, i.e., $\omega$ was reduced at the end of the for loop. We then count the discarded calls from scans of this for loop that did return a successful candidate.

Starting with the calls corresponding to full and unsuccessful scans of the for loop, observe that for every value of $\omega$ used in such a scan,
Eq. \eqref{eq:opening:smallgap} and \eqref{eq:opening:scales} imply that,
\begin{align}\label{eq:opening:gridbound}
\frac{g_t}{\omega}
\leq\sqrt{\frac{48(d^*+1)\beta D^2}{\alpha}}
\leq\sqrt{48\max\left\{1,\frac{(d^*+1)\beta D^2}{\alpha}\right\}}.
\end{align}
By the definition of the grid, each full scan of the for loop followed by a penalty reduction
therefore makes
\begin{align}\label{eq:opening:smallgap:0}
O\left(1+\log\max\left\{1,\frac{(d^*+1)\beta D^2}{\alpha}\right\}\right)
\end{align}
linear optimization calls. During these iterations, $g_t>\epsilon$, so
\eqref{eq:opening:smallgap} gives
$h_t>\epsilon^2/(2\beta D^2)$. Together with
\eqref{eq:opening:scales}, this implies
\begin{align}\label{eq:opening:penaltylowerbound}
\omega
>\frac{\epsilon}{4\sqrt{3}}\sqrt{\frac{\alpha}{(d^*+1)\beta D^2}}
\geq\frac{\epsilon}{4\sqrt{3}}
\sqrt{\min\left\{1,\frac{\alpha}{(d^*+1)\beta D^2}\right\}}.
\end{align}
Since the initial penalty is $g_1\leq\beta D^2$ and each reduction multiplies it
by the fixed factor $\sqrt{3}/2<1$, the total number of penalty reductions is
\begin{align}\label{eq:opening:failurecount}
O\left(
1+\log\left(\frac{\beta D^2}{\epsilon}\right)
+\log\max\left\{1,\frac{(d^*+1)\beta D^2}{\alpha}\right\}
\right).
\end{align}
The penalty persists across iterations, so this counts reductions over all
$T$ iterations. Multiplying the bound on calls per scan \eqref{eq:opening:smallgap:0} by
\eqref{eq:opening:failurecount} gives the logarithmic term in
\eqref{eq:opening:failedsearchcount}.

It remains to count discarded calls in scans of the for loop that do terminate with acceptance of a vertex.
If a vertex is accepted at index $j$ of the for loop in iteration $t$ of Algorithm \ref{alg:openingDICG}, exactly $j$ earlier
queries in that scan are discarded. The acceptance condition gives
\[
j\leq4^{j-1}=\frac{\widehat{L}_j}{4}\leq L_t
\qquad j=0,1,\ldots.
\]
If Algorithm~\ref{alg:openingSearch} returns $(\x_t,0)$, every query it made
belongs to a scan followed by a penalty reduction and hence of the previous kind, and $L_t=0$.
Thus the overall number of discarded calls in scans for the for loop terminated with success is indeed at most
$\sum_{t=1}^{T}L_t$ which completes the proof of  \eqref{eq:opening:failedsearchcount}.
\end{proof}

We can now complete the proof of Theorem~\ref{thm:generalDICG}.
\begin{proof}[Proof of Theorem~\ref{thm:generalDICG}]
Let $N_T$ denote the number of good updates using $\v_{t,+}$ or
$\v_{t,-}$ among the first $T$ completed updates.
Every update using $\w_t$ has $L_t\geq1$, while all other updates have $L_t=0$.
Thus the number of updates using $\w_t$ is at most
$\sum_{t=1}^{T}L_t$.
Lemma~\ref{lem:openingBadSteps} bounds the number of bad steps by
$\sum_{t=1}^{T-1}L_t$. Consequently,
\[
T\leq N_T+\sum_{t=1}^{T}L_t+\sum_{t=1}^{T-1}L_t
\leq N_T+2\sum_{t=1}^{T}L_t.
\]

We now apply the two decrease bounds separately.
Lemma~\ref{lem:localUpdate} gives the decrease on each of the
$N_T$ good updates using $\v_{t,+}$ or $\v_{t,-}$.
Lemma~\ref{lem:openingSearch} gives a decrease proportional to
$L_t$ on each update using $\w_t$.
Bad steps do not increase the error. Hence, using $1-z\leq e^{-z}$
and multiplying the decrease bounds,
\begin{align}\label{eq:opening:weightedcontraction}
h_{T+1}\leq h_1
\exp\left(
-\min\left\{1,\frac{\alpha}{(d^*+1)\beta D^2}\right\}
\left[
\frac{N_T}{48}+\frac{1}{64}\sum_{t=1}^{T}L_t
\right]
\right).
\end{align}
The preceding counting bound gives
\[
\frac{N_T}{48}+\frac{1}{64}\sum_{t=1}^{T}L_t
\geq
\frac{1}{128}\left(N_T+2\sum_{t=1}^{T}L_t\right)
\geq\frac{T}{128}.
\]
Plugging-in this bound into \eqref{eq:opening:weightedcontraction} and recalling $h_1\leq\beta D^2/2$ proves
\eqref{eq:opening:linearrate}.

Let us now upper-bound the number of calls to the optimization oracles.
If $g_1\leq\epsilon$, the initial gap computation already certifies
the required accuracy. Otherwise, let $\x_{T+1}$ be the first
iterate with $g_{T+1}\leq\epsilon$.
The rate just proved and
$g_t^2\leq2\beta D^2h_t$ from
\eqref{eq:opening:smallgap} imply that such an iterate exists and
that the number $T$ of updates has the order stated in
\eqref{eq:opening:complexity}.

To bound the additional calls inside
Algorithm~\ref{alg:openingSearch}, we upper bound the total opening mass $\sum_{t=1}^TL_t$.
Since $g_T>\epsilon$, applying
\eqref{eq:opening:weightedcontraction} to the first $T-1$ updates
and retaining only the opening-mass contribution gives
\[
\frac{\epsilon^2}{2\beta D^2}
<h_T
\leq\frac{\beta D^2}{2}
\exp\left(
-\frac{1}{64}
\min\left\{1,\frac{\alpha}{(d^*+1)\beta D^2}\right\}
\sum_{t=1}^{T-1}L_t
\right).
\]
Here the first inequality follows from
\eqref{eq:opening:smallgap}.
For the final update, Lemma~\ref{lem:openingSearch} and
$h_T-h_{T+1}\leq h_T$ give
$L_T\leq64\max\{1,(d^*+1)\beta D^2/\alpha\}$.
Taking logarithms in the preceding display and adding this
last contribution therefore gives
\[
\sum_{t=1}^{T}L_t
\leq
64\max\left\{1,\frac{(d^*+1)\beta D^2}{\alpha}\right\}
\left[1+2\log\left(\frac{\beta D^2}{\epsilon}\right)\right].
\]

Lemma~\ref{lem:openingFailedSearches} now bounds all discarded
linear optimization calls by the order in
\eqref{eq:opening:complexity}.
This proves the bound \eqref{eq:opening:complexity}.
Finally, convexity gives $h_{T+1}\leq g_{T+1}\leq\epsilon$,
establishing the claimed certificate.
\end{proof}

\section{Numerical Demonstrations}\label{sec:numerics}

\subsection{Projection onto the unit cube}
We first provide a simple numerical illustration of Algorithm \ref{alg:DICG}.  The goal of the experiment is not to provide an extensive empirical study, but rather to compare the behavior of this dyadic pairwise
decomposition-invariant method with other conditional gradient methods that use exact line-search (and hence are parameter-free).

We consider the Euclidean projection problem
\begin{align}
    \min_{\x\in[0,1]^n} f(\x)
    :=
    \frac{1}{2}\|\x-\x^0\|^2.
    \label{eq:cube-projection-experiment}
\end{align}
Throughout this experiment the target point $\x^0$ is feasible, and hence
$f^*=0$ and the unique optimal solution is $\x^0$. 
The problem is therefore a particularly transparent $1$- strongly convex and $1$-smooth problem over a $2$-level polytope.

In order to match the representation in \eqref{eq:polyStruct} we consider the lifted formulation
\[
    \mP_{\rm cube}
    =
    \left\{
        (\x,\s)\in\reals^{n}\times\reals^n
        ~\middle|~
        \x+\s=\mathbf{1}_n,\ \x\geq 0,\ \s\geq 0
    \right\}.
\]

The instances are generated as follows. We fix integers $n=1000$ and $k\leq n$.
First, a binary vector in $\{0,1\}^n$ is sampled uniformly at random. Then $k$
coordinates are chosen uniformly without replacement and are replaced by values
drawn independently and uniformly from $[\mu,1-\mu]$ for $\mu=0.001$. Thus, $\x^0$ has exactly $k$
coordinates strictly between $0$ and $1$, while all remaining coordinates are
binary. Consequently, the minimal face of the cube containing the optimal solution
has dimension exactly $d^*=k$. Each run is initialized from an independently sampled
uniform random vertex of the cube.

We compare our Algorithm \ref{alg:DICG} (without any modification) with the baselines listed in Table \ref{table:algorithms}.  Recall the hypercube is a product polytope and thus, per the discussion following Lemma \ref{lem:productMassTransfer}, we can replace the sparsity parameter $d^*$ with a universal constant and the worst-case complexity of our Algorithm \ref{alg:DICG} becomes $O\left({\frac{\beta{}D^2}{\alpha}\log(1/\epsilon)}\right)$, and hence independent of $k$.

\begin{table*}\renewcommand{\arraystretch}{1.3}
{\small
\begin{center}
  \begin{tabular}{|p{0.2\linewidth} | p{0.7 \linewidth}|} \hline
   algorithm &description and comments \\ \hline
   DI-Pairwise + ls &  decomposition-invariant CG with pairwise steps and exact line-search (variant PFW-2 in \cite{bashiri2017decomposition}). No convergence guarantee.\\ \hline
   DI-AFW + ls &  decomposition-invariant CG  with away steps and exact line-search (variant AFW-2 in \cite{bashiri2017decomposition}). $O\left({n(d^*+1)\frac{\beta{}D^2}{\alpha}\log(1/\epsilon)}\right)$ iteration complexity.\\ \hline
   Standard AFW &  the standard (decomposition-dependent) Frank-Wolfe with away steps and exact line-search of \cite{lacoste2015global}. $O\left({n\frac{\beta{}D^2}{\alpha}\log(1/\epsilon)}\right)$ iteration complexity. \\ \hline
  \end{tabular}
\caption{Description of baselines used in numerical experiments.}
  \label{table:algorithms}
\end{center}
}
\end{table*}\renewcommand{\arraystretch}{1} 

The results are given in Figure \ref{fig:cube}. Each plot is the average of 10 i.i.d. runs (both $\x^0$ and the initialization vertex resampled). We can clearly see that with the exception of the case $k=n=1000$, our Algorithm \ref{alg:DICG} indeed seems to be unaffected by the dimension of the optimal face $k$, and exhibits nearly identical convergence regardless of the value of $k$. We can also clearly see that, with the exception of the extreme cases $k=5$ and $k=n=1000$, Algorithm \ref{alg:DICG} significantly outperforms all baselines.

\begin{figure}[H]
     \centering
     \begin{subfigure}[b]{0.32\textwidth}
         \centering
         \includegraphics[width=\textwidth]{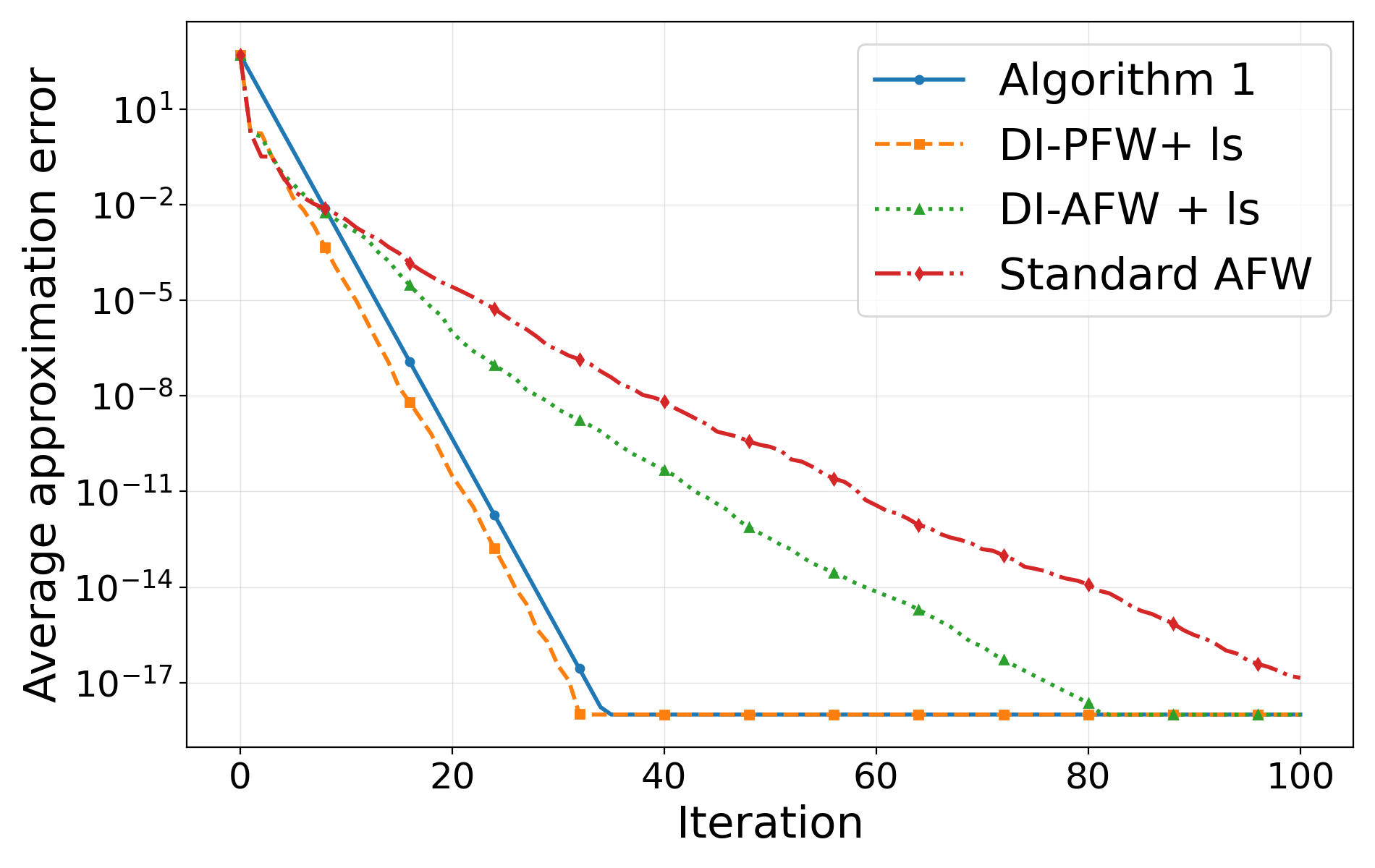}
         \caption*{$k=5$}
         \label{fig:y equals x}
     \end{subfigure}
     \hfill
     \begin{subfigure}[b]{0.32\textwidth}
         \centering
         \includegraphics[width=\textwidth]{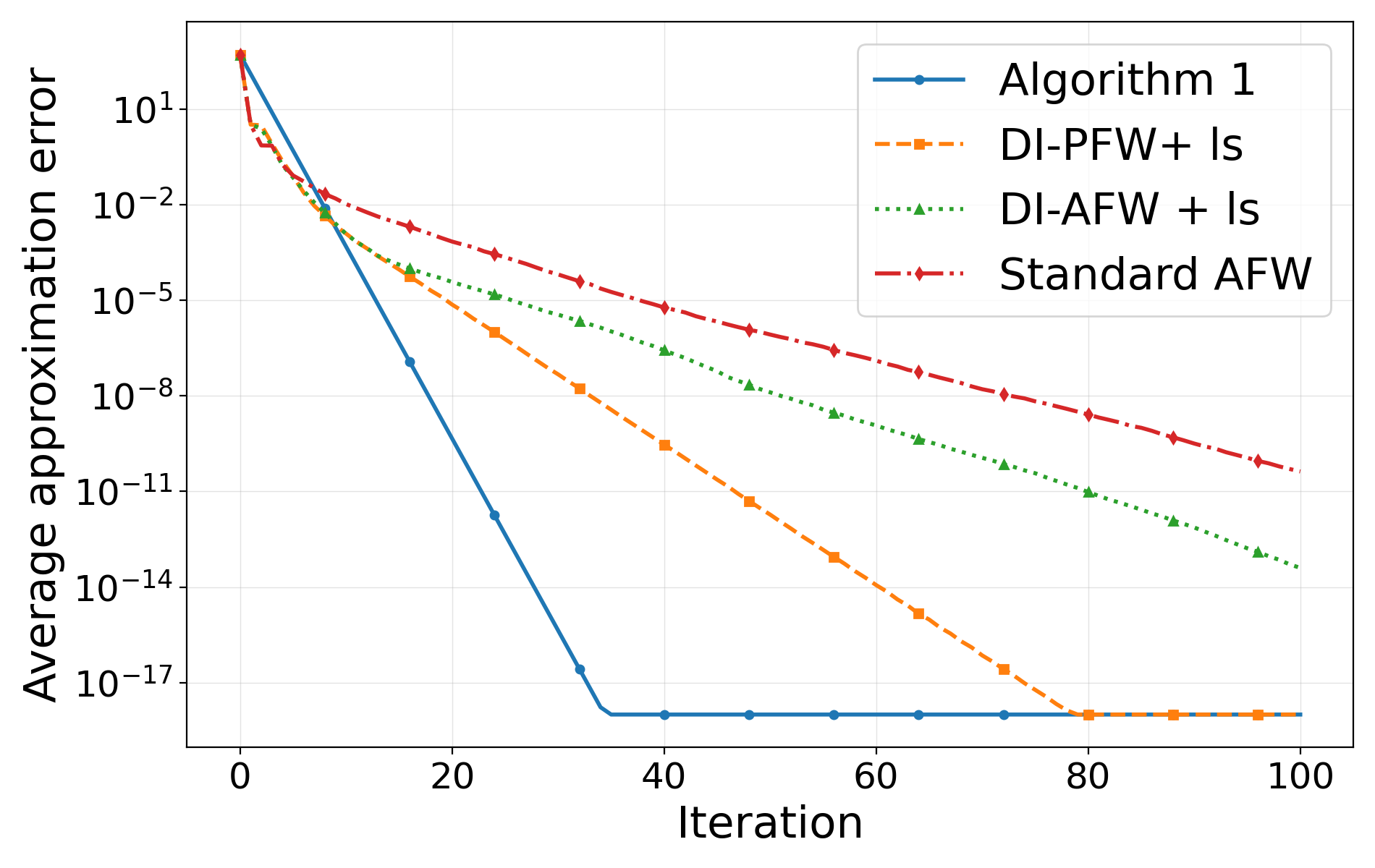}
         \caption*{$k=10$}
         \label{fig:three sin x}
     \end{subfigure}
     \hfill
     \begin{subfigure}[b]{0.32\textwidth}
         \centering
         \includegraphics[width=\textwidth]{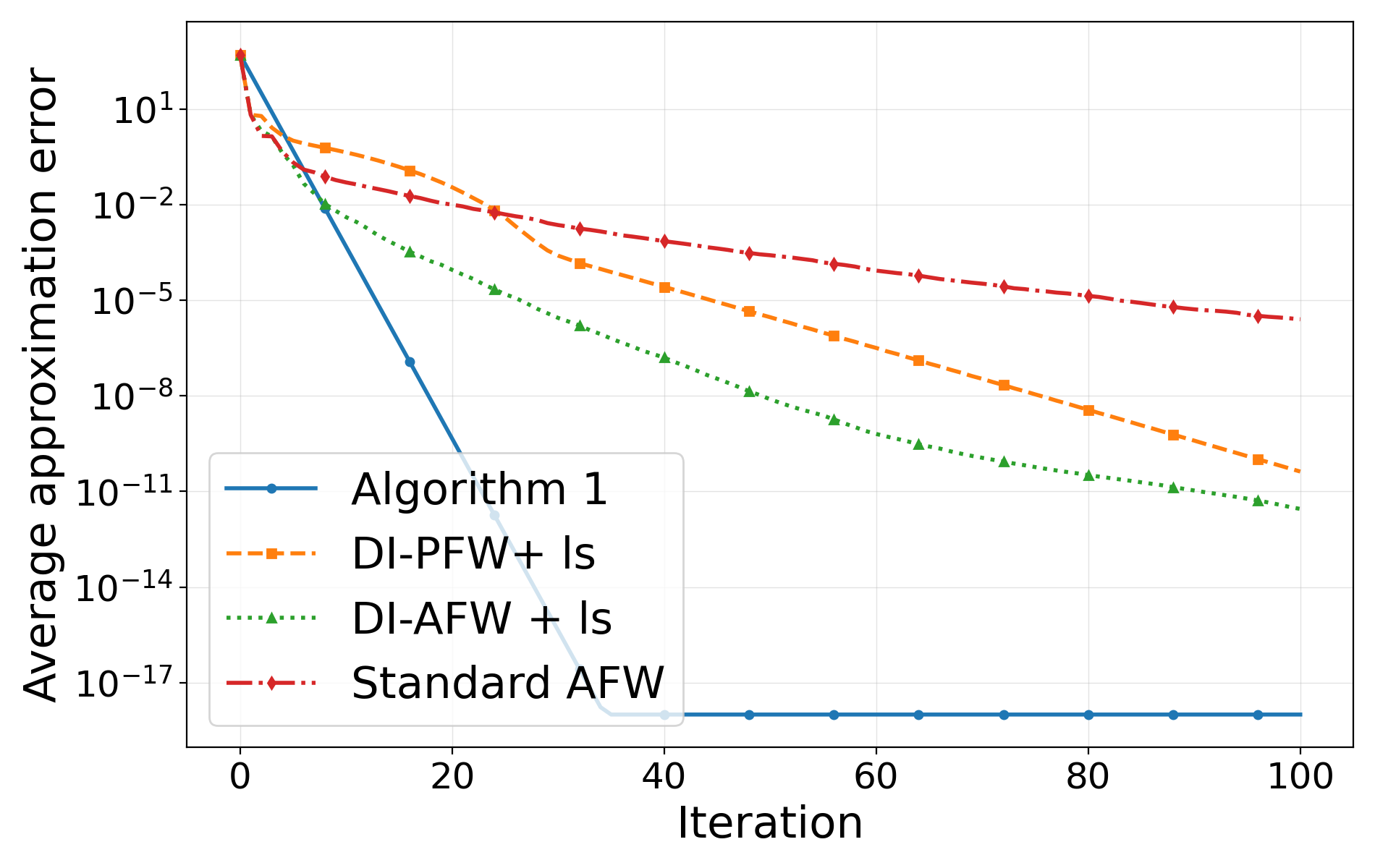}
         \caption*{$k=20$}
         \label{fig:five over x}
     \end{subfigure}

     \begin{subfigure}[b]{0.32\textwidth}
         \centering
         \includegraphics[width=\textwidth]{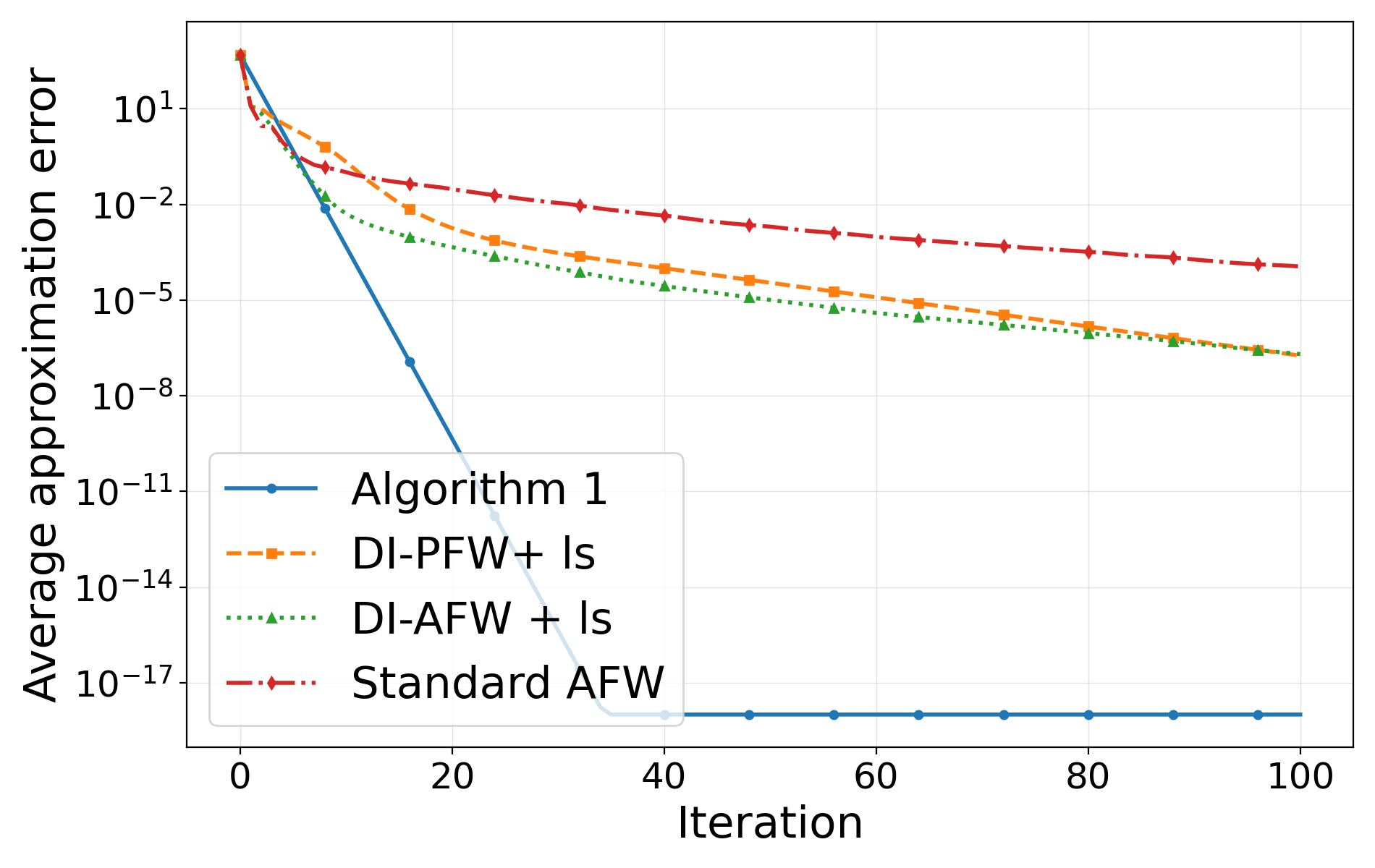}
         \caption*{$k=40$}
         \label{fig:y equals x}
     \end{subfigure}
     \hfill
     \begin{subfigure}[b]{0.32\textwidth}
         \centering
         \includegraphics[width=\textwidth]{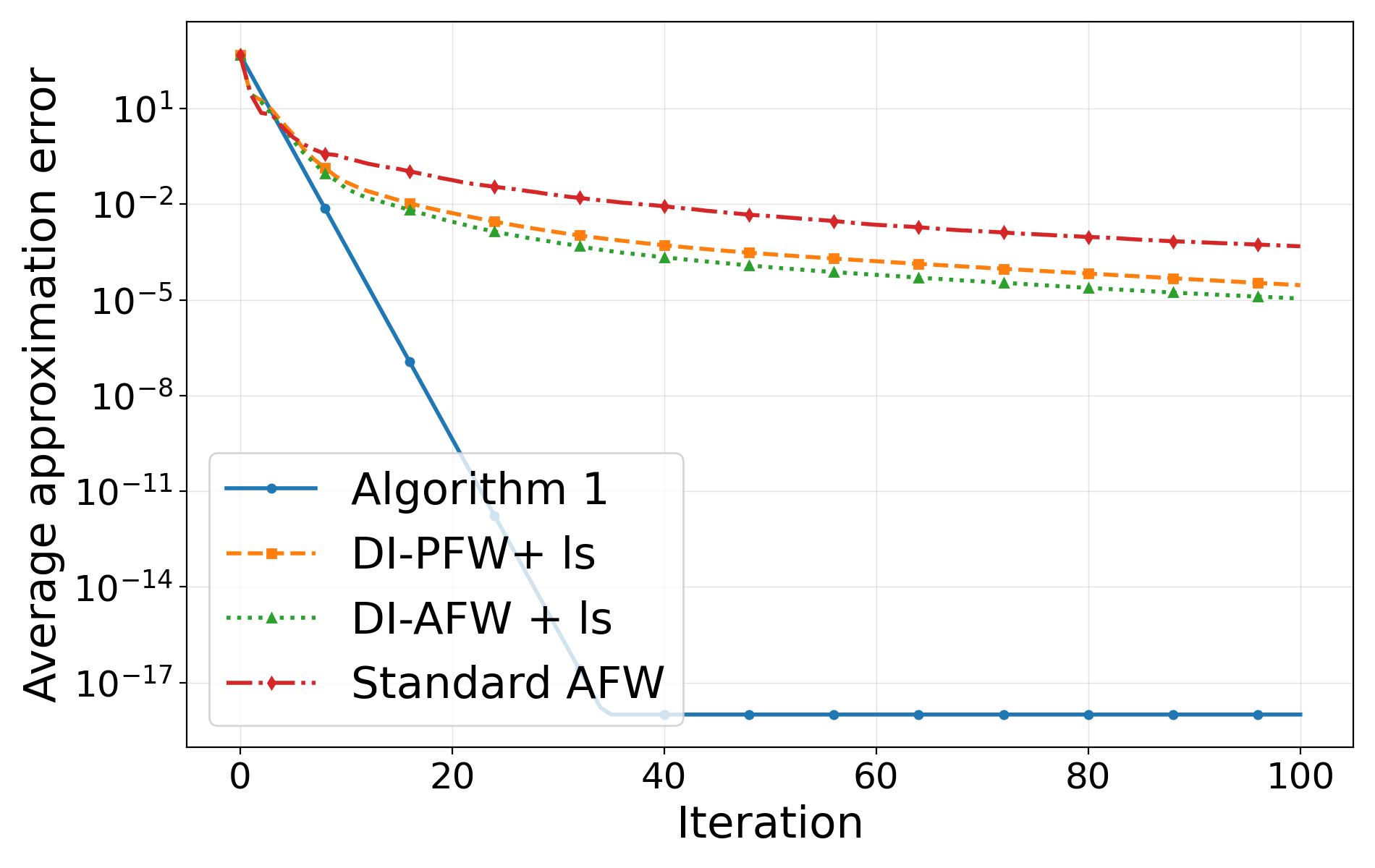}
         \caption*{$k=100$}
         \label{fig:three sin x}
     \end{subfigure}
     \hfill
     \begin{subfigure}[b]{0.32\textwidth}
         \centering
         \includegraphics[width=\textwidth]{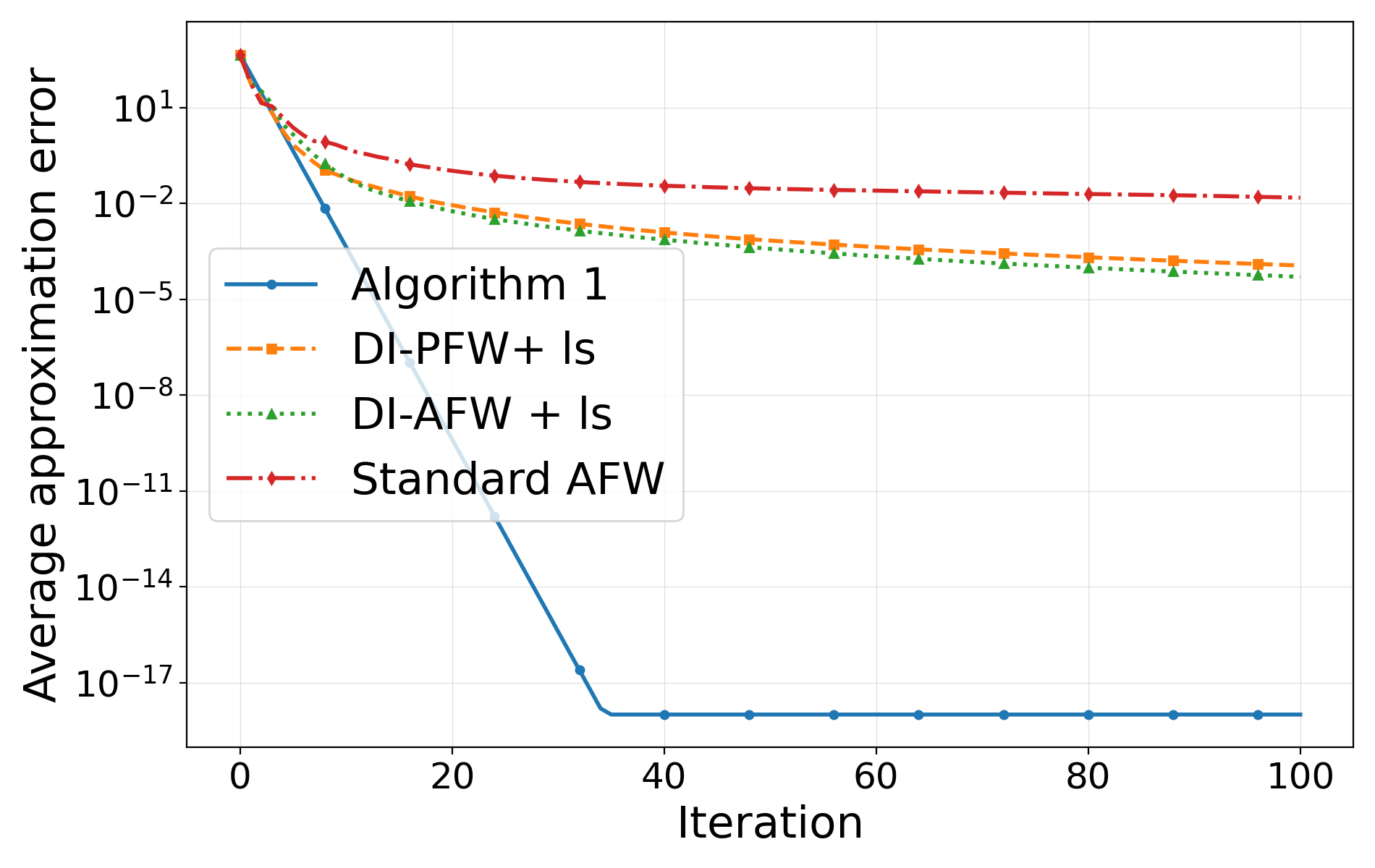}
         \caption*{$k=250$}
         \label{fig:five over x}
     \end{subfigure}

     \begin{subfigure}[b]{0.32\textwidth}
         \centering
         \includegraphics[width=\textwidth]{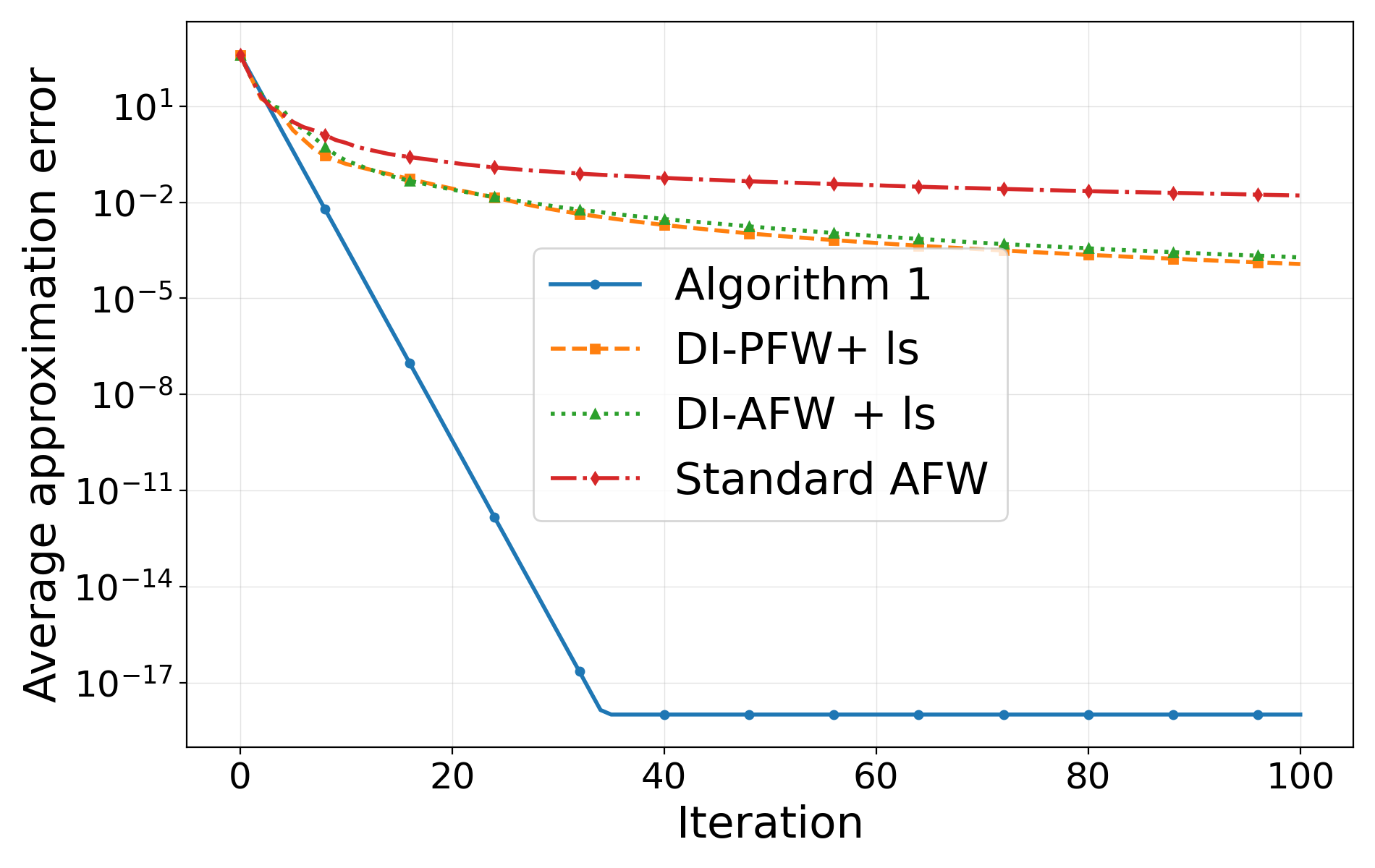}
         \caption*{$k=500$}
         \label{fig:y equals x}
     \end{subfigure}
     \hfill
     \begin{subfigure}[b]{0.32\textwidth}
         \centering
         \includegraphics[width=\textwidth]{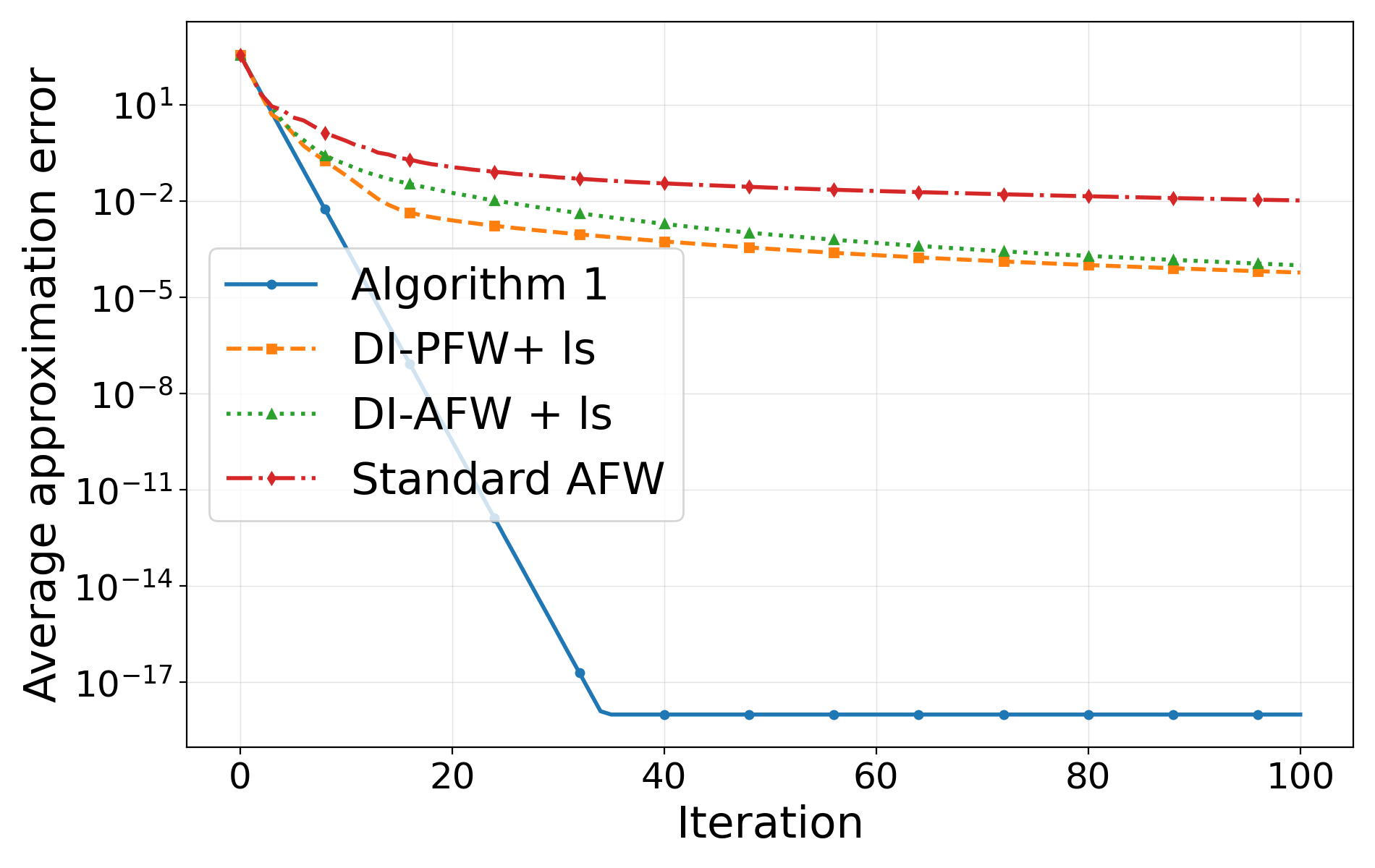}
         \caption*{$k=750$}
         \label{fig:three sin x}
     \end{subfigure}
     \hfill
     \begin{subfigure}[b]{0.32\textwidth}
         \centering
         \includegraphics[width=\textwidth]{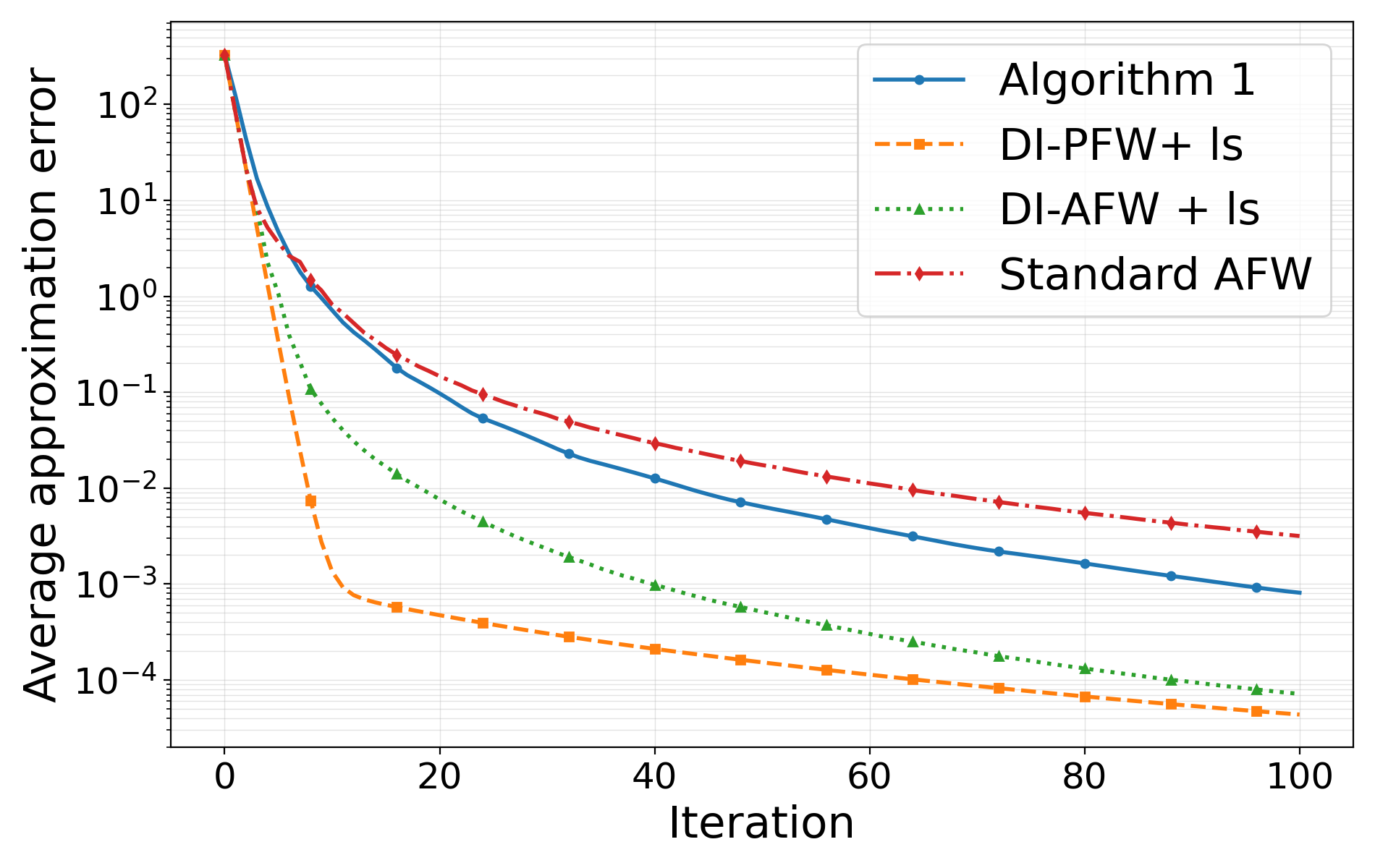}
         \caption*{$k=1000$}
         \label{fig:five over x}
     \end{subfigure}
        \caption{Approximation errors vs. number of iterations for Euclidean projection onto the unit cube.}
        \label{fig:cube}
\end{figure}

\subsection{Projection onto a fractional capped simplex}
\label{subsec:numerics-capped-simplex}

We turn to provide a numerical illustration of Algorithm
\ref{alg:openingDICG} and compare its performance to the baselines listed in
Table \ref{table:algorithms}. We consider a family of projection instances for
which the feasible region is not 2-level.

Throughout this experiment we assume that $n\geq4$ is even. We consider the
fractional capped simplex
\begin{align}
\mP_{\rm cap}
:=\left\{\x\in[0,1]^n~\middle|~
\sum_{i=1}^n\x(i)=\frac{n}{2}+\frac12\right\}.
\label{eq:capped-simplex-experiment}
\end{align}
Every vertex is a permutation of a vector with $n/2$ entries equal to one,
one entry equal to $1/2$, and $n/2-1$ entries equal to zero. In particular,
the vertices take the three levels
$0,1/2,1$, and hence the polytope is not 2-level. The single sum constraint
couples all coordinates, and the dimension of the polytope is $n-1$.

The optimization problem is again a Euclidean projection problem:
\begin{align}
\min_{\x\in\mP_{\rm cap}} f(\x)
:=\frac12\|\x-\x^0\|^2,
\label{eq:capped-projection-experiment}
\end{align}
where the target point $\x^0$ is feasible. Thus, $f^*=0$, the unique optimal
solution is $\x^0$, and the objective is $1$-strongly convex and $1$-smooth
in the original variables. Moreover, $\nabla f(\x^0)=0$. 

In order to match the representation \eqref{eq:polyStruct}, and the
normalization $\xi=1$, we use the lifted formulation
\begin{align}
\widehat{\mP}_{\rm cap}
:=\left\{(\y,\s)\in\reals^n\times\reals^n~\middle|~
\y+\s=2\mathbf{1}_n,\quad
\sum_{i=1}^n\y(i)=n+1,\quad \y,\s\geq0\right\}.
\label{eq:capped-simplex-lift}
\end{align}
Here $\y=2\x$ and $\s=2(\mathbf{1}_n-\x)$. Every positive coordinate
of a lifted vertex is either one or two, and hence the minimum positive
vertex slack is exactly one. All methods are implemented in these $2n$
working coordinates. We recover the original point by
$\x=(\y-\s+2\mathbf{1}_n)/4$, which equals $\y/2$ on the feasible region.
For a lifted direction $\d=(\d_y,\d_s)$, the corresponding original
direction is $\d_x=(\d_y-\d_s)/4$.

The target point $\x^0$ is sampled so that $k$ is exactly the dimension of its
minimal face. We fix an integer $0\leq k\leq n-1$, and set $r=k+1$, $b_r=\lfloor r/2\rfloor+1/2$.
First, $r$ coordinates are chosen uniformly without replacement; denote
their set by $S$. Among the remaining coordinates, we choose
$n/2-\lfloor r/2\rfloor$ coordinates uniformly without replacement and set
them to one, while all other coordinates outside $S$ are set to zero.
For $i\in S$, we sample independent random variables
$u_i\sim\mathrm{Unif}(-\mu/2,\mu/2)$, and set
\begin{align}
    \x^0(i)=\frac{b_r}{r}+u_i-\frac1r\sum_{j\in S}u_j,
    \qquad i\in S.
\label{eq:capped-target-experiment}
\end{align}
We use $\mu=0.1$. The coordinates in $S$ are strictly between zero and one and sum to $b_r$.
Consequently, $\x^0\in\mP_{\rm cap}$, and its minimal face is obtained by
fixing the coordinates outside $S$ and retaining one sum constraint on
the $r$ free coordinates. Its dimension is therefore $r-1=k$, and
indeed $d^*=k$. For $k\geq2$, this optimal face is itself a fractional
capped simplex with vertex coordinate levels $0,1/2,1$, and is not
2-level.

The squared Euclidean diameter
of $\mP_{\rm cap}$ is $n-3/2$, and that of its lift is $8n-12$.

Linear optimization over this polytope has a simple implementation.
For a lifted cost $\c=(\c_y,\c_s)$, we have
\[
\langle\c,(\y,\s)\rangle
=2\langle\c_y-\c_s,\x\rangle+2\langle\c_s,\mathbf{1}_n\rangle.
\]
The global oracle therefore sorts the entries of $\c_y-\c_s$ in
nondecreasing order, assigns one to the first $n/2$ original coordinates,
$1/2$ to the next coordinate, and zero to the remaining coordinates,
and returns the corresponding lift.

We implemented our Algorithm \ref{alg:openingDICG}, together with the adaptive
search in Algorithm \ref{alg:openingSearch}, exactly as detailed in Section \ref{sec:genPoly}.
 All three
baselines in Table \ref{table:algorithms} are used with the same objective,
linear optimization routines, and exact line-search.

In all experiments we set  $n=1000$ while varying the sparsity parameter $k$ and taking the average of 10 i.i.d. runs for each value of $k$.
Each random repetition samples a new target $\x^0$  and an independent uniform random initialization
vertex. 

The results are given in Figure \ref{fig:capped_simplex_it} which reports the average approximation error vs number of iterations and Figure  \ref{fig:capped_simplex_oracle} which reports the average approximation error vs number of calls to the linear optimization oracle. We can see that, with the exception of the two most dense regimes $k=750, k=990$, our Algorithm \ref{alg:openingDICG} reaches small approximation errors in significantly fewer iterations or oracle calls than the competing methods.

\begin{figure}[H]
     \centering
     \begin{subfigure}[b]{0.32\textwidth}
         \centering
         \includegraphics[width=\textwidth]{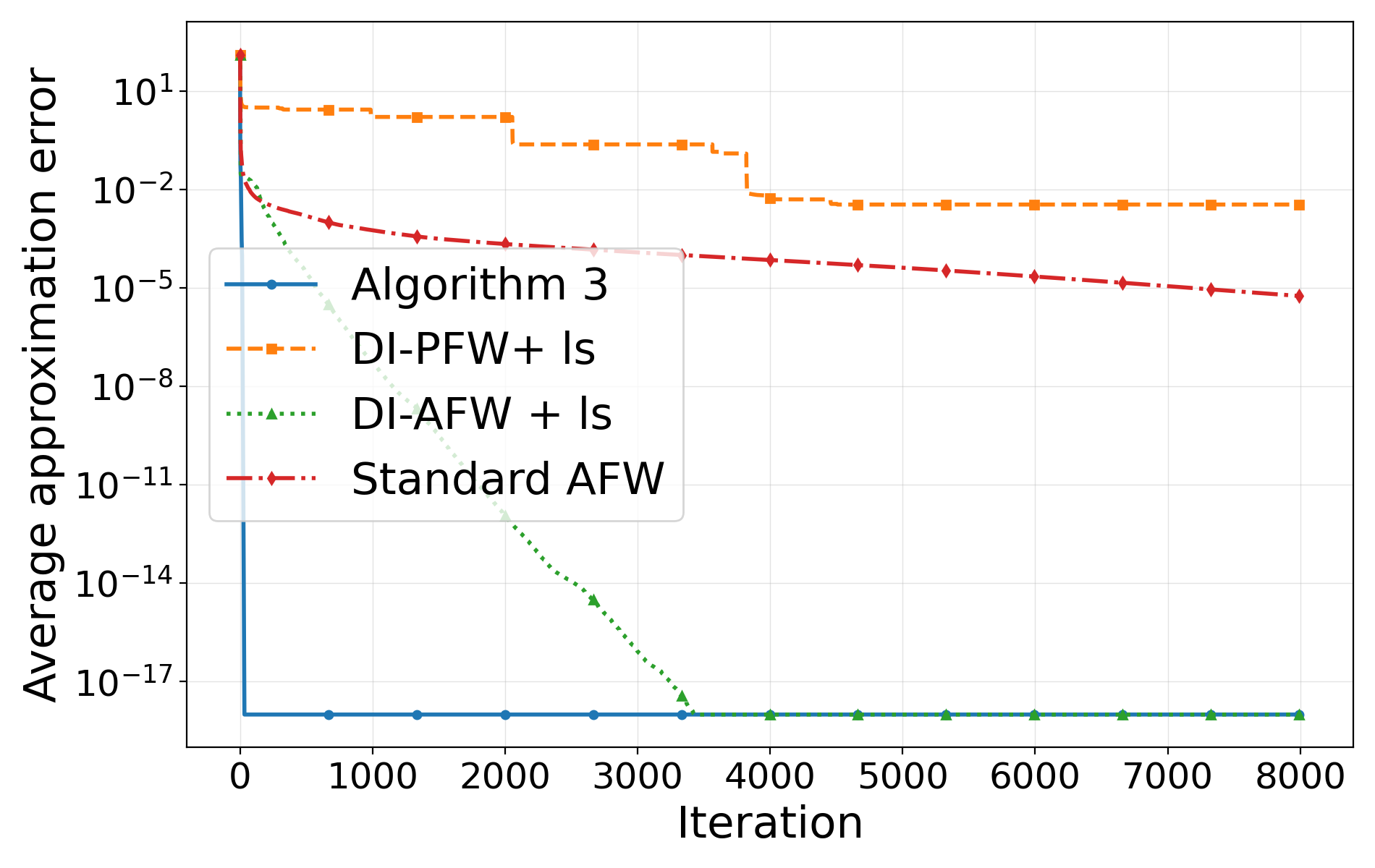}
         \caption*{$k=5$}
         \label{fig:y equals x}
     \end{subfigure}
     \hfill
     \begin{subfigure}[b]{0.32\textwidth}
         \centering
         \includegraphics[width=\textwidth]{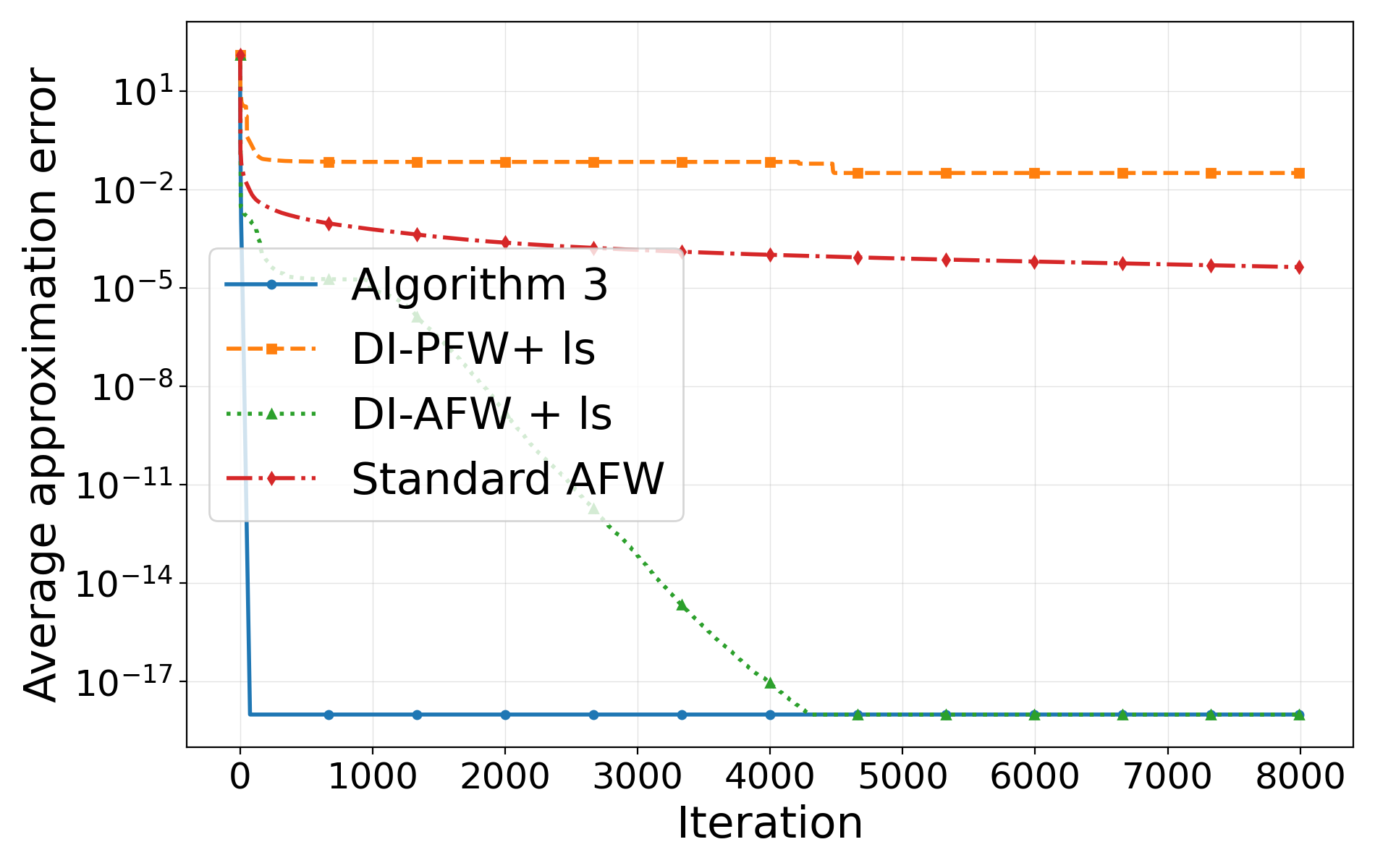}
         \caption*{$k=10$}
         \label{fig:three sin x}
     \end{subfigure}
     \hfill
     \begin{subfigure}[b]{0.32\textwidth}
         \centering
         \includegraphics[width=\textwidth]{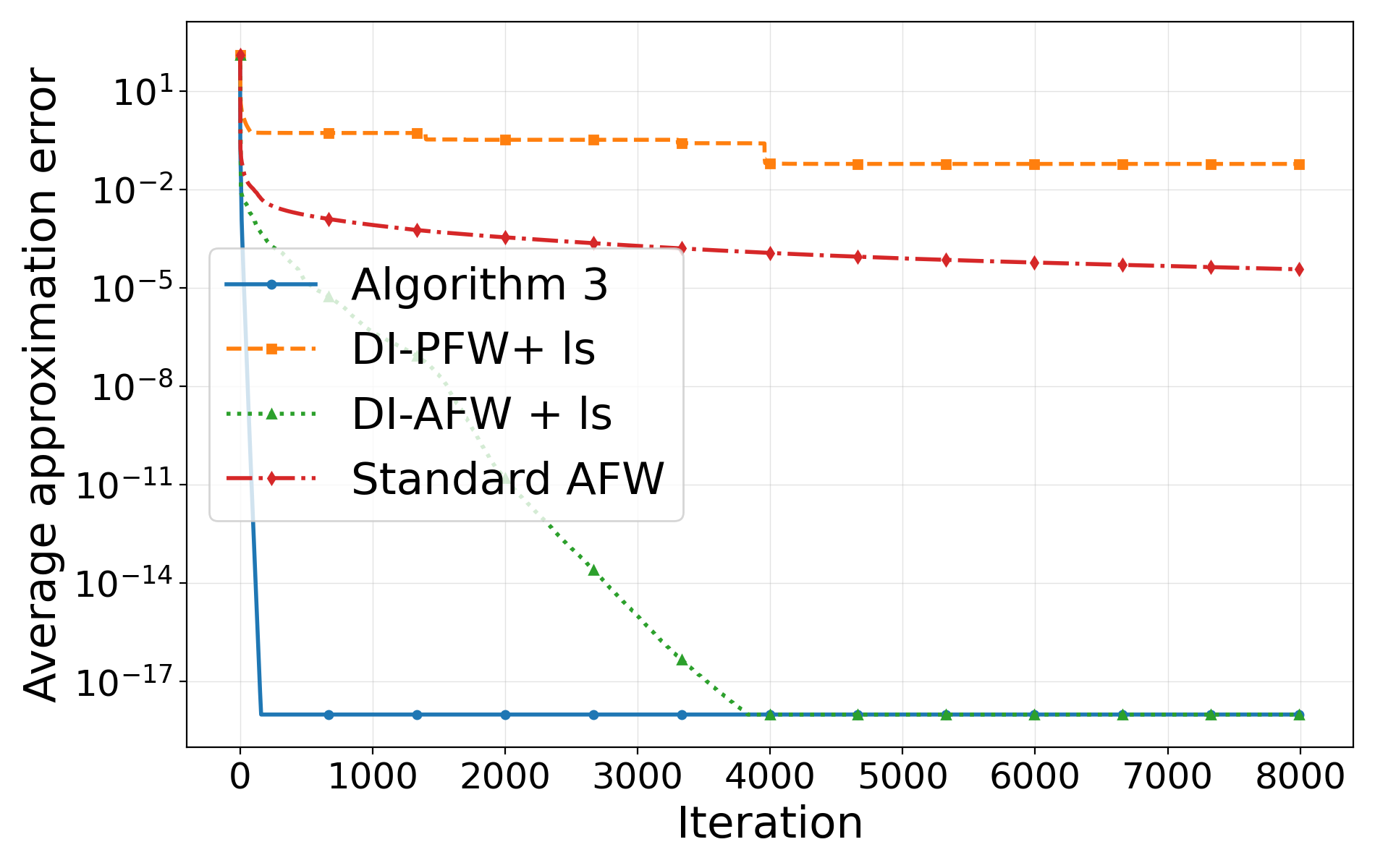}
         \caption*{$k=20$}
         \label{fig:five over x}
     \end{subfigure}

     \begin{subfigure}[b]{0.32\textwidth}
         \centering
         \includegraphics[width=\textwidth]{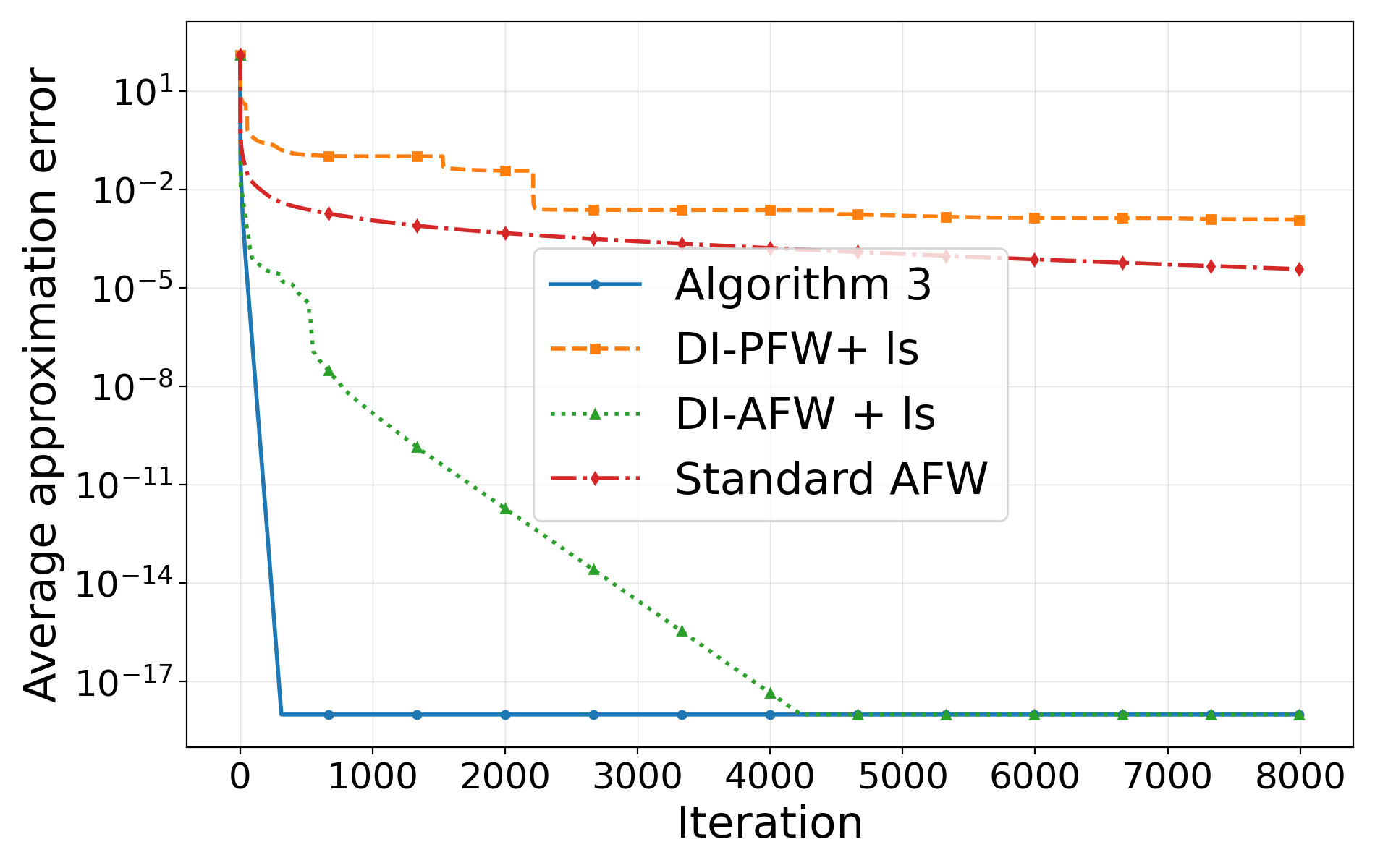}
         \caption*{$k=40$}
         \label{fig:y equals x}
     \end{subfigure}
     \hfill
     \begin{subfigure}[b]{0.32\textwidth}
         \centering
         \includegraphics[width=\textwidth]{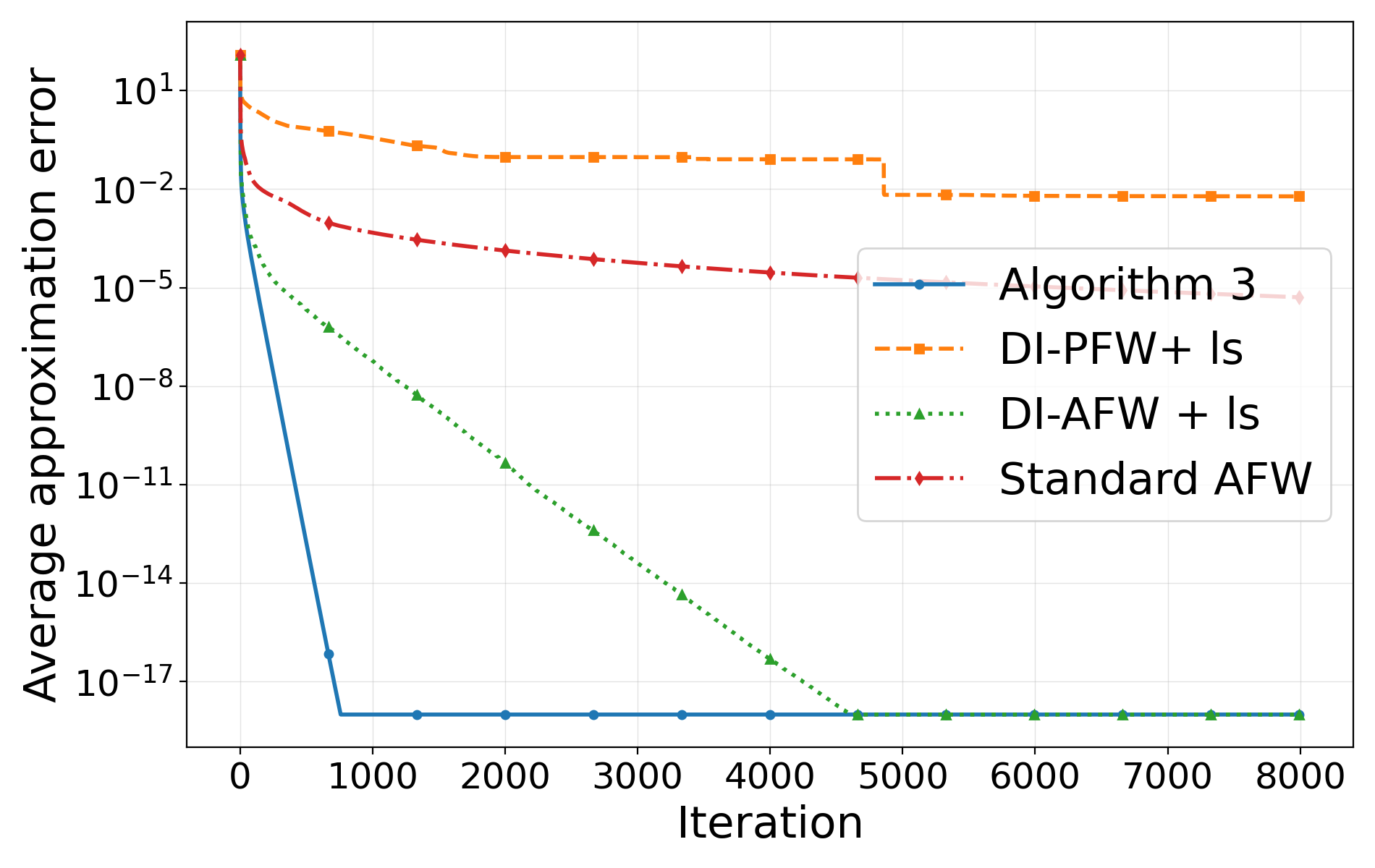}
         \caption*{$k=100$}
         \label{fig:three sin x}
     \end{subfigure}
     \hfill
     \begin{subfigure}[b]{0.32\textwidth}
         \centering
         \includegraphics[width=\textwidth]{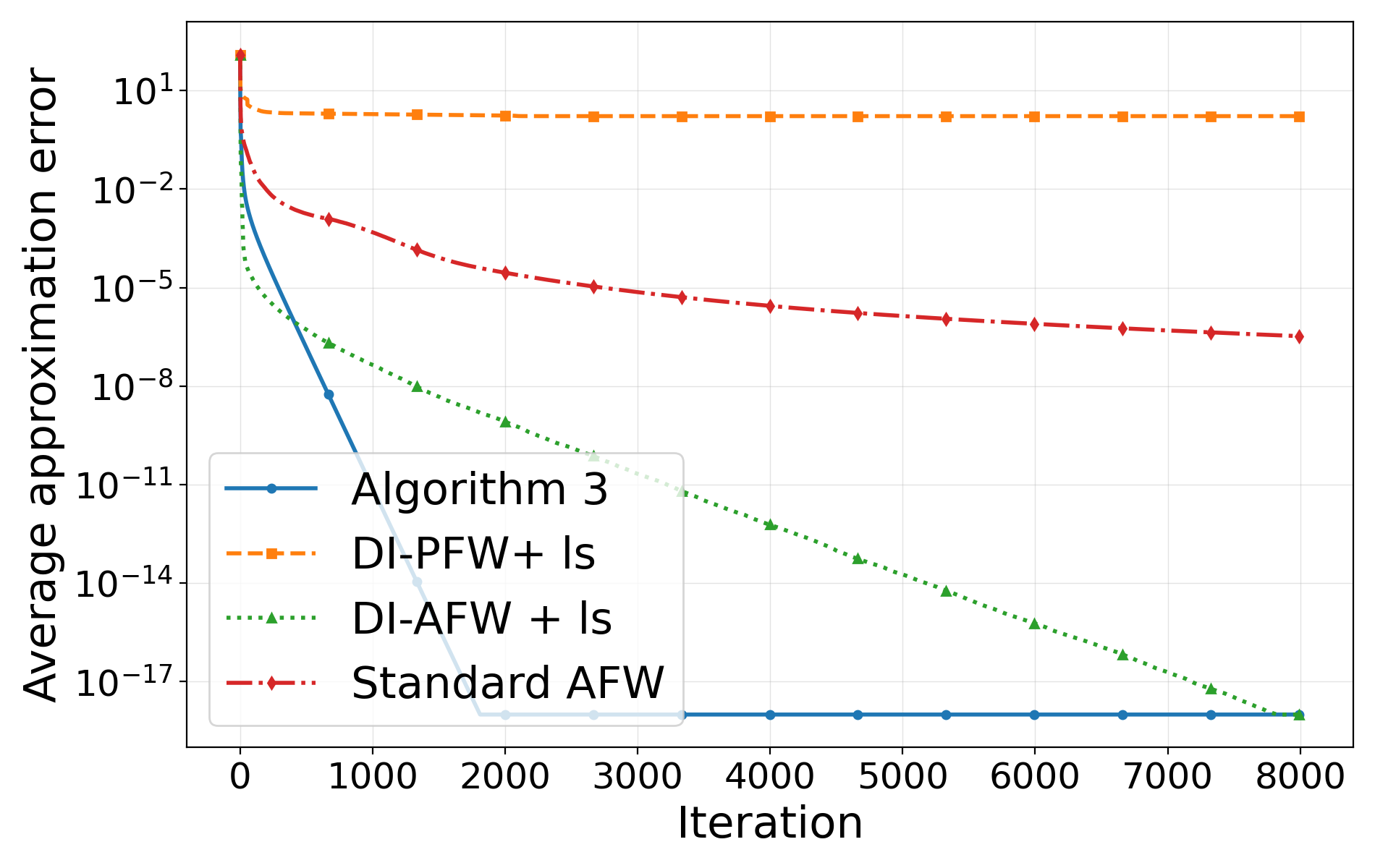}
         \caption*{$k=250$}
         \label{fig:five over x}
     \end{subfigure}

     \begin{subfigure}[b]{0.32\textwidth}
         \centering
         \includegraphics[width=\textwidth]{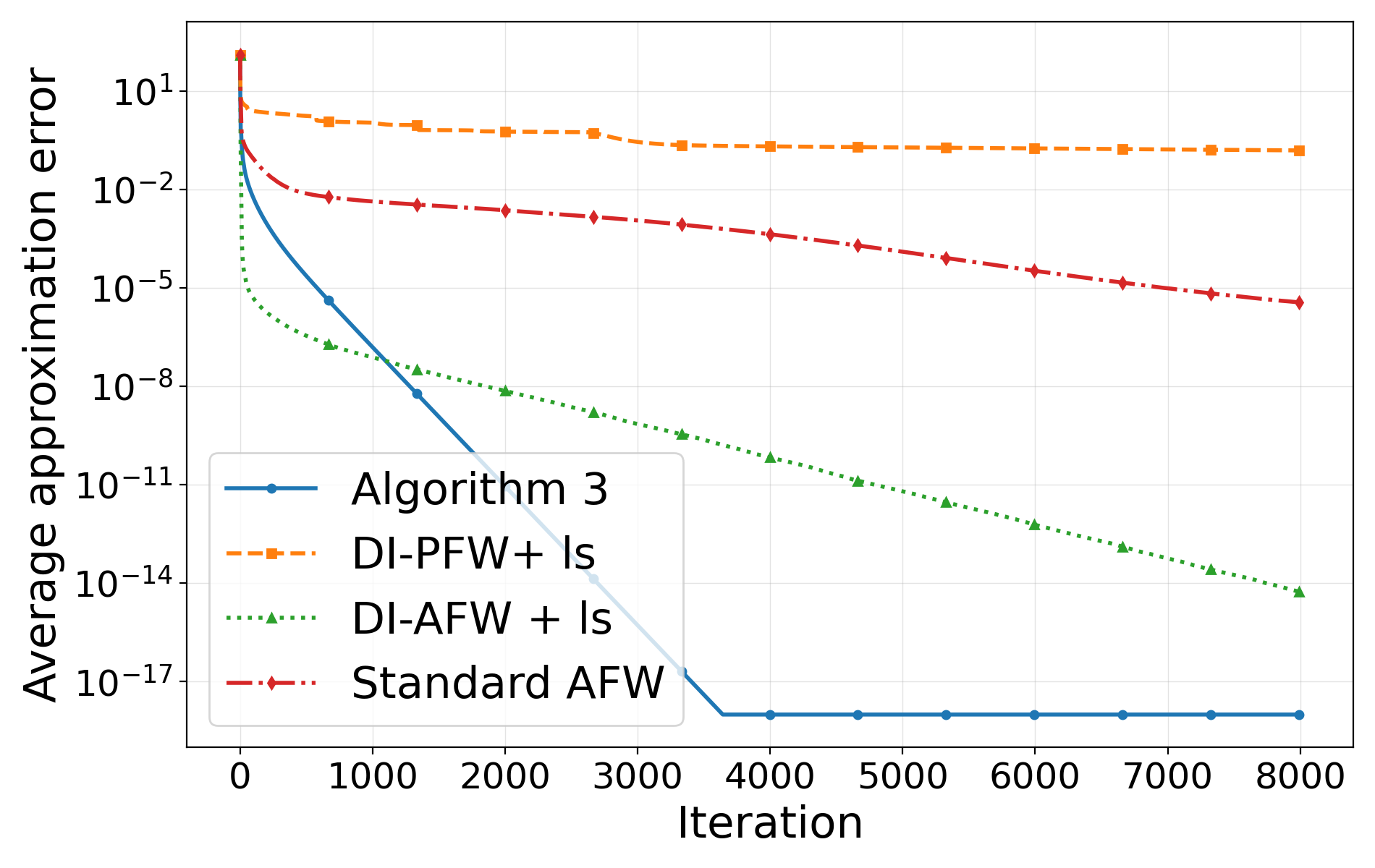}
         \caption*{$k=500$}
         \label{fig:y equals x}
     \end{subfigure}
     \hfill
     \begin{subfigure}[b]{0.32\textwidth}
         \centering
         \includegraphics[width=\textwidth]{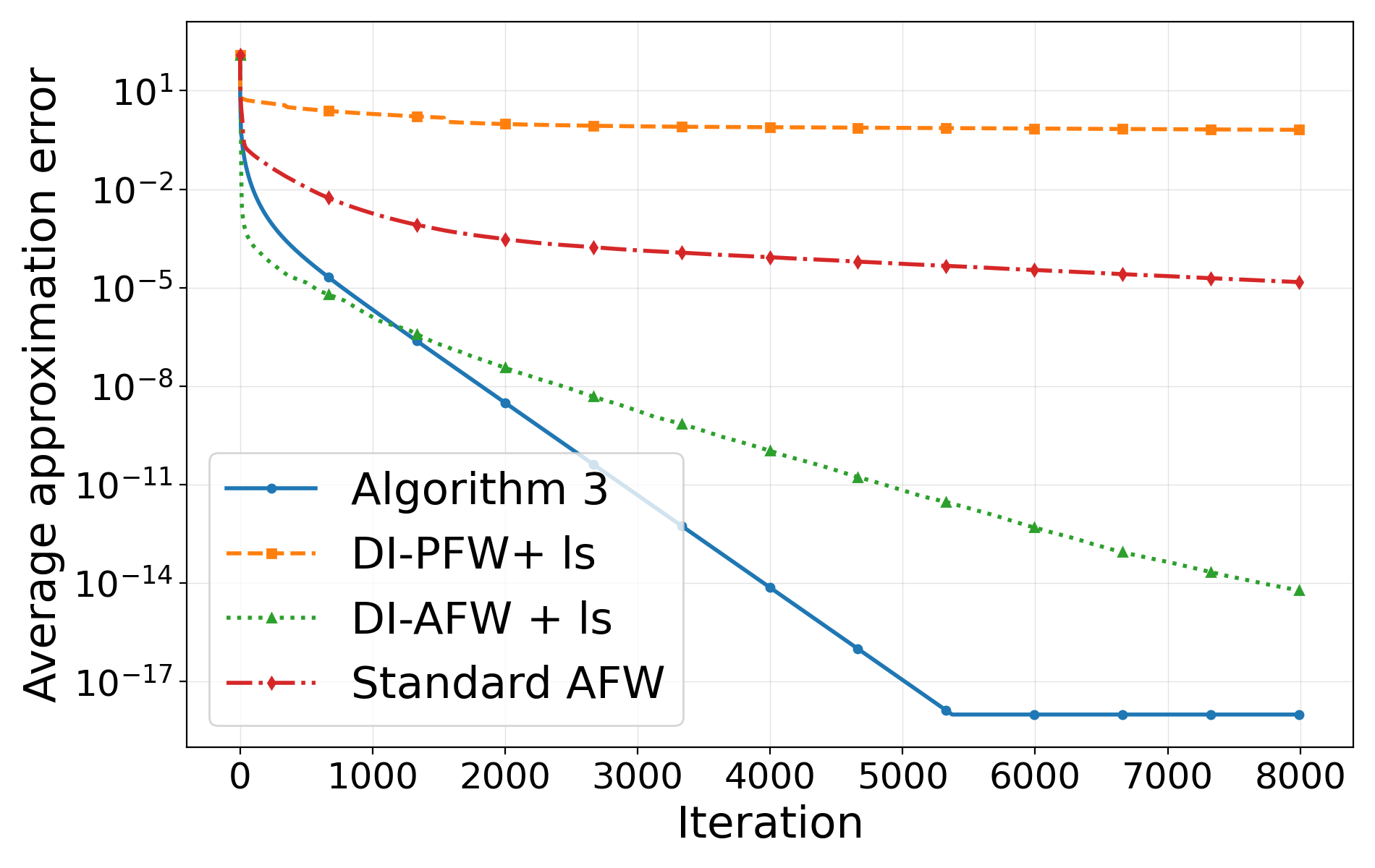}
         \caption*{$k=750$}
         \label{fig:three sin x}
     \end{subfigure}
     \hfill
     \begin{subfigure}[b]{0.32\textwidth}
         \centering
         \includegraphics[width=\textwidth]{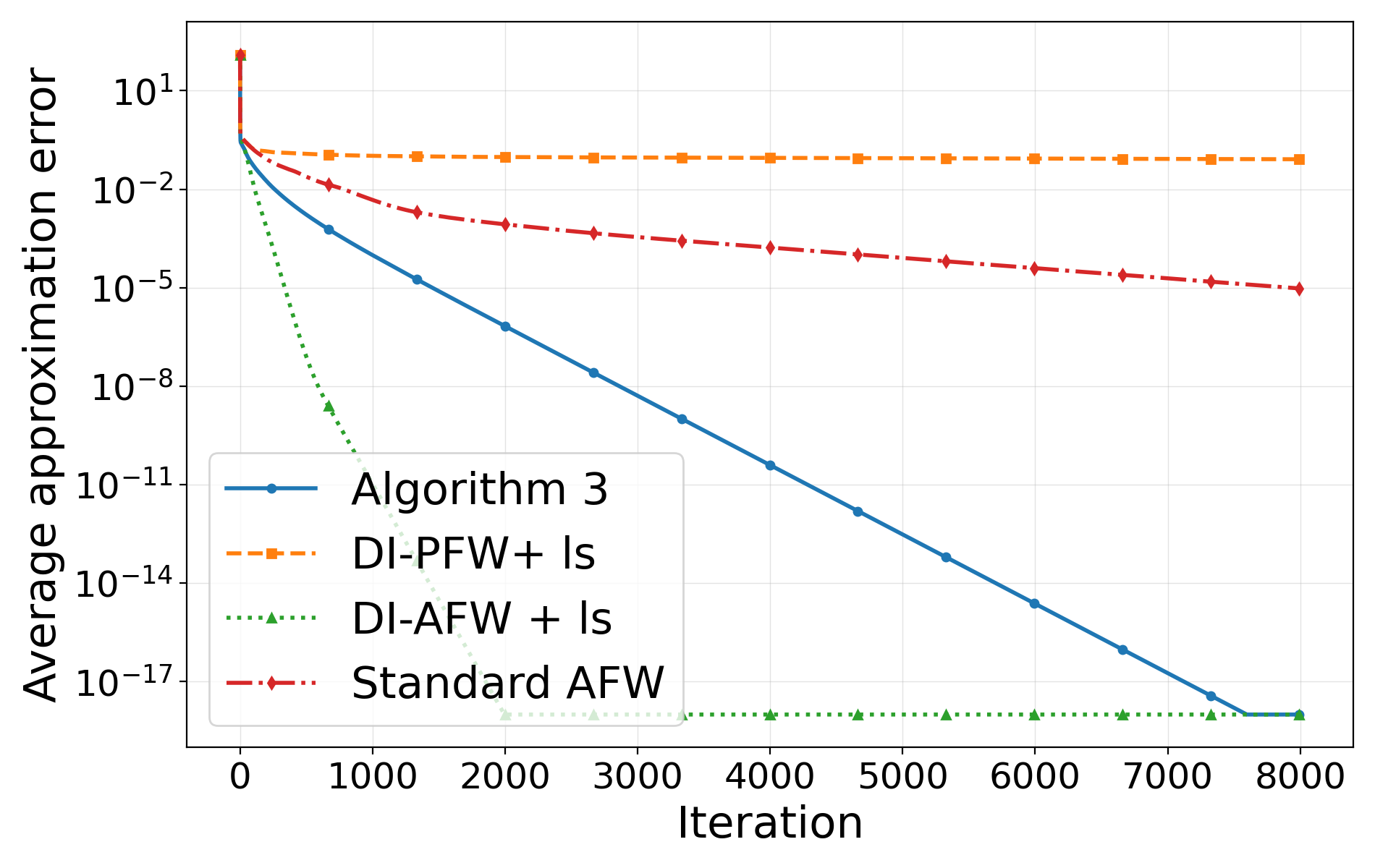}
         \caption*{$k=990$}
         \label{fig:five over x}
     \end{subfigure}
        \caption{Approximation errors vs. number of iterations for Euclidean projection onto a fractional capped simplex.}
        \label{fig:capped_simplex_it}
\end{figure}

\begin{figure}[H]
     \centering
     \begin{subfigure}[b]{0.32\textwidth}
         \centering
         \includegraphics[width=\textwidth]{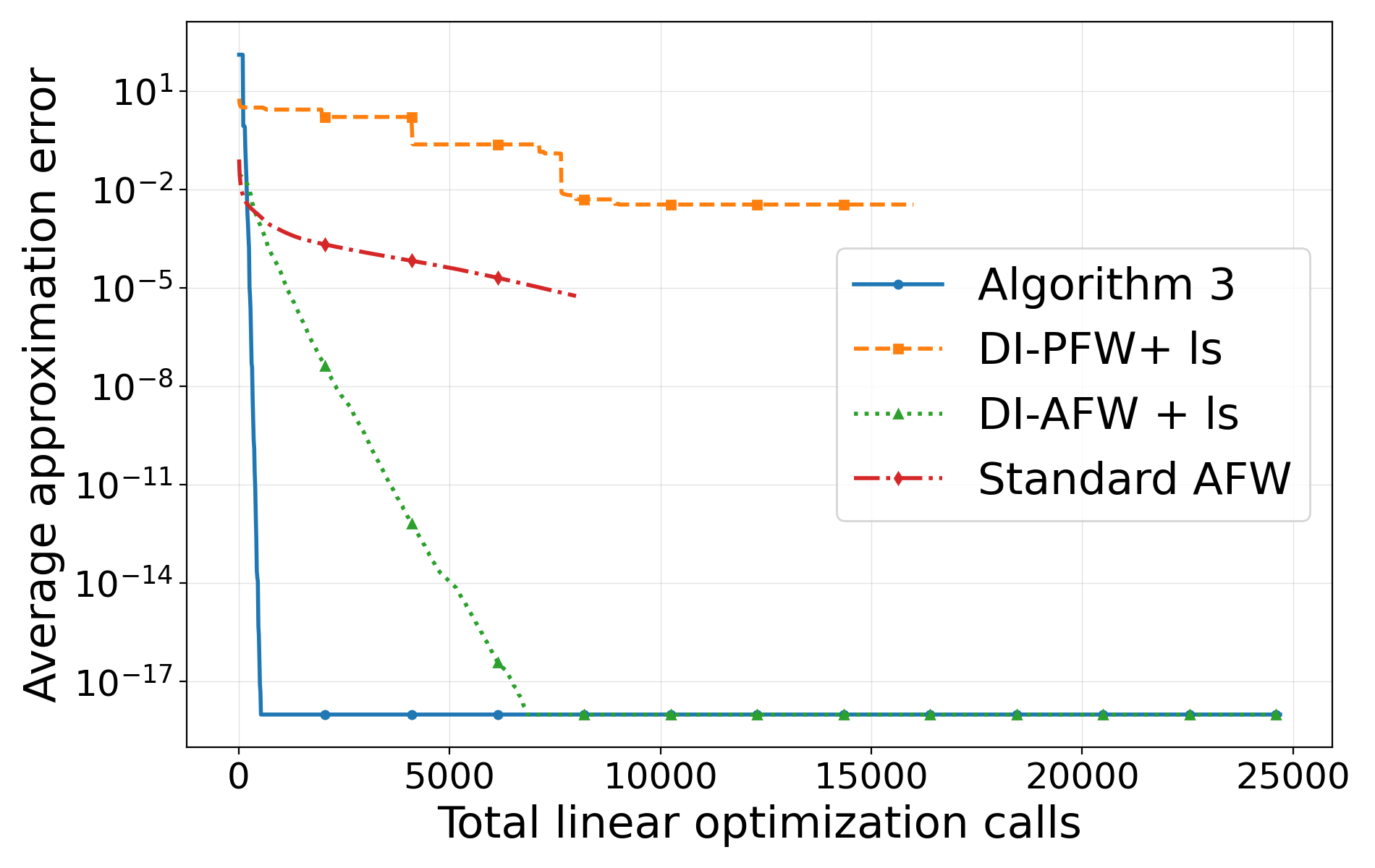}
         \caption*{$k=5$}
         \label{fig:y equals x}
     \end{subfigure}
     \hfill
     \begin{subfigure}[b]{0.32\textwidth}
         \centering
         \includegraphics[width=\textwidth]{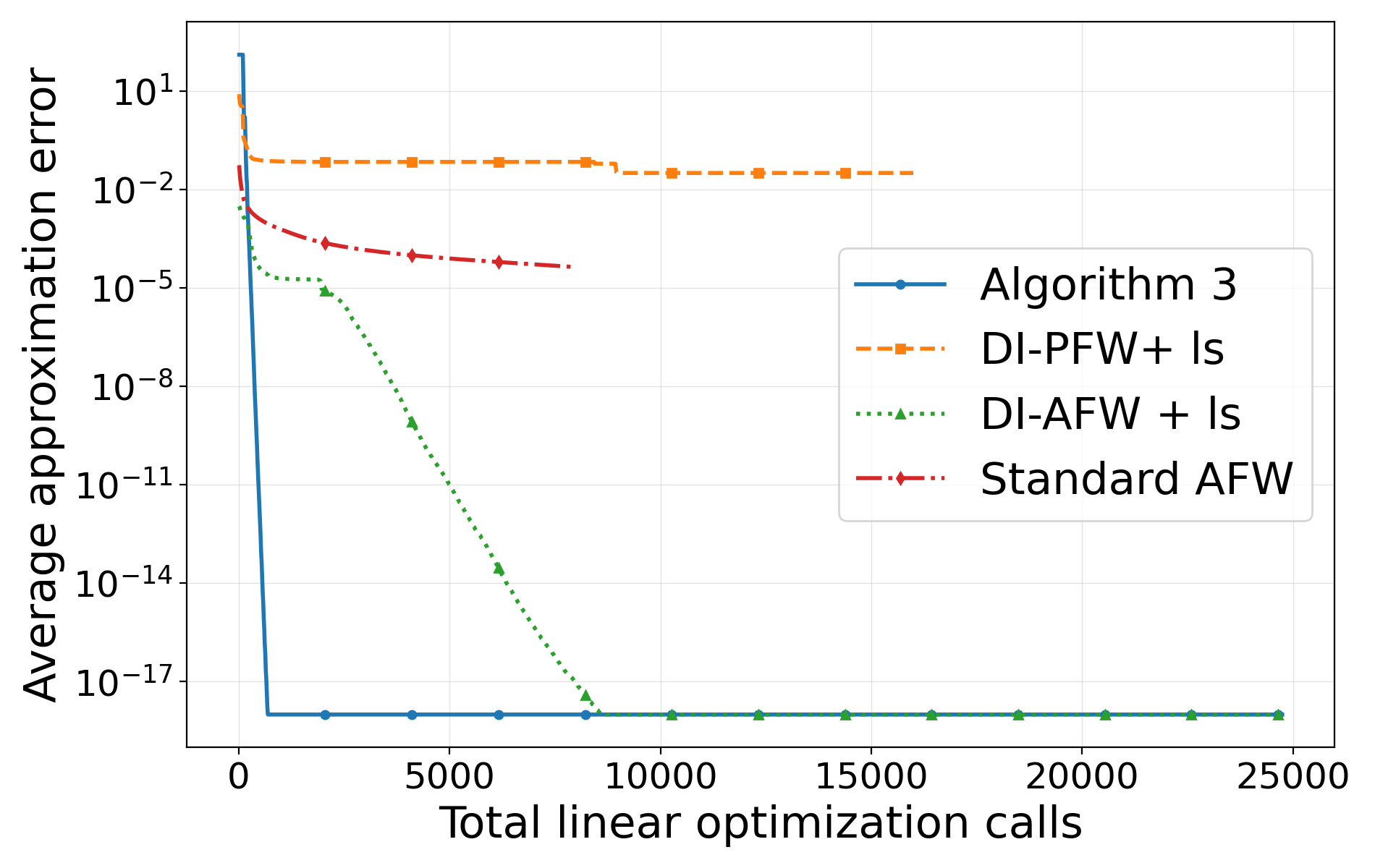}
         \caption*{$k=10$}
         \label{fig:three sin x}
     \end{subfigure}
     \hfill
     \begin{subfigure}[b]{0.32\textwidth}
         \centering
         \includegraphics[width=\textwidth]{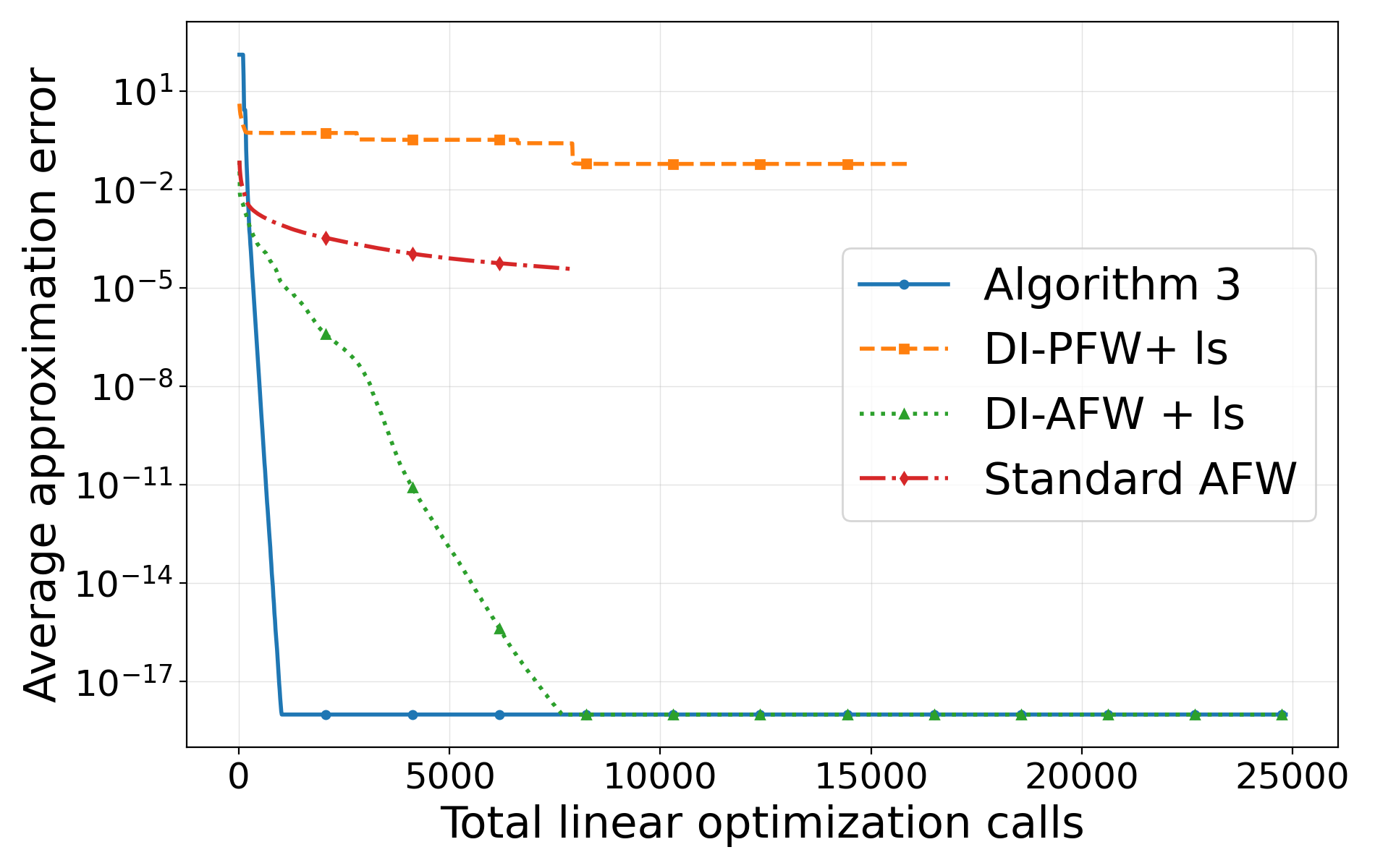}
         \caption*{$k=20$}
         \label{fig:five over x}
     \end{subfigure}

     \begin{subfigure}[b]{0.32\textwidth}
         \centering
         \includegraphics[width=\textwidth]{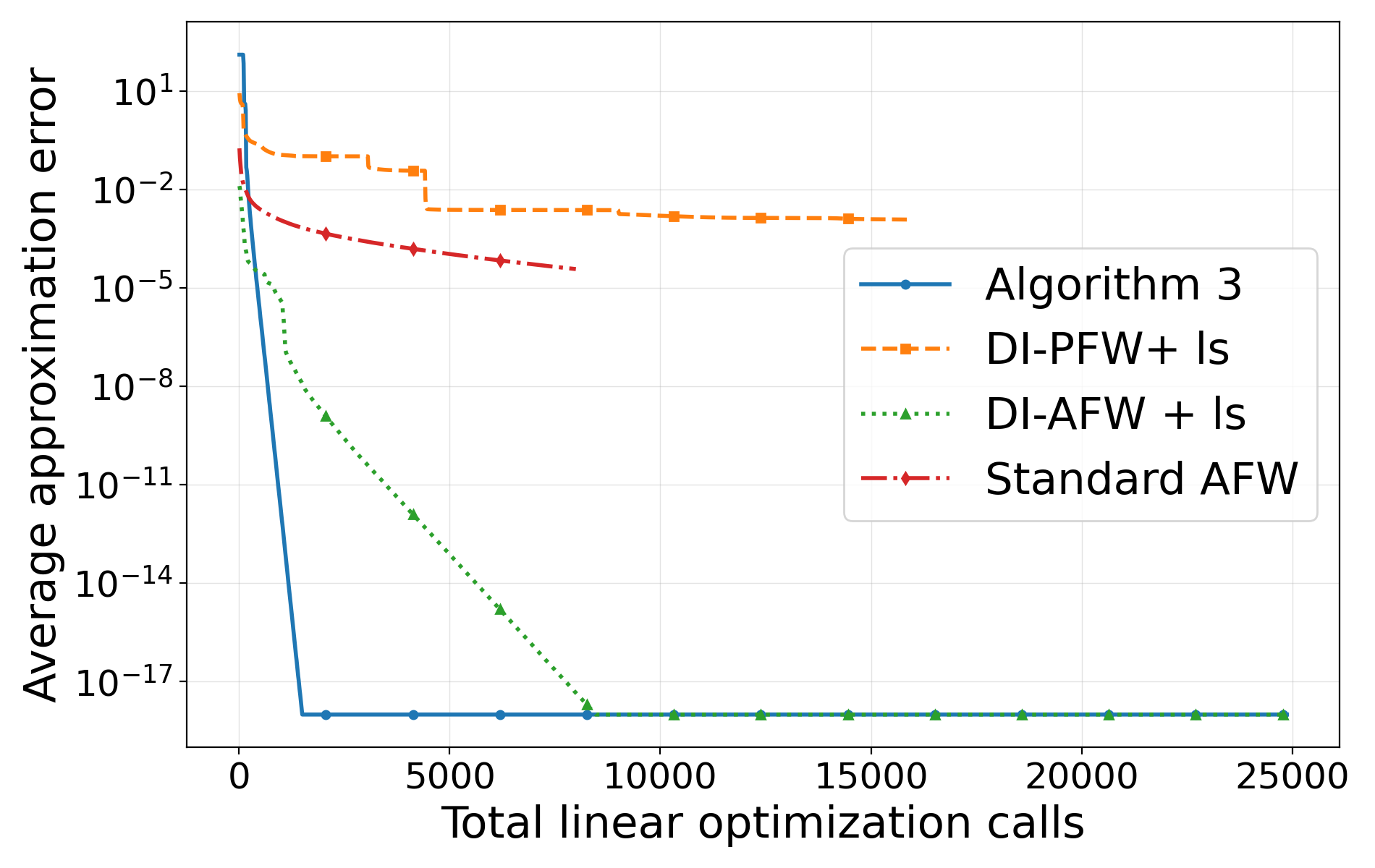}
         \caption*{$k=40$}
         \label{fig:y equals x}
     \end{subfigure}
     \hfill
     \begin{subfigure}[b]{0.32\textwidth}
         \centering
         \includegraphics[width=\textwidth]{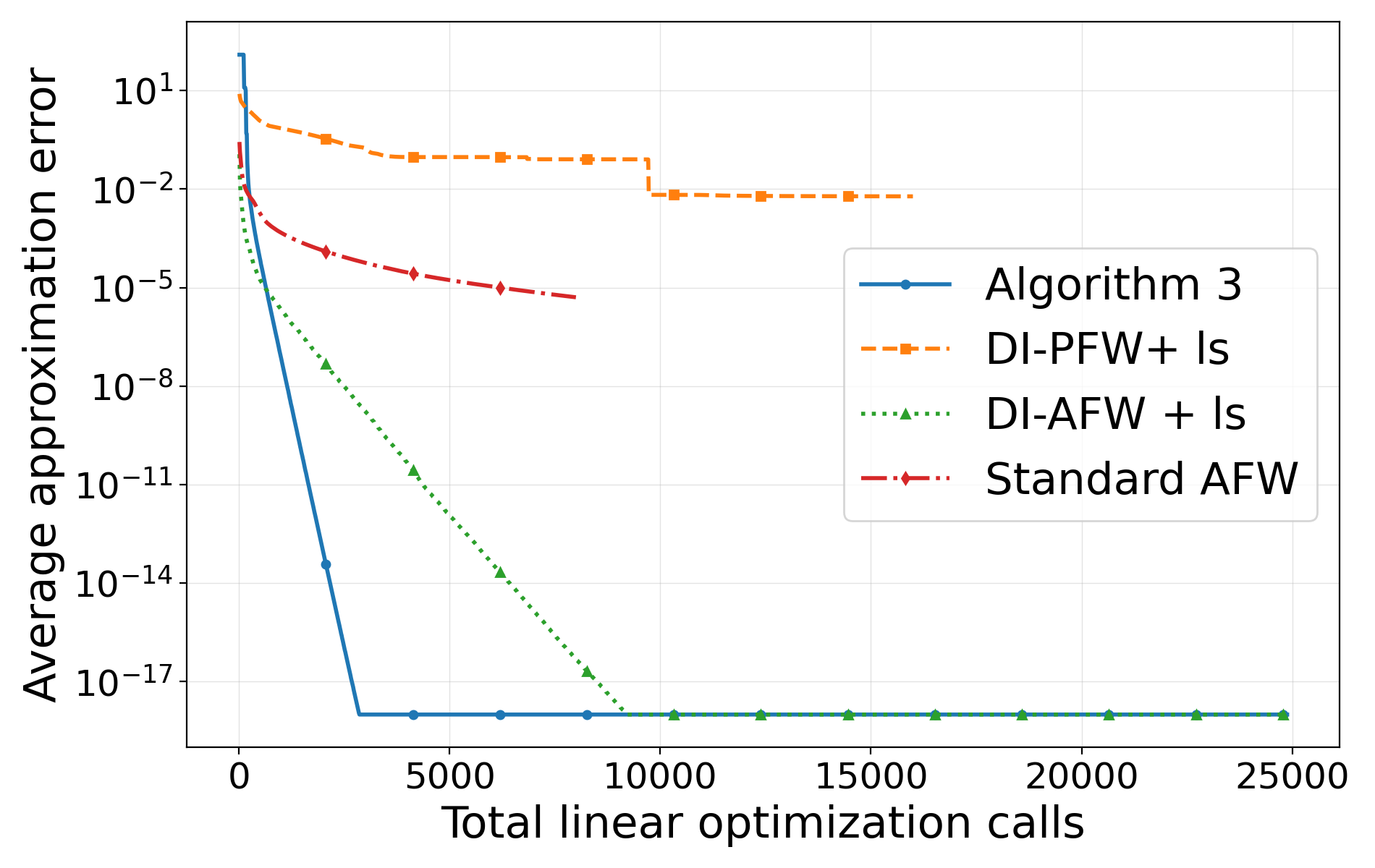}
         \caption*{$k=100$}
         \label{fig:three sin x}
     \end{subfigure}
     \hfill
     \begin{subfigure}[b]{0.32\textwidth}
         \centering
         \includegraphics[width=\textwidth]{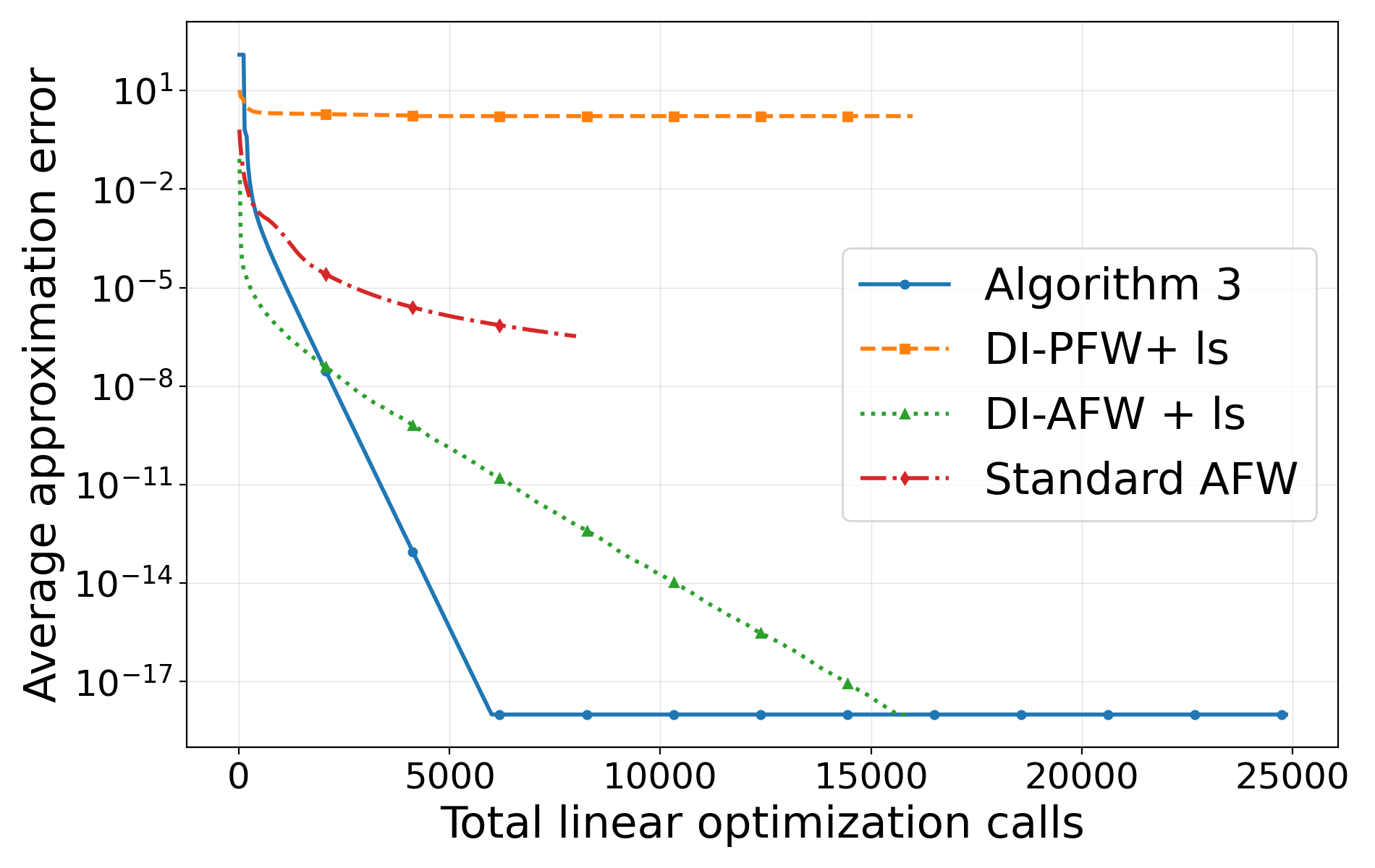}
         \caption*{$k=250$}
         \label{fig:five over x}
     \end{subfigure}

     \begin{subfigure}[b]{0.32\textwidth}
         \centering
         \includegraphics[width=\textwidth]{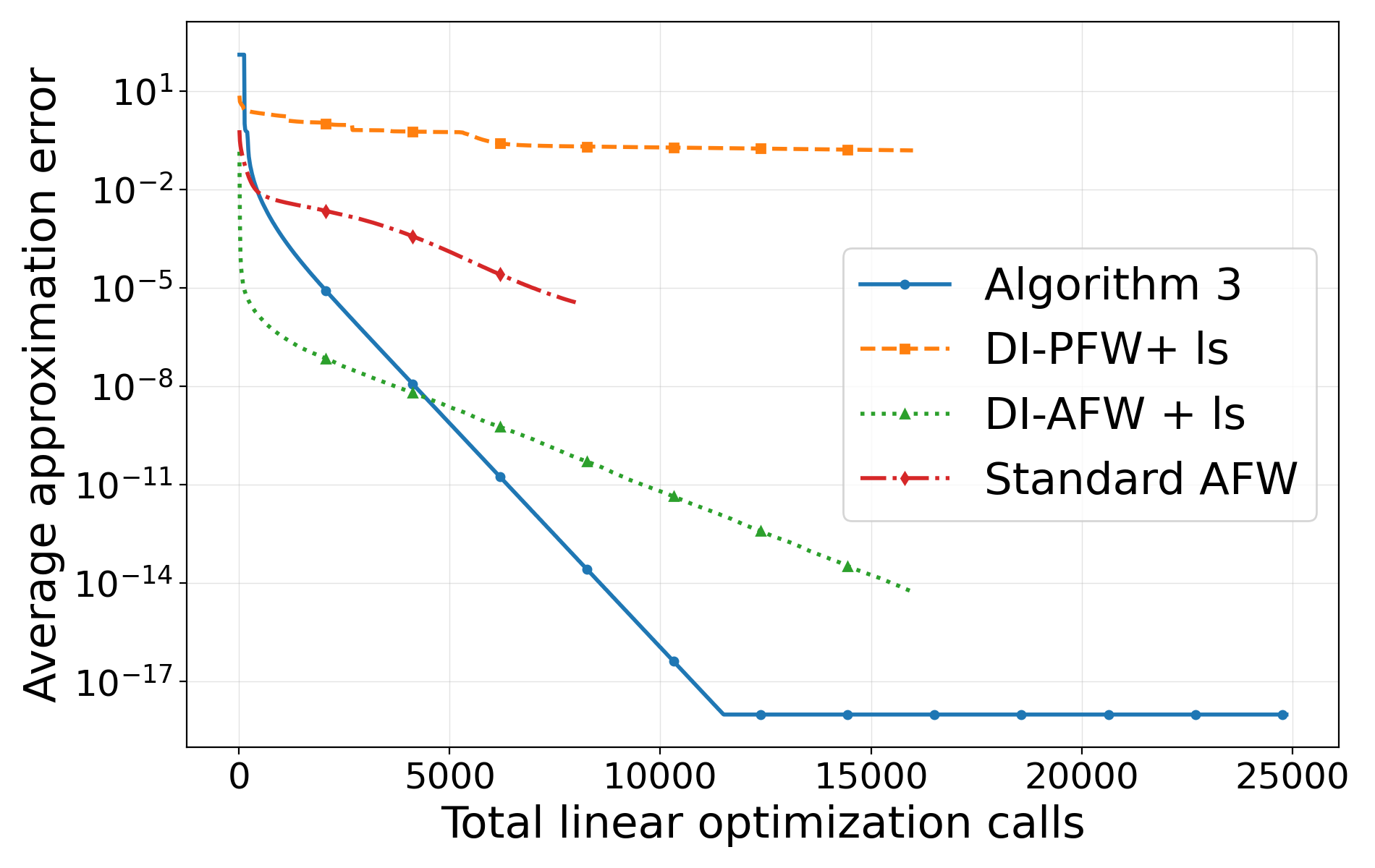}
         \caption*{$k=500$}
         \label{fig:y equals x}
     \end{subfigure}
     \hfill
     \begin{subfigure}[b]{0.32\textwidth}
         \centering
         \includegraphics[width=\textwidth]{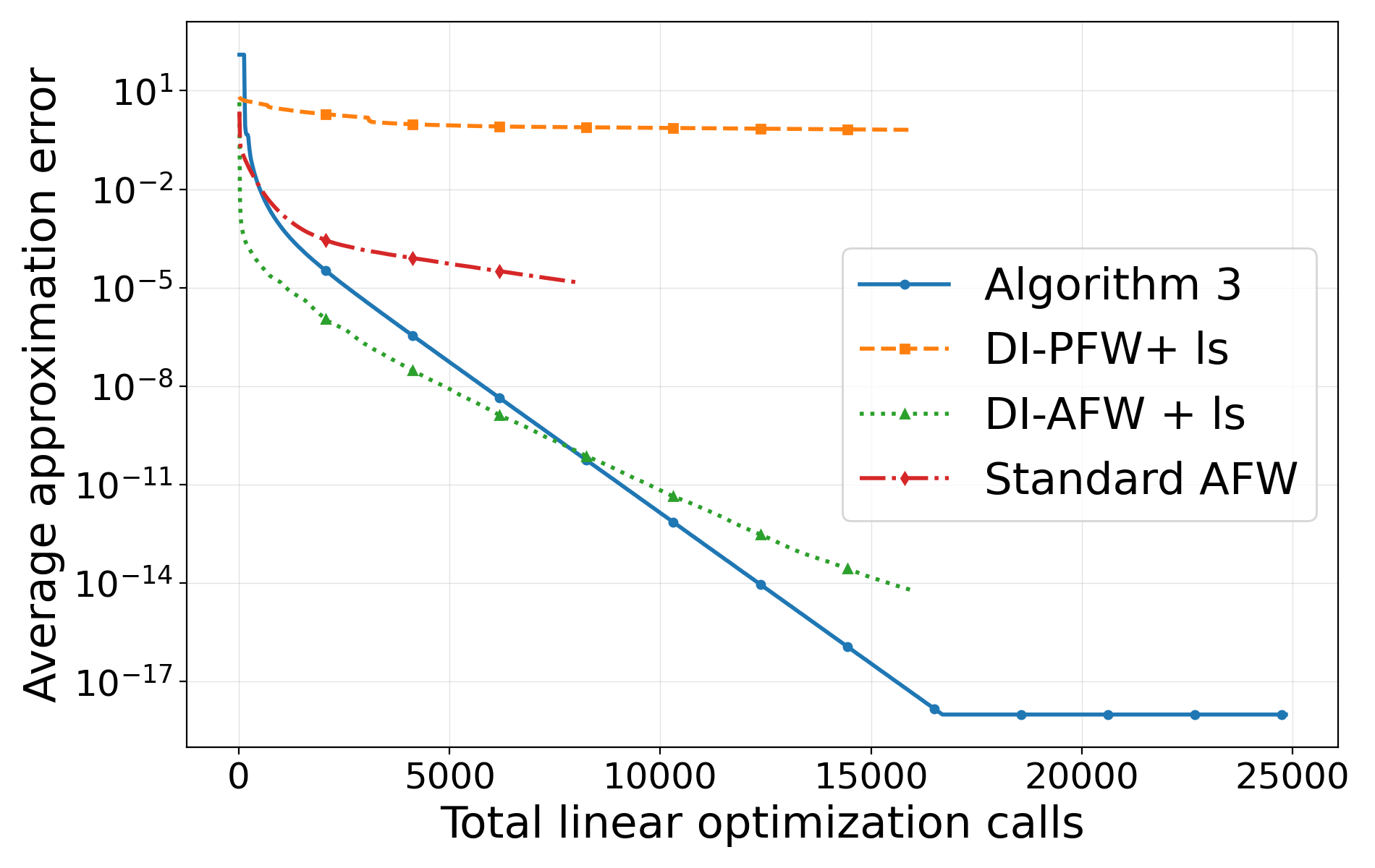}
         \caption*{$k=750$}
         \label{fig:three sin x}
     \end{subfigure}
     \hfill
     \begin{subfigure}[b]{0.32\textwidth}
         \centering
         \includegraphics[width=\textwidth]{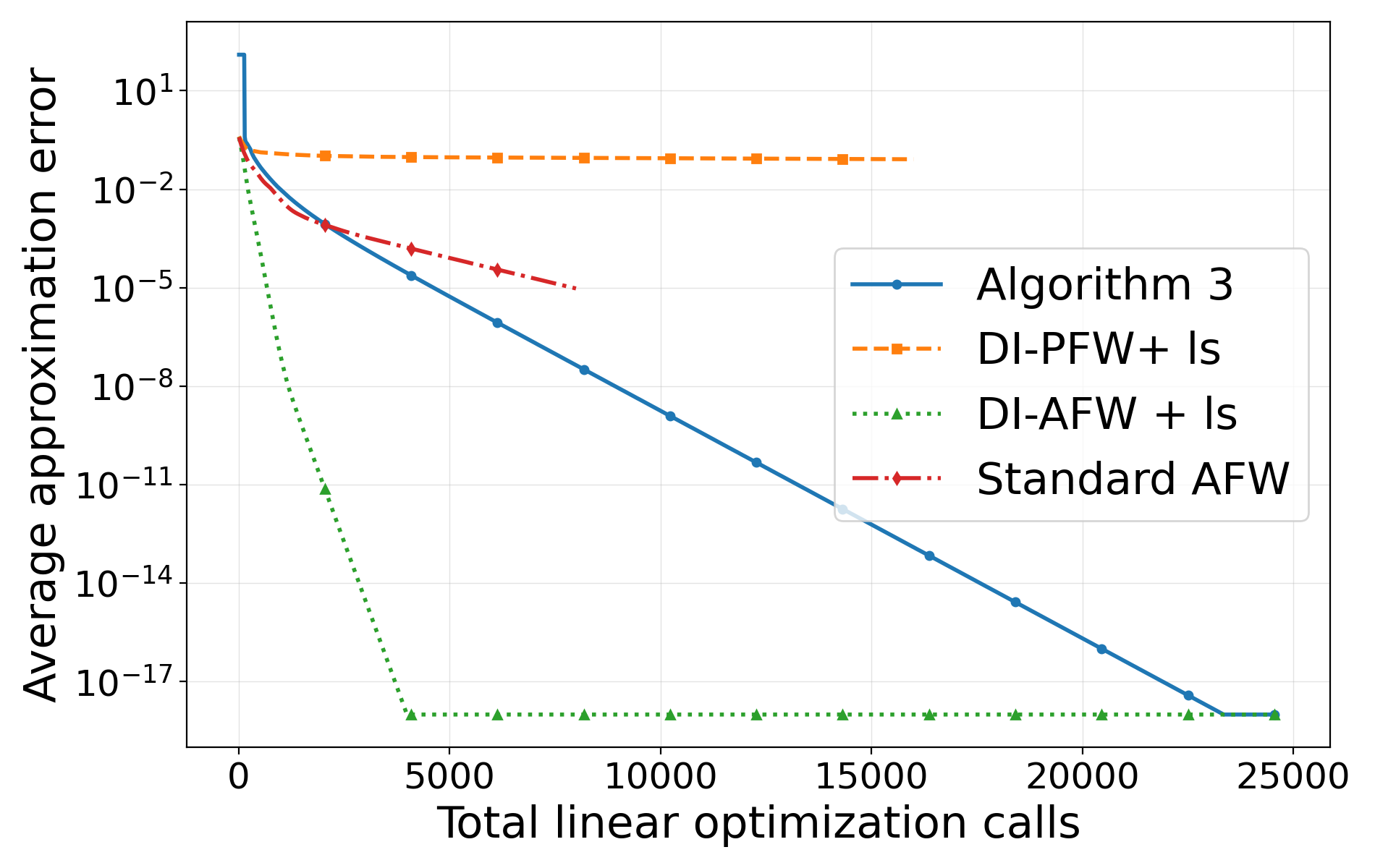}
         \caption*{$k=990$}
         \label{fig:five over x}
     \end{subfigure}
        \caption{Approximation errors vs. number of calls to the linear optimization oracle for Euclidean projection onto a fractional capped simplex.}
        \label{fig:capped_simplex_oracle}
\end{figure}

\section{Acknowledgments}
This work was funded by the European Union (ERC,  ProFreeOpt, 101170791). Views and opinions expressed are however those of the author(s) only and do not necessarily reflect those of the European Union or the European Research Council Executive Agency. Neither the European Union nor the granting authority can be held responsible for them.

\appendix

\section{Proof of Lemma \ref{lem:productMassTransfer}}
\begin{proof} [Proof of Lemma \ref{lem:productMassTransfer}]
Assume $\x\neq\x^*$.
For each $j\in[m]$, let $D_j$ and $\mI_j$ denote the diameter and
constrained coordinate set of $\mP_j$, respectively, and set
\[
M_j(\u):=
\sum_{\ell\in\mI_j\setminus\mI_j(\x^{(j)})}\u(\ell)
\qquad \u\in\mV_j.
\]

Apply the construction of Lemma~\ref{lem:L1dist} to each factor,
assigning the retained mass $\lambda-\Delta$ to the pair $(\v,\v)$
and the transferred mass $\Delta$ to the pair $(\v,\u)$.
After omitting zero masses, denote the resulting pairs and weights by
$(\v_{j,r},\u_{j,r})$ and $\theta_{j,r}>0$, $r\in[k_j]$. Thus,
\[
\sum_{r=1}^{k_j}\theta_{j,r}=1,
\qquad
\x^{(j)}=\sum_{r=1}^{k_j}\theta_{j,r}\v_{j,r},
\qquad
\x^{*(j)}=\sum_{r=1}^{k_j}\theta_{j,r}\u_{j,r},
\]
where
$\v_{j,r}\in\mV_j\cap\mF(\x^{(j)})$ and
$\u_{j,r}\in\mV_j\cap\mF(\x^{*(j)})$.
Define
\[
c_{j,r}:=
\begin{cases}
0,&\v_{j,r}=\u_{j,r},\\
\sqrt{\max\{1,M_j(\u_{j,r})\}},&\v_{j,r}\neq\u_{j,r}.
\end{cases}
\]
The weighted bound in Lemma~\ref{lem:L1dist} gives
\[
a_j:=\sum_{r=1}^{k_j}\theta_{j,r}c_{j,r}
\leq\sqrt{d_j^*+1}\,
\Vert{\x^{(j)}-\x^{*(j)}}\Vert.
\]
Consequently, for $\a:=(a_1,\ldots,a_m)$,
\begin{align}\label{eq:productMassTransfer:factor-bound}
\Vert{\a}\Vert^2
\leq\sum_{j=1}^m(d_j^*+1)
\Vert{\x^{(j)}-\x^{*(j)}}\Vert^2
\leq\left(1+\max_{j\in[m]}d_j^*\right)
\Vert{\x-\x^*}\Vert^2.
\end{align}
By definition, $c_{j,r}=0$ when $\v_{j,r}=\u_{j,r}$, and
$c_{j,r}\geq1$ otherwise. Moreover, since $\v_{j,r}$ is zero on
the coordinates summed in $M_j$, the normalization $\xi=1$ gives
\[
M_j(\u_{j,r})
\leq c_{j,r}^2
\leq \Vert{\u_{j,r}-\v_{j,r}}\Vert^2
\leq D_j^2.
\]

Reorder the pairs in each factor so that
$c_{j,1}\geq\cdots\geq c_{j,k_j}$, and let
$s_{j,0}:=0$, $s_{j,r}:=\sum_{h=1}^r\theta_{j,h}$.
Let $0=t_0<t_1<\cdots<t_k=1$ be all the distinct values
$s_{j,r}$, for $j\in[m]$ and $r\in\{0,\ldots,k_j\}$,
listed in increasing order.  For every $i\in[k]$, define
\[
\lambda_i:=t_i-t_{i-1},
\qquad
\ell_j(i):=\min\{r\in[k_j]:t_i\leq s_{j,r}\},
\]
\[
\v_i:=\bigl(\v_{j,\ell_j(i)}\bigr)_{j=1}^m,
\qquad
\u_i:=\bigl(\u_{j,\ell_j(i)}\bigr)_{j=1}^m,
\qquad
\z_i:=\bigl(c_{j,\ell_j(i)}\bigr)_{j=1}^m.
\]
Since $\sum_{i:\ell_j(i)=r}\lambda_i=\theta_{j,r}$, we have
\[
\x=\sum_{i=1}^k\lambda_i\v_i,
\qquad
\x^*=\sum_{i=1}^k\lambda_i\u_i,
\qquad
\a=\sum_{i=1}^k\lambda_i\z_i,
\]
and $\z_1\geq\z_2\geq\cdots\geq\z_k\geq\mathbf{0}$ coordinatewise.

Let $q:=\max\{i\in[k]:\z_i\neq\mathbf{0}\}$ and set
$\Delta_i:=\lambda_i$ for $i\leq q$ and $\Delta_i:=0$ otherwise.
For $i>q$, we have $\v_i=\u_i$, and hence
\[
\x^*=\sum_{i=1}^k(\lambda_i-\Delta_i)\v_i
+\sum_{i=1}^k\Delta_i\u_i,
\qquad
\u_i\in\mV\cap\mF(\x^*).
\]
Coordinatewise monotonicity gives
$\a\geq\sum_{i=1}^q\lambda_i\z_q=t_q\z_q$.
Since $\Vert{\z_q}\Vert\geq1$, it follows that
\[
\sum_{i=1}^k\Delta_i=t_q\leq\Vert{\a}\Vert
\leq\sqrt{1+\max_{j\in[m]}d_j^*}\,
\Vert{\x-\x^*}\Vert,
\]
proving \eqref{eq:lem:productMassTransfer}.

For the weighted bound, write
$r_i:=\Vert{\z_i}\Vert$ for $i\in[q]$ and $r_{q+1}:=0$.
Then
\[
1\leq r_q\leq\cdots\leq r_1
\leq\sqrt{\sum_{j=1}^mD_j^2}=D,
\]
and, for every $i\leq q$,
\[
\max\Big\{1,
\sum_{\ell\in\mI\setminus\mI(\x)}\u_i(\ell)\Big\}
=\max\Big\{1,\sum_{j=1}^m
M_j(\u_{j,\ell_j(i)})\Big\}
\leq\sum_{j=1}^mc_{j,\ell_j(i)}^2=r_i^2.
\]
Thus,
\begin{align}\label{eq:productMassTransfer:weighted-sum}
\sum_{i=1}^k\Delta_i
\sqrt{\max\Big\{1,
\sum_{\ell\in\mI\setminus\mI(\x)}\u_i(\ell)\Big\}}
\leq\sum_{i=1}^q\lambda_i r_i.
\end{align}
For $i<j$, we have
$\langle{\z_i,\z_j}\rangle\geq\Vert{\z_j}\Vert^2=r_j^2$.
Therefore,
\begin{align*}
\Vert{\a}\Vert^2
&=\sum_{i=1}^q\lambda_i^2r_i^2
+2\sum_{1\leq i<j\leq q}
\lambda_i\lambda_j\langle{\z_i,\z_j}\rangle\\
&\geq\sum_{i=1}^q
\left(\lambda_i^2+2\lambda_i t_{i-1}\right)r_i^2
=\sum_{i=1}^q(t_i^2-t_{i-1}^2)r_i^2=\sum_{i=1}^qt_i^2(r_i^2-r_{i+1}^2).
\end{align*}
Summation by parts and Cauchy--Schwarz now give
\begin{align*}
\left(\sum_{i=1}^q\lambda_i r_i\right)^2
&=\left(\sum_{i=1}^qt_i(r_i-r_{i+1})\right)^2\leq
\left(\sum_{i=1}^qt_i^2(r_i^2-r_{i+1}^2)\right)
\left(\sum_{i=1}^q
\frac{r_i-r_{i+1}}{r_i+r_{i+1}}\right)\\
&\leq\Vert{\a}\Vert^2
\left(1+\frac12\sum_{i=1}^{q-1}
\log\frac{r_i}{r_{i+1}}\right)\\
&=\left(1+\frac12\log\frac{r_1}{r_q}\right)
\Vert{\a}\Vert^2
\leq\left(1+\frac12\log D\right)\Vert{\a}\Vert^2,
\end{align*}
where we used
$\frac{u-v}{u+v}\leq\frac12\log\frac uv$ for $u\geq v>0$.
Combining this with
\eqref{eq:productMassTransfer:factor-bound} and
\eqref{eq:productMassTransfer:weighted-sum} proves
\eqref{eq:lem:productMassTransferWeighted}.
\end{proof}

\section{Proof of Lemma \ref{lem:dicg:feas}}
\begin{proof}[Proof of Lemma \ref{lem:dicg:feas}]
For every iteration $t$, let $\bar{\eta}_t$ denote the value of the variable
$\eta$ after the backtracking step on iteration $t$. Thus, the actual step-size
used by the algorithm is $\eta_t=\bar{\eta}_t/2$. By construction,
$\bar{\eta}_t$ is always a power of two, and the sequence
$\{\bar{\eta}_t\}_{t\geq 1}$ is monotone non-increasing. Hence, there
exist integers $\delta_t\geq 1$ such that $\eta_t=2^{-\delta_t}$,
and $\delta_{t+1}\geq \delta_t$ for all $t$.

We first note the following simple fact. Suppose that $\x_t$ is feasible, and
that for every $i\in\mI$ for which $\x_t(i)>0$, it holds that
$\x_t(i)\geq \eta_t$. Then, $\x_{t+1}$ is also feasible. Indeed, since
$\v_{t,-}\in\mV\cap\mF(\x_t)$, it follows that for every $i\in\mI$ such that
$\x_t(i)=0$, we have $\v_{t,-}(i)=0$. Thus, since 
$\v(i)\in\{0,1\}$ for every vertex $\v\in\mV$ and every $i\in\mI$, subtracting
$\eta_t\v_{t,-}$ from $\x_t$ cannot make any coordinate in $\mI$ negative.
Adding $\eta_t\v_{t,+}$ also cannot make any coordinate in $\mI$ negative.
Finally, since $\x_t,\v_{t,+},\v_{t,-}\in\mP$, we have
\[
\A\x_{t+1}
=
\A\x_t+\eta_t\A(\v_{t,+}-\v_{t,-})
=
\b.
\]
Thus, $\x_{t+1}\in\mP$.

We are going to prove by induction that on each iteration $t$, after the
step-size $\eta_t$ has been fixed, there exists a nonnegative integer-valued
vector $s_t$, indexed by $\mI$, such that
\[
\x_t(i)=2^{-\delta_t}s_t(i) \qquad \forall i\in\mI .
\]
The base case $t=1$ holds since $\x_1$ is a vertex of $\mP$, and hence, by the
2-level assumption, $\x_1(i)\in\{0,1\}$ for all $i\in\mI$. Since
$\delta_1\geq 1$, there indeed exists a non-negative integer-valued vector
$s_1$ such that $\x_1(i)=2^{-\delta_1}s_1(i)$ for all $i\in\mI$.

Suppose now that the induction hypothesis holds for some iteration $t$. Then,
for every $i\in\mI$ such that $\x_t(i)>0$, we have $s_t(i)\geq 1$, and hence
$\x_t(i)\geq 2^{-\delta_t}=\eta_t$. Thus, by the observation above,
$\x_{t+1}$ is feasible.

It remains only to verify the induction hypothesis for the next iterate. For
every $i\in\mI$, since $\v_{t,+}(i),\v_{t,-}(i)\in\{0,1\}$, and since
$\v_{t,-}\in\mF(\x_t)$, we have
\[
\x_{t+1}(i)
=
2^{-\delta_t}
\begin{cases}
s_t(i) & \v_{t,+}(i)=\v_{t,-}(i),\\
s_t(i)-1 & \v_{t,+}(i)=0,\ \v_{t,-}(i)=1,\\
s_t(i)+1 & \v_{t,+}(i)=1,\ \v_{t,-}(i)=0.
\end{cases}
\]
In the second case, $s_t(i)\geq 1$, since $\v_{t,-}(i)=1$ implies
$\x_t(i)>0$. Therefore, $\x_{t+1}(i)=2^{-\delta_t}\widetilde{s}_{t+1}(i)$ for
some non-negative integer-valued vector $\widetilde{s}_{t+1}$, indexed by
$\mI$. Since $\delta_{t+1}\geq\delta_t$, the number
$2^{\delta_{t+1}-\delta_t}$ is a positive integer. Setting
\[
s_{t+1}(i)=2^{\delta_{t+1}-\delta_t}\widetilde{s}_{t+1}(i)
\qquad \forall i\in\mI,
\]
we get
\[
\x_{t+1}(i)=2^{-\delta_{t+1}}s_{t+1}(i)
\qquad \forall i\in\mI.
\]
Thus, the induction also holds for iteration $t+1$, and the proof follows.
\end{proof}

\bibliography{bibs}
\bibliographystyle{plain}

\end{document}